\documentclass[11pt]{article}
\usepackage[margin=1in]{geometry}
\usepackage{amsmath,amsthm,amsfonts,amssymb}
\usepackage{graphicx,enumerate}
\graphicspath{ {./figures/} }
\usepackage[title]{appendix}
\usepackage[
            CJKbookmarks=true,
            bookmarksnumbered=true,
            bookmarksopen=true,
            colorlinks=true,
            citecolor=red,
            linkcolor=blue,
            anchorcolor=red,
            urlcolor=blue
            ]{hyperref}

\usepackage{mathtools}
\usepackage{dsfont}
\usepackage{fixltx2e}

\newcommand{\risk}{\ell}

\newcommand{\e}{e}

\let\emptyset\varnothing

\newcommand{\Mgood}{\calM_{\sf good}}
\newcommand{\Mbad}{\calM_{\sf bad}}

\newcommand{\ceil}[1]{{\left\lceil {#1} \right \rceil}}

\newcommand\numberthis{\addtocounter{equation}{1}\tag{\theequation}}

\usepackage{tikz}
\usepackage{tikz-qtree}
\usepackage{mathtools}

\usetikzlibrary{calc}
\usetikzlibrary{decorations.pathreplacing,decorations.markings}

\tikzstyle{dot}=[circle,fill,black,inner sep=1pt]

\usepackage{array}
\usepackage{multirow}

\usepackage{float}
\usepackage{tikz}
\usetikzlibrary{matrix,arrows,calc,shapes,backgrounds}
\usetikzlibrary{shapes.callouts,decorations.text}

\usetikzlibrary{shapes.misc}

\tikzset{cross/.style={cross out, draw=black, minimum size=2*(#1-\pgflinewidth), inner sep=0pt, outer sep=0pt},
cross/.default={1pt}}

\tikzstyle{int}=[draw, fill=blue!20, minimum size=2em]
\tikzstyle{dot}=[circle, draw, fill=blue!20, minimum size=2em]
\tikzstyle{dotred}=[circle, draw, fill=red!20, minimum size=2em]
\tikzstyle{init} = [pin edge={to-,thin,black}]
\tikzstyle{initred} = [pin edge={to-,thin,red}]
\tikzstyle{plan}=[draw, fill=blue!20, minimum size=2em, text width=5em, rounded corners,align=center]
\tikzstyle{planwide}=[draw, fill=blue!20, minimum size=2em, text width=8em, rounded corners,align=center]
\tikzstyle{nodedot}=[circle, draw, fill=white, minimum size=0.2cm,inner sep=0pt]
\tikzstyle{smallnode}=[circle, draw, fill=black, minimum size=0.1cm,inner sep=0pt]
\tikzstyle{vertexdot}=[circle, draw, fill=white, minimum size=3,inner sep=0pt]
\tikzstyle{Redge}=[red, line width=2.71828pt]
\tikzstyle{Bedge}=[blue, thick]
\tikzstyle{shadedgiantnode}=[circle, draw, fill=black!10, minimum size=0.7cm, inner sep=0pt]
\tikzstyle{unshadedgiantnode}=[circle, draw, fill=white, minimum size=0.7cm, inner sep=0pt]
\makeatletter
\tikzset{my loop/.style =  {to path={
  \pgfextra{}
  [looseness=5,min distance=10mm]
  \tikz@to@curve@path},font=\sffamily\small
  }}  
\makeatletter 
\newcolumntype{C}[1]{>{\centering\arraybackslash}p{#1}}
\usepackage{caption,subcaption,soul}

\newtheorem{theorem}{Theorem}
\newtheorem{lemma}{Lemma}

\theoremstyle{definition}
\newtheorem{definition}{Definition}

\newtheorem{remark}{Remark}

\usepackage{color}
\usepackage{algorithm}
\usepackage{algpseudocode}

\usepackage{xspace,prettyref}

\newcommand{\stepa}[1]{\overset{\rm (a)}{#1}}
\newcommand{\stepb}[1]{\overset{\rm (b)}{#1}}
\newcommand{\stepc}[1]{\overset{\rm (c)}{#1}}

\newcommand{\diverge}{\to\infty}

\newcommand{\iiddistr}{{\stackrel{\text{\iid}}{\sim}}}

\newcommand{\reals}{{\mathbb{R}}}

\newcommand{\diff}{\mathrm{d}}

\newcommand{\Expect}{\mathbb{E}}
\newcommand{\expect}[1]{\mathbb{E}\!\left[ #1 \right]}

\newcommand{\prob}[1]{ \mathbb{P}\left\{ #1 \right\} }

\newcommand{\var}{\mathrm{Var}}
\newcommand{\Var}{\mathrm{Var}}

\newcommand{\Bern}{{\rm Bern}}
\newcommand{\Binom}{{\rm Binom}}

\newcommand{\ie}{i.e.\xspace}
\newcommand{\iid}{i.i.d.\xspace}
\newrefformat{eq}{(\ref{#1})}
\newrefformat{chap}{Chapter~\ref{#1}}
\newrefformat{sec}{Section~\ref{#1}}
\newrefformat{alg}{Algorithm~\ref{#1}}
\newrefformat{fig}{Fig.~\ref{#1}}
\newrefformat{tab}{Table~\ref{#1}}
\newrefformat{rmk}{Remark~\ref{#1}}
\newrefformat{clm}{Claim~\ref{#1}}
\newrefformat{def}{Definition~\ref{#1}}
\newrefformat{cor}{Corollary~\ref{#1}}
\newrefformat{lmm}{Lemma~\ref{#1}}
\newrefformat{prop}{Proposition~\ref{#1}}
\newrefformat{app}{Appendix~\ref{#1}}
\newrefformat{hyp}{Hypothesis~\ref{#1}}
\newrefformat{thm}{Theorem~\ref{#1}}
\newrefformat{ass}{Assumption~\ref{#1}}
\newrefformat{conj}{Conjecture~\ref{#1}}

\newcommand{\pth}[1]{\left( #1 \right)}
\newcommand{\qth}[1]{\left[ #1 \right]}
\newcommand{\sth}[1]{\left\{ #1 \right\}}

\newcommand{\indc}[1]{{\mathbf{1}_{\left\{{#1}\right\}}}}

\newcommand{\tU}{{\widetilde{U}}}

\newcommand{\sfb}{{\mathsf{b}}}

\newcommand{\sfr}{{\mathsf{r}}}

\newcommand{\calA}{{\mathcal{A}}}
\newcommand{\calB}{{\mathcal{B}}}
\newcommand{\calC}{{\mathcal{C}}}

\newcommand{\calE}{{\mathcal{E}}}

\newcommand{\calK}{{\mathcal{K}}}
\newcommand{\calL}{{\mathcal{L}}}
\newcommand{\calM}{{\mathcal{M}}}
\newcommand{\calN}{{\mathcal{N}}}

\newcommand{\calP}{{\mathcal{P}}}
\newcommand{\calQ}{{\mathcal{Q}}}

\newcommand{\calU}{{\mathcal{U}}}

\newcommand{\calX}{{\mathcal{X}}}

\DeclareMathAlphabet{\varmathbb}{U}{bbold}{m}{n}

\renewcommand{\hat}{\widehat}
\renewcommand{\tilde}{\widetilde}

\usepackage{bbm}

\newcommand{\symdiff}{\triangle}
\newcommand{\dx}{\mathrm{d}x}
\newcommand{\dy}{\mathrm{d}y}

\newcommand{\sz}{s} 
\renewcommand{\path}{P}
\renewcommand{\sp}{\eta}
\newcommand{\sr}{\beta}

\newcommand{\wt}{\mathsf{wt}}
\newcommand{\wtr}{\wt_{\sfr}}
\newcommand{\wtb}{\wt_{\sfb}}
\newcommand{\dev}{\mathsf{dev}}
\newcommand{\devr}{\dev_{\sfr}}
\newcommand{\devb}{\dev_{\sfb}}

\newcommand{\Mmin}{\hat M_{\mathsf{ML}}}
\newcommand{\Mplanted}{M^*}

\newcommand{\ER}{Erd\H{o}s-R\'{e}nyi\xspace}

\pgfdeclarelayer{background}
\pgfdeclarelayer{foreground}
\pgfsetlayers{background,main,foreground}

\title{The planted matching problem: \\
Sharp threshold and infinite-order phase transition
}
\author{Jian Ding, Yihong Wu, Jiaming Xu, and Dana Yang\thanks{
J.\ Ding is with Department of Statistics, The Wharton School, University of Pennsylvania, Philadelphia, USA, \texttt{dingjian@wharton.upenn.edu}.
Y.\ Wu is with Department of Statistics and Data Science, Yale University, New Haven, USA, \texttt{yihong.wu@yale.edu}.
J.\ Xu and D.\ Yang are with The Fuqua School of Business, Duke University, Durham NC, USA, \texttt{\{jx77,xiaoqian.yang\}@duke.edu}.
}}

\begin{document}

\maketitle

\begin{abstract}
We study the problem of reconstructing a perfect  matching $M^*$ hidden in a randomly weighted 
$n\times n$ bipartite graph. 
The edge set includes every node pair in $M^*$ 
and each of the $n(n-1)$ node pairs not in $M^*$ independently with probability
$d/n$. The weight of each edge $e$ is independently drawn from the distribution
$\calP$ if $e \in M^*$ and from $\calQ$ if $e \notin M^*$.
We show that if $\sqrt{d} B(\calP,\calQ) \le 1$, where $B(\calP,\calQ)$ stands for the Bhattacharyya coefficient, 
 the reconstruction error (average fraction of misclassified edges) 
of the maximum likelihood estimator of $M^*$ converges to $0$ as $n\to \infty$. 
Conversely, if $\sqrt{d} B(\calP,\calQ) \ge 1+\epsilon$ for an arbitrarily small constant $\epsilon>0$,
the reconstruction error for any estimator is shown to be bounded away from $0$ under both the sparse and dense model, resolving the conjecture in \cite{Moharrami2020a,Semerjian2020}.
Furthermore, in the special case of complete exponentially weighted graph with $d=n$, $\calP=\exp(\lambda)$,
and $\calQ=\exp(1/n)$, for which the sharp threshold simplifies to $\lambda=4$, we prove that when $\lambda \le 4-\epsilon$, 
the optimal reconstruction error is $\exp\left( - \Theta(1/\sqrt{\epsilon}) \right)$, 
confirming the conjectured infinite-order phase transition in \cite{Semerjian2020}.
\end{abstract}

\tableofcontents

\section{Introduction}

Let $\calM$ denote the set of perfect matchings on the complete bipartite graph $K_{n,n}$, with left-hand vertices labeled as $[n]=\{1,\ldots,n\}$ and right-hand vertices labeled as $[n]'=\{1',\ldots,n'\}$.
We denote the weights on the edge $e=(i,j')$ by $W_{i,j'}=W_e$.
Each perfect matching $M$ is understood as a set of edges in $K_{n,n}$.

\begin{definition}[Planted matching recovery]
\label{def:model}
Consider a weighted bipartite graph $G$ randomly generated as follows.  
First sample $M^*$ uniformly at random from the set $\calM$ of all perfect matchings on $K_{n,n}$ and connect all pairs in $M^*$.
Then for every node pair  $(i,j')$ not in $M^*$, connect them independently with probability $\frac{d}{n}$.  
Finally, for edges in $M^*$, the edge weights are drawn independently from $\calP$. The remaining edge weights are drawn independently from $\calQ$. The goal is to reconstruct the hidden matching $\Mplanted$ based on  $G$. 
\end{definition}

This problem is first proposed by~\cite{Chertkov2010} and motivated from tracking moving objects in a video, such as flocks of birds, motile cells, or particles in a fluid.
A slight variation of \prettyref{def:model} is studied in \cite{Semerjian2020} for unipartite graphs, where $M^*$ is chosen uniformly at random from the set of all perfect matchings on the complete unipartite graph $K_n$ (with even $n$) and the edge set of $G$ includes all $n/2$ node pairs in $M^*$ and each of the $\binom{n}{2}-n/2$ node pairs not in $M^*$ independently with probability $\frac{d}{n}$;
the edge weights are still independently distributed according to $\calP$ for edges in $M^*$ and $\calQ$ otherwise. In this paper, we present our results and analysis for bipartite graphs; nevertheless, the proof techniques can be straightforwardly extended to the unipartite version and all conclusions hold verbatim.

Of particular interest are the following two regimes, which are the focus of \cite{Semerjian2020} and the present paper:
\begin{itemize}
	\item \textbf{Sparse model}: The average degree $d$ is a constant. In this case, the null ({\it i.e.} unplanted) distribution $\calQ$ and planted distribution $\calP$ can be arbitrary distributions independent of $n$.
	
	\item \textbf{Dense model}: The average degree $d \equiv d(n) \to \infty$ as $n$ grows. In particular, when $d=n$, we observe a complete bipartite graph with weights correlated with the hidden matching.
In this case, we focus on the following special case of weight distributions, where the planted distribution $\calP$ has a fixed density $p$ (with respect to the Lebesgue measure) and the null distribution $\calQ$ has a density $q$ of the following form:
	\begin{equation}
	q(x) = \frac{1}{d} \rho\pth{\frac{x}{d}}
	\label{eq:q-dense}
	\end{equation}
	where $\rho$ is some fixed density on $\reals$ with $\rho(0)>0$.\footnote{If $\rho(0)=0$, then for each vertex, among its incident edges, the planted edge weight has the smallest magnitude with high probability, in which case almost perfect recovery is trivially achievable.}
	This scaling is natural and meaningful, because the magnitude of each unplanted  edge weight is $O(d)$ on average, and each vertex is incident to an average of $d$ unplanted edges. Under this scaling, the minimum magnitude the unplanted edge weights incident to a given vertex is on the same order of $O(1)$ as the planted edge weight. Of special interest is the complete graph with exponential weights \cite{Moharrami2020a,Semerjian2020}, which we refer to as the \emph{exponential model}, where $d=n, \calP=\exp(\lambda)$ and $\calQ=\exp(\frac{1}{n})$.
	Note that its unplanted version is the celebrated random assignment model studied in \cite{Walkup1979,Parisi1987,Karp1987,Aldous2001,Linusson2004,Nair2005,Wastlund2009}.
  Another special case was studied by in~\cite{Chertkov2010} 	
	where $\calP=|\calN(0,\kappa)|$ is a folded Gaussian and $\calQ$
  is uniform over $[0,n]$. 
\end{itemize}

Let $\hat M\equiv \hat M(G)$ denote an estimator of $M^*$. 
The reconstruction error, namely, the fraction of misclassified edges is 
\begin{equation}
\risk(\hat M,M^*) = \frac{1}{n}| \Mplanted \symdiff \hat{M} |,
\label{eq:risk}
\end{equation}
where $\symdiff$ denotes the symmetric set difference. We say that $\hat M$ achieves \emph{almost perfect recovery} if $\Expect[\risk(\Mplanted,\hat M)] = o(1)$.
It can be shown (see e.g.~\cite[Appendix A]{HajekWuXu_one_info_lim15}) that achieving a vanishing reconstruction error in expectation is equivalent to that with high probability.
Note that for any estimator $\hat M=\hat{M}(G)$, we have
\begin{equation}
\expect{\risk \left(M^*,\hat{M}\right)} = \frac{1}{n} \sum_{e\in E(K_{n,n})} \prob{e \in M^* \symdiff \hat{M}}.
\label{eq:risk}
\end{equation}
Thus the average reconstruction error is minimized by the
\emph{marginal maximum a posteriori (MAP)} estimator, where $e \in \hat{M}$ if and only if 
$ \prob{e \in M^* | G} \ge \prob{e \notin M^* | G}  $, since it minimizes each summand in \prettyref{eq:risk}.
Note that in general the marginal MAP need not be a matching or even be of size $n$.
Nevertheless, it is easy to see that one can project any estimator  to the set of perfect matchings (in distance metric $\risk$)  
while increasing the average reconstruction error by a factor of two.

\subsection{Main results}

The information-theoretic threshold of the planted matching model is determined by the following key quantity, known has 
the Bhattacharyya coefficient (or Hellinger affinity)~\cite{bhattacharyya1943measure}:
\begin{equation}
B(\calP,\calQ) \triangleq \int \sqrt{f g}\ d\mu,
\label{eq:B}
\end{equation}
where $f,g$ denote the relative density of $\calP,\calQ$ with respect to some common dominating measure $\mu$, respectively. 
It is well-known that $0 \leq B(\calP,\calQ) \leq 1$, with $B(\calP,\calQ)$ equals $0$ (resp.~$1$) when $\calP$ and $\calQ$ are mutually singular (resp.~identical). For simplicity, we assume throughout the main body of the paper that $\calP\ll\calQ$, so that their relative density, denoted henceforth by $\frac{\calP}{\calQ}$, is well-defined. Nevertheless,
even when $P \not \ll Q$ (such as the weight distributions considered in \cite{Chertkov2010}), all results continue to hold (see Appendices~\ref{app:acpos} and \ref{app:acneg} for justification).

We first give a sufficient condition for 
the maximum likelihood estimator (MLE) to achieve almost perfect recovery.
The MLE reduces to the max-weighted matching on $K_{n,n}$ where each edge $e$ is weighted by the corresponding log likelihood ratio, namely,
\begin{equation}
\Mmin \in \arg \max_{M \in\calM}  \; \sum_{e\in M}  \log \frac{\calP}{\calQ}(W_e),
\label{eq:MLE}
\end{equation}
which can be computed in polynomial time (as linear assignment).
\begin{theorem}
\label{thm:possibility} 
Assume that
\begin{equation}
\sqrt{d} \;B(\calP, \calQ) \leq 1+\epsilon
\label{eq:possibility-condition}
\end{equation}
for some $\epsilon\geq 0$.
Then there exists universal constant $C>0$, such that for large enough $n$,
\[
\Expect\left[\risk(\Mplanted,\Mmin)\right]\leq C\max\left\{\log(1+\epsilon), \sqrt{\frac{\log n}{n}}\right\}.
\] 
\end{theorem}
Note that Theorem~\ref{thm:possibility} allows $\epsilon$ to take arbitrary values, including $0$. In particular, if $\sqrt{d}B(\calP,\calQ)\leq 1$, then $\Expect[\risk(\Mplanted,\Mmin)]\le C \sqrt{\log n/n}$; if~\eqref{eq:possibility-condition} holds with $\epsilon\rightarrow 0$ as $n\rightarrow \infty$, then $\Mmin$ achieves almost perfect recovery.

Next we proceed to negative results, which are the main focus of this paper. 
The following theorem shows the tightness of the condition \prettyref{eq:possibility-condition}:

\begin{theorem}
\label{thm:impossibility}
Assume that 
\begin{equation}
\sqrt{d} \;B(\calP, \calQ) \geq 1 + \epsilon.
\label{eq:impossibility}
\end{equation}
for some arbitrary constant $\epsilon>0$.
Suppose that  $\rho$ is continuous at $0$ and $\rho(0)<\infty$. 
Then in both the sparse and the dense model, for any estimator $\hat M$ and large $n$, 
\begin{equation}
\Expect[\risk(M^*,\hat{M})] \ge  c,
\label{eq:risk-lb}
\end{equation}
where $c>0$ is a constant independent of $n$. In the sparse model, $c$ only depends on $\epsilon,\calP,\calQ$; in the dense model, $c$ only depends on $\epsilon, \calP, \rho$. 
\end{theorem}

\prettyref{thm:possibility}  and \prettyref{thm:impossibility} together 
establish $\sqrt{d} \; B(\calP, \calQ) =1$ as the sharp threshold for almost perfect recovery in both the 
sparse and dense model, proving the conjecture in \cite{Semerjian2020}.\footnote{To be precise, the conjecture given in~\cite[eq.\ (45) and (40)]{Semerjian2020} is stated under the aforementioned unipartite version of the planted matching problem.
Nevertheless, our proof techniques as well as the sharp thresholds in Theorems \ref{thm:possibility}--\ref{thm:exp} continue to hold for the unipartite version.}
In the dense model with the scaling \prettyref{eq:q-dense}, the condition \prettyref{eq:impossibility} simplifies to
\begin{equation}
\int_0^\infty \sqrt{p(x)} dx \geq \frac{1+\epsilon}{\sqrt{\rho(0)}}.
\label{eq:impossibility-dense}
\end{equation}
Note that this condition depends on the density function $\rho$ in the null case only through its value at zero. In the special case of exponential weights of $p(x)=\lambda e^{-\lambda x}$ and $\rho(x)=e^{-x}$, the condition \prettyref{eq:impossibility-dense} further simplifies to 
\begin{equation}
\lambda \leq 4 - \epsilon.
\label{eq:impossibility-exp}
\end{equation}
In view of the positive result in \cite{Moharrami2020a}, this establishes $\lambda=4$ as the sharp threshold of almost perfect recovery, resolving a conjecture in \cite{Moharrami2020a}.
The next result, specialized to the exponential model, shows that the optimal reconstruction error is in fact $e^{-\Theta(\frac{1}{\sqrt{\epsilon}})}$, resolving a conjecture in \cite{Semerjian2020}.
Interestingly, this shows that the phase transition in the average reconstruction error is of infinite order,\footnote{In statistical physics parlance, a phase transition is called \textit{continuous} if the order parameter (in this case, the average reconstruction error) is continuous at the threshold, and \emph{of $p$th order} if its $(p-2)$th derivative is continuous \cite{Semerjian2020}.} unlike 
 other well-known planted problems such as  the stochastic block model with two groups (second order) or with four or more groups (first order)~\cite{Moore2017}. 

\begin{theorem}[Optimal reconstruction error for exponential model]
\label{thm:exp}
There exist an absolute constant $C_0$ such that the following holds. 
Suppose that $\lambda=4-\epsilon$ for some arbitrary constant $\epsilon>0$. 
Then there exists $n_0=n_0(\epsilon)$, such that for all $n\geq n_0$ and for any estimator $\hat M=\hat{M}(W)$, 
\begin{equation}
\Expect[\risk(M^*,\hat{M})] \ge \e^{-\frac{C_0}{\sqrt{\epsilon}}}.
\label{eq:risk-lb-exp}
\end{equation}
Furthermore, let $\Mmin$ be given in \prettyref{eq:MLE} (which in this case coincides with the min-weight bipartite matching). Then 
\begin{equation}
\Expect[\risk(M^*,\Mmin)] \le \frac{C_0}{\epsilon^3} \e^{-\frac{2\pi}{\sqrt{\epsilon}}}.
\label{eq:risk-ub-exp}
\end{equation}
\end{theorem}

We end this section with two remarks on the universality of the sharp threshold and infinite-order phase transitions. 
\begin{remark}
The sharp threshold  $\sqrt{d}  \;  B(\calP, \calQ) = 1$ does not hold universally for all $d, \calP, \calQ$. 
Here is a simple example where this condition is not tight. 
Consider a complete graph with Gaussian edge weights drawn from either $\calP=N(\mu,1)$ or $\calQ=N(0,1)$. This model does not follow the scaling in \prettyref{eq:q-dense} for the dense regime. In the Gaussian model, $\sqrt{n}  \;  B(\calP, \calQ) = 1$ simplifies to $\mu^2=4 \log n$. However, the sharp threshold in fact occurs at $\mu^2 = 2 \log n$. Indeed, when $\mu^2 \geq (2+\epsilon) \log n$, a simple thresholding algorithm finds a matching that differs from the planted matching by $o(n)$ edges with high probability; conversely, using the mutual information argument in \cite{DWXY20}, it is easy to show that almost perfect recovery is impossible if  $\mu^2 \leq (2-\epsilon) \log n$ for any constant $\epsilon>0$.	
\end{remark}

\begin{remark}[Finite-order phase transition for unweighted graphs]\label{rmk:unweighted}
Although we believe the infinite-order phase transition established in \prettyref{thm:exp} holds beyond exponential weights, this turns out to be not a universal phenomenon. In fact, the phase transition is of a \emph{finite order} for sparse unweighted graphs. 
In this case, the observed graph is a bipartite \ER graph $G(n,n,\frac{d}{n})$ with a planted perfect matching. Applying Theorems~\ref{thm:possibility} and \ref{thm:impossibility} with $\calP=\calQ$, we conclude that the almost perfect recovery is possible if and only if $d \le 1$. 
When $d=1+\epsilon$ for small $\epsilon$, on the one hand, \prettyref{thm:possibility} 
shows that the average reconstruction error of MLE
is at most $O(\epsilon)$. On the other hand,
by slightly modifying the proof of Theorem~\ref{thm:impossibility}, in  \prettyref{app:unweighted} we show that  
the average reconstruction error is at least $\Omega(\epsilon^8)$. 
Determining the exact order of the phase transition for unweighted \ER model is an open problem.
\end{remark}


\subsection{Proof techniques}



The proof of \prettyref{thm:possibility} is a simple application of large-deviation analysis and the union bound, similar to that of \cite[Theorem 1]{Moharrami2020a}. 
The bulk of the paper is devoted to proving the negative results of Theorems \ref{thm:impossibility} and~\ref{thm:exp}, which is much more challenging.

Our starting point is the simple observation that to prove the impossibility of almost perfect recovery, it suffices to consider the random matching sampled from the posterior distribution, which
reduces the problem to studying the typical behavior of this Gibbs distribution. 
We aim to show that with high probability, there is more posterior mass over the \emph{bad} matchings (those far away from the hidden one) than that over the \emph{good} matchings (those near the hidden one) in the posterior distribution. Via a first-moment calculation with proper truncation, 
it is not hard to bound from above the total posterior mass 
of good matchings.  To bound from below the posterior mass  of bad matchings, a key observation is that for a perfect matching $M$, the symmetric difference $M^*\symdiff M$ consists of a disjoint union of even cycles which alternate between planted and unplanted edges. Therefore, a natural idea is to show the existence of many alternating cycles of length $\Theta(n)$ that are \emph{augmenting}, that is, cycles for which the unplanted edges have a total log-likelihood that exceeds that of the planted edges. 
Unfortunately, a straightforward second-moment calculation fundamentally fails, due to  the excessive correlations among 
long augmenting cycles. 
To construct the desired long augmenting cycles, we instead proceed in two steps: First, we construct many disjoint alternating paths of constant lengths; then we connect them to form exponentially many distinct augmenting cycles using the remaining edges via sprinkling. The first step is achieved by greedily exploring the local neighborhoods in analogy to a super-critical branching process, and the second step can be attained by reducing it to a problem of finding long cycles in a super-critical Erd\H{o}s-R\'{e}nyi bipartite graph with a planted perfect matching. 
This two-stage cycle finding scheme suffices to prove the sharp threshold in \prettyref{thm:impossibility}. 

Although the above construction suffices for determining the sharp threshold, 
the local neighborhood exploration is too wasteful to extract sufficiently long alternating paths 
and falls short of proving the optimal reconstruction error bound in~\prettyref{thm:exp} for the exponential model.
To resolve this inefficiency, in the first stage, following the program in \cite{ding2015percolation}, we use the truncated first and second moment methods 
to show the existence of many alternating paths that are sufficiently long, then applying Tur\'an's theorem to extract a large disjoint subcollection. 
Notably, we further impose extra \emph{uniformity} constraints introduced in \cite{ding2013scaling,ding2015supercritical} on
the weights of the alternating paths to reduce the correlations in the second moment calculation, while at the same time keeping the first moment large.

In passing, we remark that our proof strategy of the impossibility results significantly deviates from most existing approaches in the literature. In many planted problems 
such as community detection in stochastic block models~\cite{Decelle11} or sparse PCA~\cite{lesieur2015mmse}, the optimal overlap (one minus reconstruction error) exhibits
a sharp transition from zero to strictly positive. 
Such correlated recovery threshold can be established via either mutual information arguments or reduction to detection (hypothesis testing) -- see \cite{WX18} for a survey; however, these techniques are either too loose or inapplicable for our model,
where the optimal overlap undergoes a phase transition from strictly less than one to one.


Finally, we briefly discuss the planted $k$-factor model recently studied in \cite{sicuro2020planted}. The special case $k=1$ is a variant of the planted matching model. 
The conjectured threshold for almost perfect recovery is at $ \sqrt{kd} B(\calP,\calQ)=1$. The positive direction of this conjecture can be established by extending the proof of \prettyref{thm:possibility}. Extending the impossibility results of \prettyref{thm:impossibility} 
to the planted $k$-factor model is an interesting future direction.



\subsection{Organization}
	\label{sec:org}

The rest of the paper is organized as follows. In \prettyref{sec:positive} we prove the positive result in \prettyref{thm:possibility} by analyzing the MLE. In \prettyref{sec:negative}, we outline the proof of the negative results in \prettyref{thm:impossibility} and \prettyref{thm:exp}, which are the main results of this paper. In this section, we focus on the sparse model and the exponential model (with general dense model deferred till \prettyref{app:reduction_dense_sparse}). The negative results are proved by analyzing the posterior distribution, which relies on a cycle-finding scheme involving a path construction stage (detailed in Sections \ref{sec:path.construction} and \ref{sec:exponential} for sparse and exponential models, respectively) 
 and a sprinkling stage (specified in \prettyref{sec:sprinkling}). 
Combining results from Sections \ref{sec:path.construction} and \ref{sec:sprinkling}, in \prettyref{sec:impossibility.proof} we prove the key \prettyref{lmm:Mbad} previously stated in \prettyref{sec:negative}, thereby finishing the proof of \prettyref{thm:impossibility} under the sparse model. In \prettyref{sec:exponential}, we give the details for the exponential model and complete the proof of the negative part of \prettyref{thm:exp}. 


The appendix contains auxiliary technical results and postponed proofs. \prettyref{app:LD} contains the large deviation results that are used throughout the paper. 
Next, two auxiliary results crucial for proving the negative result under the exponential model (\prettyref{thm:exp}) are presented:
 \prettyref{app:erlang} contains the Chernoff bounds of the Erlang distribution, and \prettyref{app:bridge} recalls a technical lemma from \cite{ding2015percolation} for controlling the deviation of Exp-minus-one random bridges. 
The positive part of \prettyref{thm:exp} is proved in \prettyref{app:ode}. \prettyref{app:reduction} contains all reduction-type arguments used in our proof: In \prettyref{app:reduction_dense_sparse}, we prove \prettyref{thm:impossibility} under the dense model by reducing it to the sparse model; 
in \prettyref{app:acpos} and \prettyref{app:acneg}, we drop the absolute continuity conditions on distributions $\calP,\calQ$, which are assumed in the proof given in the main part of the paper. \prettyref{app:unweighted} proves the finite-order phase transition for unweighted graphs previously announced in \prettyref{rmk:unweighted}.


\section{Positive results via maximal likelihood}\label{sec:positive}




To prove \prettyref{thm:possibility}, we first consider a general case of complete graph ($d=n$) with arbitrary weight distributions, 
 and then deduce the result for the general planted matching model by refining the weight distribution to incorporate edges not in $G$.
\begin{theorem}
\label{thm:first_moment_bound} 
For the dense model with parameters $(n,\calP,\calQ)$, suppose that
\begin{equation}
\sqrt{n}  \;  B(\calP, \calQ)
 \leq 1+\epsilon.
\label{eq:possibility-condition-general}
\end{equation}
for some $\epsilon\geq 0$.
Then there exists universal constant $C>0$, such that for large enough $n$,
\[
\expect{\risk(\Mplanted,\Mmin)}\leq C\max\left\{\log (1+\epsilon), \sqrt{\frac{\log n}{n}}\right\}.
\]
\end{theorem}

%

\begin{proof}[Proof of Theorem \ref{thm:first_moment_bound}]
Recall our standing assumption that $\calP\ll \calQ$. We give in~\prettyref{app:acpos} a reduction-based argument that handles the case where $\calP$ is not absolutely continuous with respect to $\calQ$. Under this assumption, the likelihood-ratio $\calP/\calQ$ is well-defined and we can apply the large deviation result~\eqref{eq:LDXY}. 
We also assume WLOG that $\epsilon \le 1$. For a fixed $M\in \calM$ for which $M^*\symdiff M$ contains $2t$ edges, we have
\[
\mathbb{P}\left\{\sum_{e\in M}\log \frac{\calP}{\calQ}(W_e)\geq \sum_{e\in M^*}\log\frac{\calP}{\calQ}(W_e)\right\}
=\mathbb{P}\left\{\sum_{i=1}^t (Y_i-X_i)\geq 0\right\},
\]
where $X_i$'s and $Y_i$'s be two independent sequences of random variables such that $X_i$'s are i.i.d.\ copies of $
\log(\calP/\calQ)$ under distribution $\calP$ and $Y_i $'s are i.i.d.\ copies of $\log(\calP/\calQ)$ under distribution $\calQ$. Using standard large-deviation estimates (see~\eqref{eq:LDXY} in \prettyref{app:LD}), the RHS of the inequality above is upper bounded by $e^{-t\alpha}$, where
\begin{align}
\alpha \triangleq -2 \log B(\calP,\calQ)
\label{eq:alpha}
\end{align}
is  the R\'enyi divergence of order $\frac{1}{2}$ between distributions $\calP$ and $\calQ$.
Since there are at most $\binom{n}{t}t! \le n^t e^{-t(t-1)/2n} $ perfect matchings differing from the true matching $M^*$
by $2t$ edges, it follows from a union bound and the assumption~\eqref{eq:possibility-condition-general} that
\begin{align}
\prob{|\Mplanted \symdiff \Mmin| \ge \beta n} 
\le \sum_{t\ge \beta n} \binom{n}{t} t! e^{-t\alpha}
&\le e^{1/2} \sum_{t\ge \beta n} e^{-t \alpha +t \log n - t^2/(2n)} \label{eq:exact_step}\\
&\le e^{1/2} \sum_{t\ge \beta n}  \left( (1+\epsilon)^2 e^{-\beta/2} \right)^t  \nonumber  \\
&\le e^{1/2} \frac{e^{-\beta^2n/4}}{1-e^{-\beta/4}}, \nonumber 
\end{align}
where  the last equality holds for all strictly positive $\beta\geq 8 \log(1+\epsilon)$ so that $(1+\epsilon)^2 e^{-\beta/2}\leq e^{-\beta/4}<1$.
Thus, 
 \begin{align*}
 \expect{|\Mplanted \symdiff \Mmin|} 
 & =
 \expect{|\Mplanted \symdiff \Mmin|\indc{|\Mplanted \symdiff \Mmin|\le \beta n}}+ 
 \expect{|\Mplanted \symdiff \Mmin|\indc{|\Mplanted \symdiff \Mmin|> \beta n} } \\
& \le \beta n + n \prob{|\Mplanted \symdiff \Mmin| \ge \beta n}\\
&\le \beta n + e^{1/2}\frac{e^{-\beta^2n/4}}{1-e^{-\beta/4}} n.
 \end{align*}

Next, by choosing
\[
\beta = \max\left\{8\log(1+\epsilon), 2\sqrt{\frac{\log n}{n}}\right\},
\]
we have $e^{-\beta^2 n/4}\leq \beta^2/4$ and $1-e^{-\beta/4} \ge \beta/8.$ 
Therefore, 
\[
\expect{\risk(\Mplanted,\Mmin)}
=\frac{1}{n} \expect{|\Mplanted \symdiff \Mmin|}
 \leq \beta+e^{1/2}\frac{\beta^2/4}{\beta/8}
 \leq 5 \beta \leq 10 \max\left\{ 4\log(1+\epsilon), \sqrt{\frac{\log n}{n}}\right\}.
\]
\end{proof}

\prettyref{thm:possibility} follows from \prettyref{thm:first_moment_bound}  as a corollary. 

\begin{proof}[Proof of \prettyref{thm:possibility}]
To apply  \prettyref{thm:first_moment_bound}, 
let us first reformulate the planted matching model in \prettyref{def:model} in a more convenient form. Recall that the observed weight on the edge $e$ is drawn independently from $\calP$ or $\calQ$ depending on whether $e$ belongs to the planted matching $M^*$ or not, where $\calP$ and $\calQ$ are arbitrary probability measures on some space $\calX$. Let us use a special symbol $\star\notin \calX$ to signify an edge $e$ that is not in $G$ and write $W_e=\star$. Therefore, for $e\in M^*$, $W_e$ is drawn from $\calP$ independently; for each $e \in E(K_{n,n}) \backslash M^*$, with probability $1-\frac{d}{n}$, $W_e=\star$; with probability $\frac{d}{n}$, $W_e$ is drawn from $\calQ$ independently. In other words, $W_e \iiddistr \calQ ' \triangleq (1-\frac{d}{n}) \delta_{\star }+\frac{d}{n} \calQ$ for $e \notin M^*$. The model is thus reformulated into a dense model with parameters $(n,\calP,\calQ')$.


 By \prettyref{thm:first_moment_bound}, $\Mmin$ achieves $\Expect[\risk(\Mplanted,\Mmin)] \leq C\max\left\{ \log (1+\epsilon),\sqrt{\frac{\log n}{n}}\right\}$, provided that 
 $\sqrt{n} B(\calP, \calQ')\leq 1+\epsilon$. Since $\sqrt{n} B(\calP, \calQ') = \sqrt{d} B(\calP, \calQ)$, \prettyref{thm:possibility} readily follows.
\end{proof}

\begin{remark}\label{rmk:first_moment}
We note that when $\alpha  - \log n \to +\infty$, from~\prettyref{eq:exact_step}
we have  $\Mmin=\Mplanted$ with high probability, \ie, 
$\Mmin$ achieves exact recovery. 
This coincides with the exact recovery threshold for the hidden Hamiltonian cycle problem~\cite{bagaria2020hidden}, in which a (unipartite) weighted graph is observed such that the edge weights on the planted Hamiltonian cycle are drawn from $\calP$ and other weights are drawn from $\calQ$.
 \end{remark}

\section{Negative results via analyzing posterior distribution}
\label{sec:negative}
In this section, we prove the impossibility results in \prettyref{thm:impossibility} and \prettyref{thm:exp} by directly analyzing the posterior distribution. For \prettyref{thm:impossibility}, we focus on the sparse model where the average degree $d$ is a constant. The impossibility result under the dense model follows from a reduction argument (see \prettyref{app:reduction_dense_sparse}).


\subsection{Proof outline of negative results}
\label{sec:neg.proof.outline}
In this subsection we outline the proof of the impossibility results in \prettyref{thm:impossibility} under the sparse model, and \prettyref{thm:exp}.

The negative results are proved by studying the posterior distribution of the hidden matching. 
Recall in the proof of \prettyref{thm:possibility} the reformulation of the planted matching model in \prettyref{def:model} in the complete graph  model with weight distributions $(\calP, \calQ ')$. 
In particular, $W_e \iiddistr \calQ ' \triangleq (1-\frac{d}{n}) \delta_{\star }+\frac{d}{n} \calQ$ for $e \notin M^*$. Here for simplicity we again assume $\calP\ll\calQ$ (the general case is handled in \prettyref{app:acneg}). Under this assumption, $\calP\ll\calQ'$ and the density $\calP/\calQ'$ is well-defined. 
The likelihood function of $W=(W_e: e\in E(K_{n,n}))$ given $M^*=m$ is given by 
\begin{equation}
\begin{aligned}
 \prob{W \mid M^*=m} = \prod_{e \in m} \calP(W_e) \prod_{e \notin m} \calQ'(W_e) 
&\propto \prod_{e\in m}  \frac{\calP}{\calQ '}(W_e) = \exp\underbrace{\pth{\sum_{e\in m} \log \frac{\calP}{\calQ '}(W_e)}}_{\triangleq L(m)} \, \label{eq:density_planted},
\end{aligned}
\end{equation}
where $\log \frac{\calP}{\calQ '}$ takes extended real-values in $\reals\cup\{-\infty\}$, and $L(m) = \sum_{e \in m} W_e$ is the total log-likelihood ratio on a set $m$ of edges.
Thus, conditioned on $W$, the posterior distribution of $M^*$ is a Gibbs distribution, given by
\begin{equation}
\mu_W(m) = \frac{1}{Z(W)}  \exp \left( L(m) \right), \quad m \in \calM,
\label{eq:posterior}
\end{equation}
where $Z(W) = \sum_{m\in\calM} \exp \left( L(m) \right)$ is the normalization factor. 

In order to reduce the impossibility proof to a statement on the posterior distribution, the first observation is that  
it suffices to consider the estimator $\tilde M$ which is sampled from the posterior distribution $\mu_W$. 
Indeed, 
given any estimator $\hat M=\hat M(W)$, we have
$(\hat M,M^*) \overset{\calL}{=} (\hat M,\tilde M)$ (in distribution) and hence
\[
\Expect[d(\tilde M,M^*)] \leq  \Expect[d(\tilde M,\hat M)]+  \Expect[d(M^*,\hat M)]= 2 \Expect[d(M^*,\hat M)],
\]
which shows that $\tilde M$ is optimal within a factor of two.
Thus it suffices to bound $\Expect[d(\tilde M,M^*)] $ from below.
To this end, fix some $\delta$ to be specified later and 
define the sets of good and bad solutions respectively as 
\begin{align*}
\Mgood = & ~ \{M \in \calM: \ell(M,M^*) < 2 \delta \} \\
\Mbad = & ~ \{M \in \calM: \ell(M,M^*)  \geq 2 \delta \}.
\end{align*}
By the definition of $\tilde M$, we have 
\[
\Expect[\ell(\tilde M,M^*)] \geq 2\delta  \cdot \Expect[\mu_W(\Mbad)].
\]
Next we show
\begin{lemma}
\label{lmm:Mgood}	
Assume  \prettyref{eq:impossibility} holds for some arbitrary constant $\epsilon>0$. 
There exist an absolute constant $\epsilon_0$ and constants $n_0=n_0(\epsilon)$ and $c=c(\delta,\epsilon)$, such that for all $\epsilon < \epsilon_0$ and $n\geq n_0$, 
in both the sparse and the dense model, 
with probability at least $1-e^{- c n/\log n}$, 
	\begin{equation}
	\frac{\mu_W(\Mgood)}{\mu_W(M^*)} \leq 2 e^{c_1n},
	\label{eq:Mgood}
	\end{equation}
	where $c_1 = 7 \epsilon \delta$.
\end{lemma}

\begin{lemma}
\label{lmm:Mbad}
Consider the sparse model 
such that \prettyref{eq:impossibility} holds for some arbitrary constant $\epsilon>0$. 
There exist an absolute constants $\epsilon_0$ and constants $c_0$, $c$, 
$c_2$, $n_0$ that only depend on $\epsilon, \calP, \calQ$, 
such that for all $\epsilon<\epsilon_0$, $n \geq n_0$, 
and $\delta \leq c_2$, 
with probability at least $1-\frac{c}{n}$, 
		\begin{equation}
	\frac{\mu_W(\Mbad)}{\mu_W(M^*)} \geq e^{c_0 n}.
	\label{eq:Mbad}
	\end{equation}

Furthermore, in the exponential model with $d=n$ and $\calP=\exp(\lambda)$ and $\calQ=\exp(1/n)$ where $\lambda \le 4-\epsilon$ for some arbitrary constant $\epsilon>0$,
\prettyref{eq:Mbad} continues to hold with probability at least 
$\frac{1}{2} - \frac{c}{n}$, where
all constants $c_0$, $c$, $c_2$, $n_0$ depend only on $\epsilon$ 
and 
$c_0, c_2=e^{-O(1/\sqrt{\epsilon})}$.
\end{lemma}

Given the above two lemmas, Theorems \ref{thm:impossibility} and \ref{thm:exp} readily follow. Indeed, combining
\prettyref{lmm:Mgood} and \prettyref{lmm:Mbad} and choosing 
$\delta = \min\{c_2, c_0/(14\epsilon)\}$ yields that $c_0 \ge 2c_1$
and hence
 $\mu_W(\Mbad) \geq \frac{e^{c_1 n}}{2+ e^{c_1n} }$ with probability at least $\frac{1}{2} - o(1) $, 
 which shows that $\Expect[\ell(\tilde M,M^*)] \gtrsim  \delta$ as desired
 in both the sparse and the exponential model on complete graph. 
 
We prove \prettyref{lmm:Mgood} in Section~\ref{sec:Mgood}, and outline the proof of \prettyref{lmm:Mbad} in Section~\ref{sec:Mbad}. For the rest of the proof, we will assume WLOG\footnote{Indeed, suppose that $\sqrt{d}B(\calP,\calQ)=c>1+\epsilon$. Consider the model parametrized with $(d',\calP,\calQ)$ where $d'=d(1+\epsilon)^2/c^2<d$ so that $\sqrt{d'}B(\calP,\calQ)=1+\epsilon$. From an observed graph $G$ generated from the $(d',\calP,\calQ)$ model, one can ``densify'' $G$ by add edges independently with edge weight drawn from $\calQ$ to arrive at an instance of the $(d,\calP,\calQ)$ model. Therefore, the lower bound on the average reconstruction error carries over to the $(d,\calP,\calQ)$ model.} that~\eqref{eq:impossibility} holds with equality. That is,
 \begin{equation}
\sqrt{d} \;B(\calP, \calQ) = 1+\epsilon,
\label{eq:impossible-assumption}
\end{equation}
for some small $\epsilon$. 

\subsection{Upper bounding the posterior mass of good matchings}
\label{sec:Mgood}

In this section, we prove \prettyref{lmm:Mgood}. Before proceeding with the proof, let us first introduce some notation that will be used throughout the remainder the paper. Fix the true matching $M^*$. We shall represent the planted edges (those in $M^*$) and the unplanted edges (those outside $M^*$) as red and blue edges, respectively. A cycle in $K_{n\times n}$ is called \emph{alternating} if it is an even cycle and alternates between red and blue edges. An important observation is that the symmetric difference between the truth $M^*$ and another perfect matching $M$ is always a disjoint union of alternating cycles. See \prettyref{fig:cycles} for an example. 

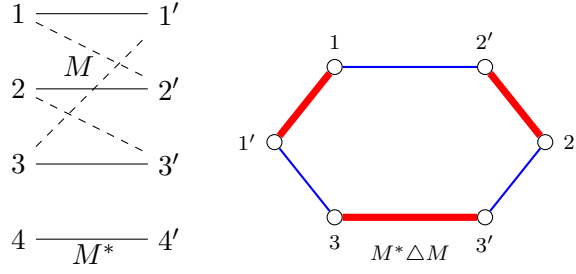
\begin{figure}[ht]
  \begin{center}
	\begin{tikzpicture}[scale=1,transform shape,node distance=1cm,auto]
\node at (0,3) (a1) {$1$};
\node at (0,2) (a2) {$2$};
\node at (0,1) (a3) {$3$};
\node at (0,0) (a4) {$4$};

\node at (2,3) (b1) {$1'$};
\node at (2,2) (b2) {$2'$};
\node at (2,1) (b3) {$3'$};
\node at (2,0) (b4) {$4'$};

\draw (a1) -- (b1);
\draw (a2) -- (b2);
\draw (a3) -- (b3);
\draw (a4) -- (b4);

\draw[dashed] (a1) -- (b2);
\draw[dashed] (a2) -- (b3);
\draw[dashed] (a3) -- (b1);

\node at (0.8,2.3) {$M$};
\node at (1,-0.2) {$M^*$};

\end{tikzpicture}
~~
\begin{tikzpicture}[scale=1,every edge/.append style = {thick,line cap=round},font=\small,font=\scriptsize]
\draw (-1,1) node (a1) [nodedot, label = above:$1$] {} ;
\draw (1,1) node (b2) [nodedot, label=above:$2'$] {};
\draw (1.8,0) node (a2) [nodedot, label = right:$2$] {};
\draw (1,-1) node (b3) [nodedot, label = below:$3'$] {};
\draw (-1,-1) node (a3) [nodedot, label = below:$3$] {};
\draw (-1.8,0) node (b1) [nodedot, label = left:$1'$] {};
	
\draw[Bedge]  (a1)--(b2);
\draw[Redge] (b2)--(a2);
\draw[Bedge]  (a2)--(b3);
\draw[Redge] (b3)--(a3);
\draw[Bedge]  (a3)--(b1);
\draw[Redge] (b1)--(a1);
	
\node at (0,-1.5) {$M^* \triangle M$};

 \end{tikzpicture}
\end{center}
\caption{Example of matching and alternating cycle ($n=4,\ell=3$). The matching $M^*$ and $M$ matches $1234$ to $1'2'3'4'$ and $2'3'1'4'$, respectively. The difference of their corresponding permutation is a cyclic shift of length three, i.e., $\pi^{-1}\pi'=(123)$ in the notation of cycle decomposition, and the symmetric difference graph $M^* \triangle M$ is an alternating $6$-cycle that contains $3$ planted (red, thick) edges and $3$ unplanted (blue, thin) edges.}%
\label{fig:cycles}%
\end{figure}

For any set $T$ of edges in $K_{n,n}$, let $\sfr(T) \triangleq T\cap M^*$ and $\sfb(T)\triangleq T\backslash M^*$ denote the set of red and blue edges in $T$, respectively. 
Define the \emph{excess weight} of $T$ as the total log-likelihood of blue edges minus that of the blue edges:
\begin{equation}\label{eq:Delta.def}
\Delta(T)=\sum_{e\in \sfb(T)}\log\frac{\calP}{\calQ'}(W_e) - \sum_{e\in \sfr(T)}\log\frac{\calP}{\calQ'}(W_e).
\end{equation}
For a perfect matching $M\in\calM$, we will denote $\Delta(M\symdiff M^*)$ simply as $\Delta(M)$. 
Recall the posterior distribution of $M^*$ given $W$ in \prettyref{eq:posterior}. Therefore, if $M^*=m^*$, we have
\begin{align*}
\Delta(m)
=\sum_{e\in m \backslash m^*} \log \frac{\calP}{\calQ '}(W_e) - \sum_{e\in m^* \backslash m} \log \frac{\calP}{\calQ '}(W_e)
= L(m)-L(m^*)
= \log \frac{\mu_W(m)}{\mu_W(m^*)}.
\end{align*}
Note that in the display above, the likelihood ratio $\calP/\calQ'$ can be replaced $\calP/\calQ$ if $m$ only contains edges in $G$. To see this, note the fact that $|m \backslash m^*|=|m^* \backslash m|=\ell$ and $\log \frac{\calP}{\calQ '}(w)=\log \frac{\calP}{\calQ}(w)+\log \frac{n}{d}$ whenever $w\neq \star$. Moreover, if $W_e=\star$ for some $e\in m$ then $\mu_W(m)=0$ and $\Delta(m)=-\infty$. Thus, we have
\begin{equation}
\Delta(m) = 
\begin{cases}
\sum_{e\in m \backslash m^*} \log \frac{\calP}{\calQ}(W_e) - \sum_{e\in m^* \backslash m} \log \frac{\calP}{\calQ}(W_e)  & W_e \neq \star, \forall e \in m \backslash m^* \\
-\infty & \text{else}.\\
\end{cases}
\label{eq:Deltam-cases}
\end{equation}

We have the following lemma:
\begin{lemma}
\label{lmm:muW}	
Let $m \in \calM$ be such that $|m\triangle m^*|=2\ell$.
For each $x\geq 0$, 
\[
\prob{\Delta(m) \geq {x\ell}  \mid  M^*=m^*}
\leq \pth{\frac{d}{n} e^{-(\alpha + x/2)}}^\ell.
\]
\end{lemma}
\begin{proof}
From~\eqref{eq:Deltam-cases}, we have
\begin{align*}
& ~ \prob{\Delta(m) \geq x\ell \mid M^*=m^*} \\
= & ~ \prob{\sum_{e\in m \backslash m^*} \log \frac{\calP}{\calQ '}(W_e) - \sum_{e\in m^* \backslash m} \log \frac{\calP}{\calQ '}(W_e) \geq x \ell  \mid M^*=m^*, W_e \neq \star, \forall e\in m} \pth{\frac{d}{n}}^\ell \\
\stepa{=} & ~ \prob{\sum_{e\in m \backslash m^*} \log \frac{\calP}{\calQ}(W_e) - \sum_{e\in m^* \backslash m} \log \frac{\calP}{\calQ}(W_e) \geq x \ell  \mid M^*=m^*, W_e \neq \star, \forall e\in m} \pth{\frac{d}{n}}^\ell \\
\stepb{=} & ~ \prob{\sum_{i=1}^\ell (Y_i-X_i) \geq x \ell} \pth{\frac{d}{n}}^\ell  \stepc{\leq} e^{-(\alpha+x/2)\ell} \pth{\frac{d}{n}}^\ell,
\end{align*}	
where (a) follows from \prettyref{eq:Deltam-cases}; 
in (b) we write $X_i$ and $Y_i$ are iid copies of $\log \frac{\calP}{\calQ}$ under $\calP$ and $\calQ$ respectively;
(c) follows from the large deviation bound \prettyref{eq:LDXY} in \prettyref{app:LD}.
\end{proof}

We now proceed to the proof of \prettyref{eq:Mgood}.
Recall that our standing assumption from~\eqref{eq:impossible-assumption} is that
\[
\sqrt{d} \,B(\calP, \calQ) = 1+\epsilon,
\]
for some small $\epsilon$. 
Throughout the proof we condition on $M^*=m^*$.

Note that 
\[
\frac{\mu_W(\Mgood)}{\mu_W(M^*)} =  \sum_{m: d(m,m^*) < 2\delta n } e^{\Delta(m)} = 
 R_1 + R_2, 
\]
where 
\begin{align*}
R_1 \triangleq   & ~\sum_{m: d(m,m^*) < 2\beta n/\log n }  e^{\Delta(m)} \\
R_2 \triangleq   & ~ \sum_{m: \frac{2 \beta n}{\log n} \leq d(m,m^*) < 2\delta n  } e^{ \Delta(m)}
\end{align*}
for some $\beta$ to be specified.
Next we bound $R_1$ and $R_2$ separately.


First, we note that 
given a perfect matching $m$, the number of matchings that differ from $m$ by $2\ell$ edges is
\begin{equation}
|\{m\in\calM: d(m,m^*)=2\ell\}| = !\ell \cdot \binom{n}{\ell},
\label{eq:mcount}
\end{equation}
where $!\ell$ denotes the number of derangements of $\ell$ elements, given by 
\[
!\ell = \ell! \sum_{i=0}^\ell \frac{(-1)^i}{i!} = \qth{\frac{\ell!}{e} },
\]
and $[\cdot]$ denotes rounding to the nearest integer. Thus
\begin{equation}
\frac{1}{2e} n(n-1)\cdots (n-\ell+1) \leq 
|\{m\in\calM: d(m,m^*)=2\ell\}| \leq \frac{2}{e} n(n-1)\cdots (n-\ell+1).
\label{eq:mcount2}
\end{equation}

Furthermore, for any $m$,
\begin{equation}
\Expect[\exp \left(  \Delta(m) \right)] = 1.
\label{eq:LR}
\end{equation}
This follows from the definition of log likelihood ratio as 
\[
\Expect[\exp \left(  \Delta(m) \right)] = 
\prod_{e\in m \backslash m^*} \underbrace{\Expect_{W_e\in \calQ '}\qth{\frac{\calP}{\calQ '}(W_e)}}_{=1} 
\prod_{e\in m^* \backslash m} \underbrace{\Expect_{W_e\in \calP}\qth{\frac{\calQ '}{\calP}(W_e)}}_{=1} 
 = 1.
\]
To bound $R_1$, using \prettyref{eq:mcount2} and \prettyref{eq:LR} we have
\begin{align*}
\expect{R_1} 
& = \sum_{d(m,m^*) < \frac{2\beta n}{\log n} } \expect{ e^{ \Delta(m)}}  \leq \sum_{\ell < \frac{\beta n}{\log n} }  \frac{2}{e} n^\ell \leq \frac{2n}{e} \exp(\beta n).
\end{align*}
By Markov's inequality, 
\begin{equation}
\prob{R_1 \geq e^{2 \beta n}} \leq \frac{2n}{e} \exp(-\beta n).
\label{eq:R1}
\end{equation}

To bound $R_2$, the calculation above shows that directly applying the Markov inequality is too crude since $\Expect[R_2]= e^{\Theta(n \log n)}$.
%
Note that although $\Delta(m)$ is negatively biased, 
when $\Delta(m)$ is atypically large it results in an excessive contribution to the exponential moments. Thus we truncate on the following event:
$$
\calE \triangleq  \bigcap_{m: 2\beta n/\log n \leq d(m,m^*) < 2\delta n}  \left\{ \Delta(m) \le r d(m,m^*)/2 \right\}
$$
for some constant $r$ to be chosen.
Then 
\begin{align*}
& \prob{ R_2  \ge e^{c' n} } \\
& \le \prob{\calE^c} + \prob{ \{R_2  \ge e^{c' n}\}   \cap \calE } \\
& \le \prob{\calE^c} +\prob{ \sum_{\frac{2\beta n}{\log n} \leq d(m,m^*)  <  2\delta n } e^{  \Delta(m)  }  \indc{ \Delta(m) \leq r d(m,m^*)/2}  \ge e^{c' n} } \\
& \le  \prob{\calE^c} + e^{-c' n} \sum_{\frac{2\beta n}{\log n} \leq d(m,m^*)  <  2\delta n } \expect{e^{ \Delta(m)}  \indc{ \Delta(m) \leq r d(m,m^*)/2} }. \numberthis \label{eq:R2}
\end{align*}

To bound the first term, note the fact that if $d(m,m^*)=2\ell$, by \prettyref{lmm:muW}, we have 
\[
\prob{\Delta(m) \geq x \ell} \leq \pth{\frac{d}{n} e^{-(\alpha+x/2)}}^\ell ,
\]
Therefore, it follows from a union bound and the fact that $\alpha=-2\log\int\sqrt{\calP\calQ}$ that
\begin{align*}
\prob{\calE^c} & = \sum_{ \frac{2\beta n}{\log n} \le d(m,m^*) <2\delta n }  \prob{ \Delta(m) \ge r d(m,m^*)/2 }  \\
& \le \frac{2}{e} \sum_{ \frac{\beta n}{ \log n} \le \ell <\delta n } n^\ell  
\pth{\frac{d}{n} \pth{\int \sqrt{\calP \calQ}}^2  }^\ell e^{-r\ell/2} \\
& = \frac{2}{e} \sum_{ \frac{\beta n}{ \log n} \le \ell <\delta n } \pth{(1+\epsilon)^2e^{-r/2}}^{\ell}
\end{align*}
Choose $r = 8 \epsilon$. We have 
\begin{equation}
\prob{\calE^c} \leq e^{-c \beta n/\log n} 
\label{eq:R2a}
\end{equation}
for some $c=c(\epsilon)$ and all sufficiently large $n$. 

For the second term in \prettyref{eq:R2a}, we bound the truncated MGF as follows: 
\begin{align*}
 \expect{e^{ \Delta(m)}  \indc{ \Delta(m) \le r \ell} } 
&\le \expect{\exp \left( \frac{1}{2} \left(  \Delta(m) + r \ell \right) \right) }  \\
& =  \expect{e^{ \Delta(m)/2} \mid  W_e \neq \star, \forall e \in m\backslash m^*} \pth{\frac{d}{n}}^{\ell}  e^{r \ell/2}\\
& \stepa{=}
\prod_{e\in m \backslash m^*} \Expect_{W_e\sim \calQ} \qth{\sqrt{\frac{\calP}{\calQ}(W_e)}}
\prod_{e\in m^* \backslash m} \Expect_{W_e\sim \calP} \qth{\sqrt{\frac{\calQ}{\calP}(W_e)}} \pth{\frac{d}{n}}^{\ell}  e^{r \ell/2}  \\
& = \pth{\frac{d}{n} B(\calP,\calQ)^2   }^{\ell}  e^{r \ell/2}  \leq \pth{\frac{e^{6\epsilon }}{n}}^\ell
\end{align*}
where (a) follows from \prettyref{eq:Deltam-cases}.
Combining the above with \prettyref{eq:mcount2}, we have
\begin{align*}
\sum_{\frac{2\beta n}{\log n} \leq d(m,m^*)  <  2\delta n } \expect{e^{ \Delta(m)}  \indc{ \Delta(m) \le  r d(m,m^*)/2} } 
\le   \frac{2}{e} \delta n e^{6 \epsilon \delta n}
\end{align*}
Choosing  $c' = 7 \epsilon \delta$, we get that 
\begin{equation}
 e^{-c' n}\sum_{\frac{2\beta n}{\log n} \leq d(m,m^*)  <  2\delta n } \expect{e^{ \Delta(m)}  \indc{ \Delta(m) \le  r d(m,m^*)/2} }  \le e^{- c \delta n}.
\label{eq:R2b}
\end{equation}
for some $c=c(\epsilon)$ and all large $n$. Substituting \prettyref{eq:R2a} and \prettyref{eq:R2b} into \prettyref{eq:R2}, we get
$$
\prob{R_2 \ge e^{7  \epsilon \delta n}  } \le e^{-\Omega(n/\log n)}
$$
Combining this with \prettyref{eq:R1} and upon choosing $\beta=\epsilon \delta$, we have 
$\prob{R_1+R_2 \geq 2 e^{ 7 \epsilon \delta n}} \leq e^{-\Omega(n/\log n)}$, concluding the proof.

\subsection{Lower bounding the posterior mass of bad matchings}
\label{sec:Mbad}

In this section, we outline the proof of \prettyref{lmm:Mbad}. 
Recall that 
\[
\frac{\mu_W(\Mbad)}{\mu_W(M^*)} =  \sum_{ m: d( m,m^*)  \ge 2\delta n } e^{ \Delta(m) } .
\]
For any perfect matching $m$ such that $d(m,m^*)=2\ell$, the set difference $m \triangle  m^*$ can be represented by a disjoint union of \emph{alternating} cycles,
denoted by $C$, where the edges in $m^* \backslash  m$
and $m \backslash m^*$ are colored red and blue respectively, so that in total there are $\ell$ red edges and $\ell$ blue edges. 

Recall from~\eqref{eq:Delta.def} that the excess weight $\Delta(C)$ of a cycle $C$
denotes the difference between the total blue and red edge (log-likelihood) weights. In order to lower bound the posterior mass of bad matchings, we show that with probability at least $1/2- c_1/n$, 
there exist at least $ e^{n c_2}$ distinct alternating cycles $C$ of length at least
$n c_3$, so that 
$\Delta(C) \geq n c_4$, for some constants $c_1, c_2, c_3, c_4$ that are independent of $n$ and only depend on $\epsilon, \calP, \calQ$.
Note that $c_1, c_2, c_3>0$, while $c_4$ is non-negative and is $0$ when $\calP=\calQ$. 
Since each alternating cycle $C$ corresponds to a perfect matching in $\Mbad$, we have 
\begin{align}
\frac{\mu_W(\Mbad)}{\mu_W(M^*)} \geq e^{n \left( c_2 + c_4 \right)}.  \label{eq:prob_bad_desired}
\end{align}


Let us point out a simple yet useful observation: if a perfect matching $m$ contains any edges not in $G$ ($e$ with $W_e=\star$), then it has zero posterior mass. Thus, all alternating cycles that give rise to a perfect matching in $\Mbad$ with positive posterior mass must consist only of edges in $G$. Therefore, to show~\eqref{eq:prob_bad_desired}, it suffices to consider the alternating cycles in $G$. Let us also remark that by the same reasoning as~\eqref{eq:Deltam-cases}, for any set $T$ of edges in $G$, 
\[
\Delta(T)=\sum_{e\in \sfb(T)} \log \frac{\calP}{\calQ}(W_e) - \sum_{e\in \sfr(T)} \log \frac{\calP}{\calQ}(W_e).
\]
Since we will be focusing on edge sets in $G$, for the remainder of the proof, we will work directly with $\calP/\calQ$ instead of $\calP/\calQ'$.

To show the existence of many alternating cycles in $G$, 
one natural idea is to define $S$ as the set of alternating cycles $C$ satisfying the aforementioned length and weight requirements;
and then bound the cardinality of  $S$ from below using the first and second-moment methods. 
This boils down to proving that $\var(|S|) \lesssim  ( \expect{|S|} )^2 $. 
Unfortunately, this idea fails because the variance of $|S|$ turns out to be exponentially larger than
$(\expect{|S|} )^2$, due to the excessive correlations among these long alternating cycles. 
A similar phenomenon is also observed in counting the number of long cycles in \ER graphs~\cite[Section 3]{marinari2006number}.

Alternatively, recall that the set difference  $m \symdiff  m^*$ 
is allowed to be a \emph{disjoint union} of alternating cycles rather than a \emph{single} alternating cycle. 
Thus, one can resort to showing the existence of disjoint unions of 
many but short alternating cycles satisfying the total length and weight requirements. 
However, using the first and second moment methods, at best we can show that 
there exists a disjoint union of $\Theta(n/\log^2 n) $ desired alternating cycles of length $\Theta(\log n)$;
thus the total length is  only $\Theta(n/\log n)$, falling short of meeting the total length requirement of $\Omega(n)$. 



To construct the desired alternating cycles, we instead proceed in two steps. 
We first reserve a set of vertices and construct many but short alternating paths with the desired total weight 
in the subgraph induced by the non-reserved vertices.
Then we use the edges incident to the reserved vertices to connect these paths to form the desired alternating cycles. 
At a high level, our two-stage cycle finding scheme is inspired by the previous work~\cite{ding2015percolation} in a different
context: The goal therein is to find a long path whose average weight is below a certain threshold in a complete graph with i.i.d.\ exponentially weighted edges.
More broadly, our second step is similar in spirit  to the sprinkling idea commonly used in random graph theory (see, e.g.~\cite[Section 11.9]{Alon2016}
and~\cite[Section 2.1.4]{krivelevich2016long}).

The specific construction is described in Algorithm~\ref{alg:cycle_finding}. Therein, $V$ denotes the set of reserved left vertices from $[n]$, 
where $|V|=\gamma n$  and $\gamma=\gamma_0 \epsilon$ for some small constant $\gamma_0>0$. 
Denote by $V'\subset [n]'$ its counterpart on the right side defined by the red edges. 
We write $V^c=[n]\backslash V$ and therefore $(V^c)' = [n]'\backslash V'$.

\begin{algorithm}[h]
\caption{Two-stage cycle finding algorithm} \label{alg:cycle_finding}
\begin{algorithmic}[1]
\State{\bfseries Input:} Weighted bipartite graph $G$ on $[n]\times [n]'$ with weight vector $w$, thresholds $\tau_\mathsf{red}$, $\tau_\mathsf{blue}$, and parameters $s,\Delta_0, c_5, c_6, c_7.$
   \State{\bfseries Step 1: Path construction.}
Let $G_1$ denote the weighted subgraph of $G$ induced by 
the edges in $V^c \times (V^c)'$. 
Construct a family of disjoint sets $L_k \subset V^c$ of left vertices 
and $R_k \subset (V^c)'$ of right vertices for $k \in \calK_1 \subset V^c$
such that: (1) $|L_k| \geq s, |R_k| \geq s$ and $K_1 \triangleq |\calK_1| \geq n c_5$; 
(2) each pair of vertices $u \in L_k$ and $v' \in R_k$ are connected via an alternating
path $P$ that starts and ends in red edges, with $\Delta(P)\geq \Delta_0$, where $\Delta(P)$ is defined in~\prettyref{eq:Delta.def}.

   \State{\bfseries Step 2: Sprinkling.}
Let $V^*=\{i\in V:\log(\calP/\calQ)(W_{i,i'})\leq \tau_\mathsf{red}\}$. Let $G_2$ be the subgraph of $G$ that contains every red edge in $V^*\times (V^*)'$ and every blue edge $e$ in $V^*\times (V^*)'$, $V^c\times (V^*)'$, or $V^*\times (V^c)'$, if and only if $\log(\calP/\calQ)(W_e)\geq \tau_\mathsf{blue}$.
Let $\{U_k: k\in \mathcal{K}_1\}$ (resp.~$\{V_k: k\in \mathcal{K}_1\}$) be a collection of disjoint subsets of left (resp.~right) vertices, such that  every vertex in $U_k'$ is connected to $L_k$ by at least one blue edge in $G_2$, and every vertex in $V_k$ is connected to $R_k$ by at least one blue edge in $G_2$. Next, on a subset $\calK_2\subset \calK_1$ of size $K_2\ge K_1/16$, define a bipartite ``super graph'' $G_{\rm super}$ with vertex sets $\calK_2$ and $\mathcal{K}_2'$. In $G_{\rm super}$, there is a red edge between $k$ and $k'$ for every $k\in \calK_2$, and a blue edge between $i$ and $j'$ if and only if there is at least a blue edge in $G_2$ connecting $U_i$ and $V'_j$. 
Construct $e^{c_6 K_1}$ distinct alternating cycles in $G_{\rm super}$, each of length at least $2c_7K_1$ for some universal constants $c_6,c_7>0$. 

\State{\bfseries Output:} 
Expand each alternating cycle on $G_{\rm super}$ into an alternating cycle on $G$,
by replacing each red edge $(k,k')$ in $G_{\rm super}$ by 
an alternating path between $U_k$ and $V'_k$ that starts and ends in red edges and connects 
$L_k$ and $R_k.$  See Figure~\ref{fig:sprinkling} for an illustration.  Output all the resulting alternating cycles on $G$.
\end{algorithmic}
\end{algorithm}

\begin{figure}[ht]
  \begin{center}
  \begin{tikzpicture}[scale=1,every edge/.append style = {thick,line cap=round},font=\scriptsize]
\draw (-1,1) node (l1)[unshadedgiantnode]{$1$};
\draw (-1,-1) node (l2)[unshadedgiantnode]{$2$};
\draw (1,1) node (r1)[unshadedgiantnode]{$1'$};
\draw (1,-1) node (r2)[unshadedgiantnode]{$2'$};
\draw[Bedge] (l1)--(r2) (l2)--(r1);
\draw[Redge] (l1)--(r1) (l2)--(r2);
\node at (0,-1.5) {$G_{\rm super}$};
   \end{tikzpicture}
~~\quad\quad\quad
  \begin{tikzpicture}[scale=1,every edge/.append style = {thick,line cap=round},font=\scriptsize]
  \draw (-2,2.5) ellipse (1cm and 0.3cm);
  \draw (2,2.5) ellipse (1cm and 0.3cm);
  \draw (-2,1.5) ellipse (1cm and 0.3cm);
  \draw (2,1.5) ellipse (1cm and 0.3cm);
  \draw (-2,0.5) ellipse (1cm and 0.3cm);
  \draw (2,0.5) ellipse (1cm and 0.3cm);
  \draw (-2,-0.5) ellipse (1cm and 0.3cm);
  \draw (2,-0.5) ellipse (1cm and 0.3cm);
  \draw (-2,-1.5) ellipse (1cm and 0.3cm);
  \draw (2,-1.5) ellipse (1cm and 0.3cm);
  \draw (-2,-2.5) ellipse (1cm and 0.3cm);
  \draw (2,-2.5) ellipse (1cm and 0.3cm);
\node at (-3.5,2.5) {$L_1$};
\node at (-3.5,1.5) {$U_1'$};
\node at (-3.5,0.5) {$U_1$};
\node at (-3.5,-0.5) {$L_2$};
\node at (-3.5,-1.5) {$U_2'$};
\node at (-3.5,-2.5) {$U_2$};
\node at (3.5,2.5) {$R_1$};
\node at (3.5,1.5) {$V_1$};
\node at (3.5,0.5) {$V_1'$};
\node at (3.5,-0.5) {$R_2$};
\node at (3.5,-1.5) {$V_2$};
\node at (3.5,-2.5) {$V_2'$};

\draw(-2.5,2.5) node(l11) [nodedot] {};  
\draw(-2,2.5) node(l12) [nodedot] {};  
\draw(-1.5,2.5) node(l13) [nodedot] {};  
\draw(-2.5,1.5) node(l21) [nodedot] {};  
\draw(-2,1.5) node(l22) [nodedot] {};  
\draw(-1.5,1.5) node(l23) [nodedot] {}; 
\draw(-2.5,0.5) node(l31) [nodedot] {};  
\draw(-2,0.5) node(l32) [nodedot] {};  
\draw(-1.5,0.5) node(l33) [nodedot] {};  
\draw(-2.5,-0.5) node(l41) [nodedot] {};  
\draw(-2,-0.5) node(l42) [nodedot] {};  
\draw(-1.5,-0.5) node(l43) [nodedot] {};   
\draw(-2.5,-1.5) node(l51) [nodedot] {};  
\draw(-2,-1.5) node(l52) [nodedot] {};  
\draw(-1.5,-1.5) node(l53) [nodedot] {};  
\draw(-2.5,-2.5) node(l61) [nodedot] {};  
\draw(-2,-2.5) node(l62) [nodedot] {};  
\draw(-1.5,-2.5) node(l63) [nodedot] {};  
\draw(2.5,2.5) node(r11) [nodedot] {};  
\draw(2,2.5) node(r12) [nodedot] {};  
\draw(1.5,2.5) node(r13) [nodedot] {};  
\draw(2.5,1.5) node(r21) [nodedot] {};  
\draw(2,1.5) node(r22) [nodedot] {};  
\draw(1.5,1.5) node(r23) [nodedot] {}; 
\draw(2.5,0.5) node(r31) [nodedot] {};  
\draw(2,0.5) node(r32) [nodedot] {};  
\draw(1.5,0.5) node(r33) [nodedot] {};  
\draw(2.5,-0.5) node(r41) [nodedot] {};  
\draw(2,-0.5) node(r42) [nodedot] {};  
\draw(1.5,-0.5) node(r43) [nodedot] {};   
\draw(2.5,-1.5) node(r51) [nodedot] {};  
\draw(2,-1.5) node(r52) [nodedot] {};  
\draw(1.5,-1.5) node(r53) [nodedot] {};  
\draw(2.5,-2.5) node(r61) [nodedot] {};  
\draw(2,-2.5) node(r62) [nodedot] {};  
\draw(1.5,-2.5) node(r63) [nodedot] {};  

\draw [Redge] (l21)--(l31) (l22)--(l32) (l23)--(l33) (l51)--(l61) (l52)--(l62) (l53)--(l63) (r21)--(r31) (r22)--(r32) (r23)--(r33) (r51)--(r61) (r52)--(r62) (r53)--(r63);
\draw [Bedge] (l12)--(l21) (l43)--(l51) (r13)--(r21) (r43)--(r51) (l11)--(l22) (l13)--(l23) (l43)--(l52) (l43)--(l53) (r13)--(r23) (r12)--(r22) (r11)--(r21) (r43)--(r53) (r41)--(r52) (l33)--(r63) (l63)--(r33);
\draw [red, dashed, line width = 2pt] (l13) -- node[black, above]{$P_1$} (r13);
\draw [red, dashed, line width = 2pt] (l43) -- node[black, above]{$P_2$} (r43);
\end{tikzpicture}
\end{center}
\caption{Example of the sprinkling stage of \prettyref{alg:cycle_finding}. To the left is an alternating cycle $C_{\rm super}=(1,1',2,2')$ on the super graph $G_{\rm super}$. To the right is its expansion to an alternating cycle on $G$. The two blue edges in $G_{\rm super}$ correspond to the two long crossed blue edges in $G$: the blue edge $(1,2')$ in $G_{\rm super}$ means that there exist $u_1\in U_1$, $v_2'\in V_2'$ such that $(u_1,v_2')$ is a blue edge in $G_2$. Similarly, $(u_2,v_1')$ is a blue edge for some $u_2\in U_2$ and $v_1'\in V_1'$. The dashed red edges represent alternating paths $P_1$ and $P_2$ that start and end in red edges. The existence of $P_1$ and $P_2$ follows from the first stage of \prettyref{alg:cycle_finding}.}%
\label{fig:sprinkling}%
\end{figure}
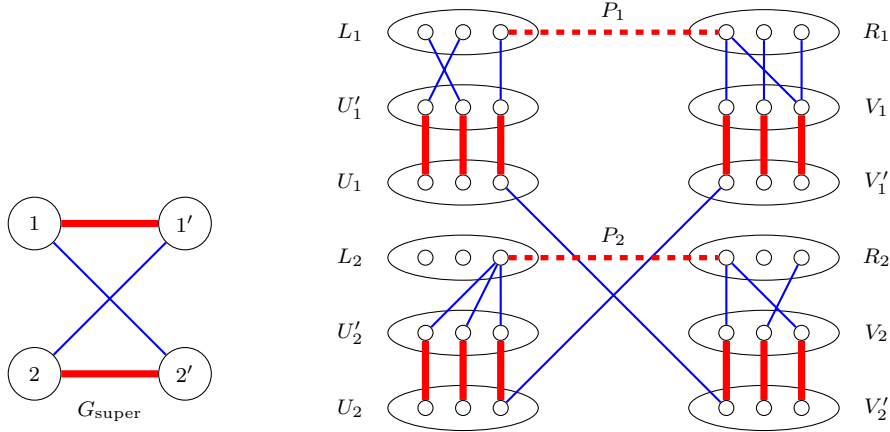

The construction of the super graph is the key step in Algorithm~\ref{alg:cycle_finding}. In the construction of $L_k,R_k,U_k,V_k$, only the edges in $V^c\times (V^c)'$, $V\times (V^c)'$ and $V^c\times V'$ are inspected. Therefore, whether a blue edge appears in the super graph is independent of all the steps prior, since it only depend on edges in $V\times V'$. Furthermore, the sets $\{U_k\}$ and $\{V_k\}$ are disjoint, so that the blue edges in the super graph are all formed independently. Therefore, $G_{\rm super}$ can be viewed as (or at least contains) an \ER bipartite graph with a planted perfect matching. 
In the sprinkling stage, by choosing the thresholds $\tau_\mathsf{red}$ and $\tau_\mathsf{blue}$ appropriately and by showing that the sets $\{U_k\}$ and $\{V_k\}$ are large for all $k\in \calK_2$, we ensure that the average degree of $G_{\rm super}$ is a large enough constant to be in the very supercritical regime.
In this regime, we will show in Lemma~\ref{lmm:bip.many.cycles} that $G_{\rm super}$ contains exponentially many alternating cycles of length $\Omega(K_2)=\Omega(n)$. As illustrated in Figure~\ref{fig:sprinkling}, each such ``super cycle'' can be expanded into an alternating cycles in $G$, which, in turn, gives rise to a perfect matching in $\Mbad$. Again by choosing the thresholds appropriately for the edges of $G_2$, we can make the sprinkling edge-weights negligible compared to the weights of the paths, so that for an alternating cycle $C$ formed by connecting paths $P_1,...,P_r$, we have $\Delta(C)\approx\sum_{k\leq r}\Delta(P_k)=\Omega(n)$, completing the proof of~\eqref{eq:prob_bad_desired}.

In \prettyref{alg:cycle_finding} the specific construction of disjoint paths is not spelled out. This part in fact
differs for the sparse model (\prettyref{thm:impossibility}) and the exponential model (\prettyref{thm:exp}, which are detailed in \prettyref{sec:path.construction} and \prettyref{sec:exponential} respectively. Specifically, for the sparse model with bounded average degree, the desired alternating paths are found by exploring the local neighborhood
in $G_1$ using breadth-first-search and constructing two-sided trees $T_k$ where $L_k$  and $R_k$ 
correspond to the set of leaves of the left- and right-sided tree, respectively. This suffices to prove the sharp threshold in~\prettyref{thm:impossibility} for the sparse model (and, by a reduction argument, for the dense model as well). However, this path finding scheme via constructing two-sided trees is wasteful in the sense
that the length of the path extracted is only about the depth of the tree, which is much smaller than the size of the tree;
thus the obtained cycles are not long enough to yield the optimal reconstruction error bound in~\prettyref{thm:exp} for the exponential model.
To remedy this, we take a more direct approach: Following the program in \cite{ding2015percolation},
we use first and second moment methods combined with Tur\'an's theorem to
show the existence of many disjoint, short alternating paths in $G_1$ of desired total weights. It is worth noting that this method
yields fewer paths in total, which makes the super graph smaller. Nevertheless, for exponential model we can appropriately choose the threshold for the edge weights in $G_2$ in the sprinkling stage, so that the super graph is still very supercritical. In contrast, under the sparse model with bounded average degree, this shortage of paths due to applying Tur\'an's theorem cannot be salvaged by tuning the threshold, and it is necessary to resort to neighborhood-exploration to find paths.

\section{Path construction (under the sparse model)}
\label{sec:path.construction}
In this section we construct $K$ disjoint subgraphs on $G_1$, which contain the disjoint alternating short paths we will later use to form the long cycles. In particular, each subgraph consists of two trees on $G_1$ whose root nodes are connected via a red edge. We will refer to the subgraphs as \emph{two-sided trees} denoted as $T_1,...,T_K$. See Figure~\ref{fig:two.sided.tree} for an illustration. For each $k$, our construction is such that on either side of $T_k$, the path from any leaf node to its root node is of alternating color that starts in a red edge and ends in a blue edge. As such, the path from any leaf node on one side to any leaf node on the other side is also an alternating path. The two-sided trees are constructed via a greedy neighborhood exploration process. Before elaborating on the exploration process, the following are a few desirable features that we aim for.

\begin{enumerate}
\item The trees are not too small: on either side, we want the exploration process to survive long enough to yield sufficiently long paths;
\item There is a large number of leaf nodes on each side of the trees, so that there are plenty of paths to choose from when forming the long cycles via sprinkling.
\item The trees are not ``overgrown'': since we need the $K$ trees to be vertex-disjoint, none of them should have too many vertices;
\item The edge weights on all the alternating paths are well behaved. The definition of ``well behaved'' will become clear later. This is to ensure that the long cycles constructed from these paths occupy sufficiently large posterior probability.
\end{enumerate}

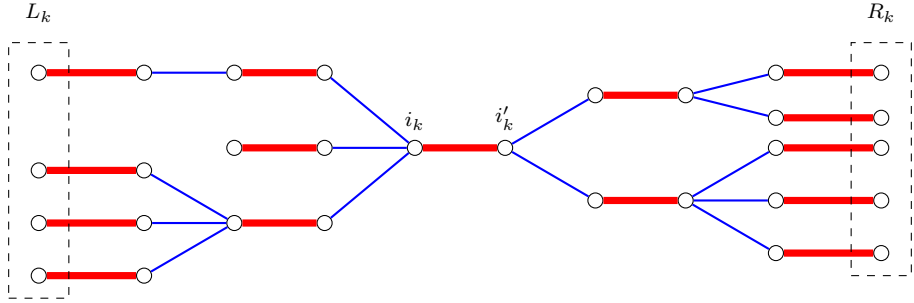
\begin{figure}[ht]
  \begin{center}
  \begin{tikzpicture}[scale=1,every edge/.append style = {thick,line cap=round},font=\scriptsize]
  \draw(-0.6,0) node (l0) [nodedot, label = above:$i_k$] {};
  \draw(0.6,0) node (r0) [nodedot, label = above:$i_k'$] {};
  \draw(-1.8,1) node (l1) [nodedot] {};
  \draw(-1.8,0) node (l2) [nodedot] {};
  \draw(-1.8,-1) node (l3) [nodedot] {};
  \draw(-3,1) node (l4) [nodedot] {};
  \draw(-3,0) node (l5) [nodedot] {};
  \draw(-3,-1) node (l6) [nodedot] {};
  \draw(-4.2,1) node (l7) [nodedot] {};
  \draw(-4.2,-0.3) node (l8) [nodedot] {};
  \draw(-4.2,-1) node (l9) [nodedot] {};
  \draw(-4.2,-1.7) node (l10) [nodedot] {};
  \draw(-5.6,1) node (l11) [nodedot] {};
  \draw(-5.6,-0.3) node (l12) [nodedot] {};
  \draw(-5.6,-1) node (l13) [nodedot] {};
  \draw(-5.6,-1.7) node (l14) [nodedot] {};
  \draw(1.8,0.7) node (r1) [nodedot] {};
  \draw(1.8,-0.7) node (r2) [nodedot] {};
  \draw(3,0.7) node (r3) [nodedot] {};
  \draw(3,-0.7) node (r4) [nodedot] {};
  \draw(4.2,1) node (r5) [nodedot] {};
  \draw(4.2,0.4) node (r6) [nodedot] {};
  \draw(4.2,0) node (r7) [nodedot] {};
  \draw(4.2,-0.7) node (r8) [nodedot] {};
  \draw(4.2,-1.4) node (r9) [nodedot] {};
    \draw(5.6,1) node (r10) [nodedot] {};
  \draw(5.6,0.4) node (r11) [nodedot] {};
  \draw(5.6,0) node (r12) [nodedot] {};
  \draw(5.6,-0.7) node (r13) [nodedot] {};
  \draw(5.6,-1.4) node (r14) [nodedot] {};
  \draw[Redge] (l0)--(r0) (l1)--(l4) (l2)--(l5) (l3)--(l6) (r1)--(r3) (r2)--(r4) (l7)--(l11) (l8)--(l12) (l9)--(l13) (l10)--(l14) (r5)--(r10) (r6)--(r11) (r7)--(r12) (r8)--(r13) (r9)--(r14);
  \draw[Bedge] (l0)--(l1) (l0)--(l2) (l0)--(l3) (l4)--(l7) (l6)--(l8) (l6)--(l9) (l6)--(l10) (r0)--(r1) (r0)--(r2) (r3)--(r5) (r3)--(r6) (r4)--(r7) (r4)--(r8) (r4)--(r9);
\draw[dashed] (-6,-2) -- (-6,1.4) -- (-5.2,1.4) -- (-5.2,-2) -- cycle;
\draw[dashed] (6,-1.7) -- (6,1.4) -- (5.2,1.4) -- (5.2,-1.7) -- cycle;
\node at (-5.6,1.8) {$L_k$};
\node at (5.6,1.8) {$R_k$};
 \end{tikzpicture}

  \end{center}
\caption{Example of a two-sided tree $T_k$. For each $u\in L_k$ and $v\in R_k$, there is an alternating path from $u$ to $v$ that starts and ends in red edges, and passes through the red edge $(i_k,i_k')$.}%
\label{fig:two.sided.tree}%
\end{figure}

In the remainder of this section, we give the precise construction of the two-sided trees, and show that sufficiently many of them fulfill all the features above.

\subsection{Construction of two-sided trees}\label{sec:exploration.construction}
As shown in Figure~\ref{fig:two.sided.tree}, each two-sided tree is centered at a red edge. We will refer to the subtree rooted at its left (resp.~right) end point as the left (resp.~right) subtree of this two-sided tree. We construct each left or right subtree via a neighborhood exploration process that starts from the root vertex. To ensure that the paths are alternating in color, each explored blue edge must be followed by (the unique) red edge (see Figure~\ref{fig:two.sided.tree}). In other words, the vertices are always explored in \emph{pairs}, where the pairs are identified by the red edges. Recall that $G_1$ is a bipartite graph with $(1-\gamma)n$ left and right vertices. We will have $\gamma n$ pairs of vertices actively participate in the construction, and leave $(1-2\gamma) n$ pairs unexplored in the neighborhood exploration process.


The local graph neighborhoods are explored in a fashion analogous to the breadth-first search (BFS). However, as opposed to the vanilla BFS algorithm, we design a two-stage exploration-selection scheme, where the neighborhoods (left and right subtrees) are grown in epochs, where each epoch contains a few exploration steps and a leaf node selection step. The selection step is necessary to ensure that the weights of the paths, and the resulting cycles, fulfill the weight requirement of bad matchings. Since we need to construct many disjoint two-sided trees, the selection step is done periodically to prevent each tree from overgrowing and using up too many vertices.

\begin{algorithm}[H]
\caption{Construction of two-sided trees} \label{alg:exploration}
\begin{algorithmic}[1]
\State{\bfseries Input:} $n,\gamma$, a bipartite graph $G_1$ that contains $(1-\gamma)n$ pairs of vertices, threshold $\zeta$,
and parameters $H,L,\epsilon,\alpha $.
\State Initialize $\calU=\{\text{all left vertices of }G_1\}$ as the set of unexplored left vertices. Define
\[
m=(1+\epsilon)^{2HL}\exp(3H\alpha+\epsilon H)
\]
for the number of left vertices each of the left and right subtrees will contain. Let $K=\frac{\gamma n}{2m}$. For $k=1,2,...K$, repeat the following steps 3-6 to construct $T_k$.
   \State Let $i_k$ be the member of $\calU$ with the smallest index. Update $\mathcal{U}\leftarrow \mathcal{U}\backslash\{i_k\}$ to mark $i_k$ as explored. Initialize $T_k=\{(i_k,i_k')\}$ to be a tree containing only one red edge.
	
   \State Construct the left tree of $T_k$ via $L$ epochs of color-alternating breadth-first search (BFS) on $G_1$. Define $V_0=\{i_k\}$. In the $i$'th epoch, the BFS starts from the the set of vertices in $V_{i-1}$. Repeat the following two steps for $i=1,...,L$:
\begin{enumerate}[(a)]
\item Exploration: for each $v\in V_{i-1}$, grow a color-alternating subtree in $G_1$ with $v$ as its root. Concretely, define the offsprings of $v$ as
\[
O_{v}'= \left\{u'\in \mathcal{U}': (u', v)\in E(G_1)\right\}.
\]
For all $u'\in O_v'$, append edges $(u',v)$ and $(u,u')$ to $T_k$. Update $\mathcal{U}\leftarrow\mathcal{U}\backslash O_v$ to mark all members of $O_v$ as explored.

Grow the next 2 layers of the left tree similarly: sequentially (ordered by the vertex indices) for each $u\in O_v$, define its offsprings as the set of all unexplored vertices that are connected to $u$ via a blue edge in $G_1$; append to $T_k$ all the blue edges from $u$ to the offsprings and their corresponding red edges; and mark all the offsprings as explored. 
Repeat this exploration step above until the tree is of depth $2H$, unless the exploration process becomes extinct or the following termination condition is met:
\begin{equation}
\label{eq:terminate}
\text{Terminate when the number of left vertices in the tree exceeds } m. 
\end{equation}

\item Selection of leaf nodes: for each $v\in V_{i-1}$, let $\mathcal{L}_v$ 
denote the set of all leaf nodes at depth $2H$ in the subtree rooted at $v$. Among those, select
\begin{equation}
\label{eq:pruning}
\mathcal{L}_{v,\mathsf{sel}}=\left\{u\in \mathcal{L}_v: \Delta(P_{u,v})\geq \zeta H\right\},
\end{equation}
where
\[
P_{u,v}= \text{the path from } u\text{ to }v\text{ on the subtree rooted at }v.
\]
Let
$
V_i=\cup_{v\in V_{i-1}} \mathcal{L}_{v,\mathsf{sel}}
$
be the set of selected leaf nodes from the $i$'th epoch. The next epoch of the construction will start from $V_i$.
\end{enumerate}
After repeating the exploration-selection process for $L$ epochs, we arrive at a tree of depth $2HL$, rooted at $i_k$.

\State Let $m'$ be the total pairs of vertices used in the construction of the left tree. Let $\mathcal{U}'$ be the first $m-m'$ indices in $\mathcal{U}$, and update $\mathcal{U}$ to be $\mathcal{U}\backslash \mathcal{U}'$.

\State Construct the right tree via the same scheme, starting from root vertex $i_k'$.

\end{algorithmic}
\end{algorithm}

\begin{figure}[ht]
  \begin{center}
  \begin{tikzpicture}[scale=1,every edge/.append style = {thick,line cap=round},font=\scriptsize]
  \draw(0,0) node (0) [nodedot, fill = black, label={[label distance=-0.1cm]left:$i_k$}] {};
  \draw(-1,-0.7) node (11) [nodedot] {};
\draw(0,-0.7) node (12) [nodedot] {};
\draw(1,-0.7) node (13) [nodedot] {};
  \draw(-1,-1.4) node (21) [nodedot] {};
\draw(0,-1.4) node (22) [nodedot] {};
\draw(1,-1.4) node (23) [nodedot] {};
\draw(-1.5,-2.1) node (31) [nodedot] {};
\draw(-0.5,-2.1) node (32) [nodedot] {};
\draw(0.25,-2.1) node (33) [nodedot] {};
\draw(0.75,-2.1) node (34) [nodedot] {};
\draw(1.25,-2.1) node (35) [nodedot] {};
\draw(1.75,-2.1) node (36) [nodedot] {};
\draw(-1.5,-2.8) node (41) [nodedot, fill = black, label={[label distance=-0.1cm]left:$u$}] {};
\draw(-0.5,-2.8) node (42) [nodedot] {};
\draw(0.25,-2.8) node (43) [nodedot] {};
\draw(0.75,-2.8) node (44) [nodedot] {};
\draw(1.25,-2.8) node (45) [nodedot, fill = black, label = {[label distance=-0.1cm]left:$v$}] {};
\draw(1.75,-2.8) node (46) [nodedot] {};
\draw(-2.25,-3.5) node(51) [nodedot] {};
\draw(-1.5,-3.5) node(52) [nodedot] {};
\draw(-0.75,-3.5) node(53) [nodedot] {};
\draw(0.75,-3.5) node(54) [nodedot] {};
\draw(1.75,-3.5) node(55) [nodedot] {};
\draw(-2.25,-4.2) node(61) [nodedot] {};
\draw(-1.5,-4.2) node(62) [nodedot] {};
\draw(-0.75,-4.2) node(63) [nodedot] {};
\draw(0.75,-4.2) node(64) [nodedot] {};
\draw(1.75,-4.2) node(65) [nodedot] {};
\draw(-2.75,-4.9) node(71) [nodedot] {};
\draw(-1.75,-4.9) node(72) [nodedot] {};
\draw(-1.25,-4.9) node(73) [nodedot] {};
\draw(-0.25,-4.9) node(74) [nodedot] {};
\draw(0.75,-4.9) node(75) [nodedot] {};
\draw(1.25,-4.9) node(76) [nodedot] {};
\draw(2.25,-4.9) node(77) [nodedot] {};
\draw(-2.75,-5.6) node(81) [nodedot] {};
\draw(-1.75,-5.6) node(82) [nodedot, fill = black] {};
\draw(-1.25,-5.6) node(83) [nodedot] {};
\draw(-0.25,-5.6) node(84) [nodedot, fill = black] {};
\draw(0.75,-5.6) node(85) [nodedot] {};
\draw(1.25,-5.6) node(86) [nodedot] {};
\draw(2.25,-5.6) node(87) [nodedot, fill = black] {};
\draw(-2.5,-6.3) node(91) [nodedot] {};
\draw(-1.75,-6.3) node(92) [nodedot] {};
\draw(-1,-6.3) node(93) [nodedot] {};
\draw(-0.25,-6.3) node(94) [nodedot] {};
\draw(1.5,-6.3) node(95) [nodedot] {};
\draw(2.25,-6.3) node(96) [nodedot] {};
\draw(3,-6.3) node(97) [nodedot] {};
\path (0,-6.65) -- (0,-7.35) node [black, midway, sloped] {$\dots$};
\draw(-3.75,-7.7) node(101) [nodedot] {};
\draw(-3,-7.7) node(102) [nodedot] {};
\draw(-2.25,-7.7) node(103) [nodedot] {};
\draw(-1.5,-7.7) node(104) [nodedot] {};
\draw(-0.75,-7.7) node(105) [nodedot] {};
\draw(0,-7.7) node(106) [nodedot] {};
\draw(0.75,-7.7) node(107) [nodedot] {};
\draw(1.5,-7.7) node(108) [nodedot] {};
\draw(2.25,-7.7) node(109) [nodedot] {};
\draw(3,-7.7) node(1010) [nodedot] {};
\draw(3.75,-7.7) node(1011) [nodedot] {};
\draw(-3.75,-8.4) node(111) [nodedot, fill=black] {};
\draw(-3,-8.4) node(112) [nodedot, fill=black] {};
\draw(-2.25,-8.4) node(113) [nodedot] {};
\draw(-1.5,-8.4) node(114) [nodedot, fill=black] {};
\draw(-0.75,-8.4) node(115) [nodedot] {};
\draw(0,-8.4) node(116) [nodedot] {};
\draw(0.75,-8.4) node(117) [nodedot, fill=black] {};
\draw(1.5,-8.4) node(118) [nodedot] {};
\draw(2.25,-8.4) node(119) [nodedot, fill=black] {};
\draw(3,-8.4) node(1110) [nodedot] {};
\draw(3.75,-8.4) node(1111) [nodedot, fill=black] {};

\node at (-6,0.7) {Depth};
\node at (-6,0) {$0$};
\node at (-6,-0.7) {$1$};
\node at (-6,-1.4) {$2$};
\path (-6,-1.75) -- (-6,-2.45) node [black, midway, sloped] {$\dots$};
\node at (-6,-2.8) {$2H$};
\draw (-7.5,-3.15) edge[dashed, line width = 0.5pt] (7.5,-3.15);
\node at (-6,-3.5) {$2H+1$};
\path (-6,-3.85) -- (-6,-5.25) node [black, midway, sloped] {$\dots$};
\node at (-6,-5.6) {$4H$};
\draw (-7.5,-5.95) edge[dashed, line width = 0.5pt] (7.5,-5.95);
\node at (-6,-6.3) {$4H+1$};
\path (-6,-6.65) -- (-6,-8.05) node [black, midway, sloped] {$\dots$};
\node at (-6,-8.4) {$2HL$};
\draw (6,0.7) node(sel) [nodedot, fill=black] {};
\node at (6,0) {$V_0=\left\{i_k\right\}$};
\node at (6,-2.8) {$V_1=\mathcal{L}_{i_k,\mathsf{sel}}=\left\{u,v\right\}$};
\node at (6,-5.6) {$V_2=\mathcal{L}_{u,\mathsf{sel}}\cup\mathcal{L}_{v,\mathsf{sel}}$};
\node at (6,-8.4) {$V_L=: L_k$};

  \draw[Bedge] (0)--(11) (0)--(12) (0)--(13) (21)--(31) (21)--(32) (23)--(33) (23)--(34) (23)--(35) (23)--(36) (41)--(51) (41)--(52) (41)--(53) (45)--(54) (45)--(55) (61)--(71) (61)--(72) (63)--(73) (63)--(74) (64)--(75) (65)--(76) (65)--(77) (82)--(91) (82)--(92) (82)--(93) (84)--(94) (87)--(95) (87)--(96) (87)--(97);
  \draw[Redge] (21)--(11) (22)--(12) (23)--(13) (41)--(31) (42)--(32) (43)--(33) (44)--(34) (45)--(35) (46)--(36) (51)--(61) (52)--(62) (53)--(63) (54)--(64) (55)--(65) (71)--(81) (72)--(82) (73)--(83) (74)--(84) (75)--(85) (76)--(86) (77)--(87) (101)--(111) (102)--(112) (103)--(113) (104)--(114) (105)--(115) (106)--(116) (107)--(117) (108)--(118) (109)--(119) (1010)--(1110) (1011)--(1111);
 \end{tikzpicture}

  \end{center}
\caption{Example of the left tree of $T_k$ rooted at $i_k$, with $H=2$. The shaded vertices represent the set of selected vertices in each epoch. The epochs are separated by the horizontal dashed lines. In the figure, after the first epoch, the set of selected vertices $V_1$ consists of two vertices $u,v$. Hence the exploration step in the second epoch is run with root vertices $u,v$. The set $L_k$ is defined as the set of selected vertices after the $L$'th epoch.}%
\label{fig:exploration}%
\end{figure}
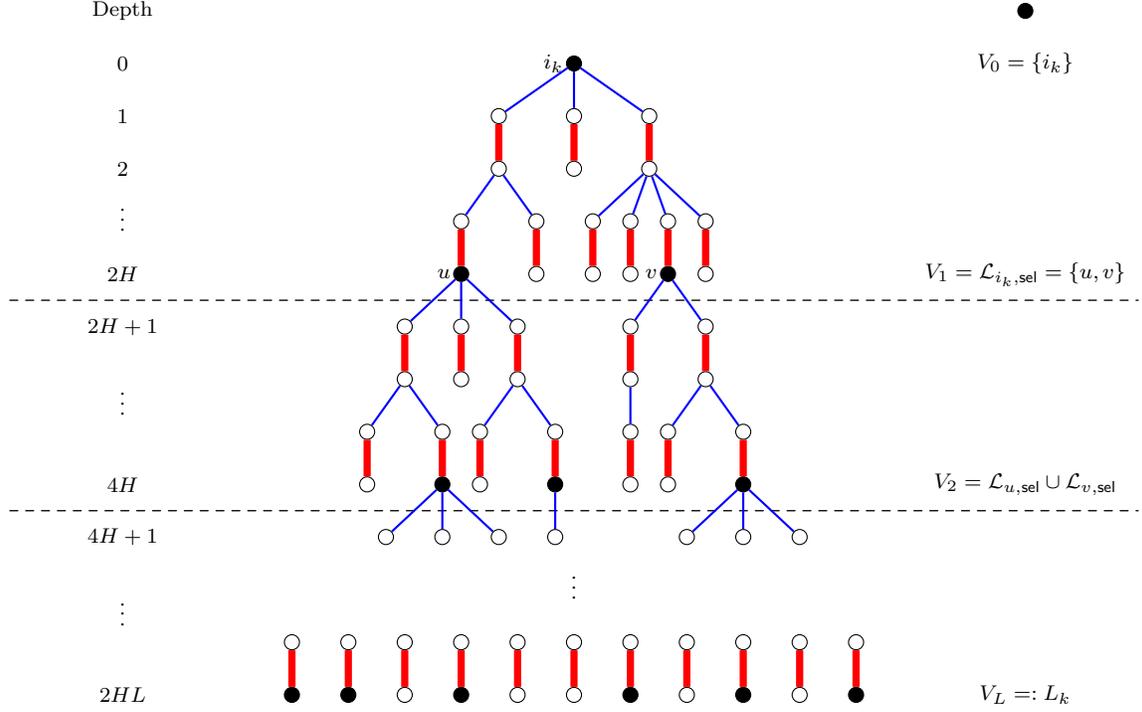

An illustration of the construction of the left tree of $T_k$ is given in Figure~\ref{fig:exploration}. 
At the end of each epoch (of $2H$ hops), we select those leaf nodes whose path to the previous root has a sufficiently large excess weight $\Delta$, and continue to grow the tree from these vertices.
As we will see in Section~\ref{sec:many.good.trees}, the parameters $H$ and $L$ needed to be chosen appropriately to ensure that existence of many two-sided trees that are not overgrown, yet contain sufficiently many leaf nodes. On one hand, we cannot perform the selection step too frequently, i.e., $H$ must be large enough. This is to guarantee that the large deviation analysis is tight, so that $\mathbb{P}\{\Delta(P) \ge \zeta H \}\approx e^{-\alpha H}$ for each 
alternating path of length $2H$ -- see \prettyref{eq:def.p}. 
On the other hand, we also need to select the leaf nodes sufficiently often, i.e. $L$ must be large enough. This is to ensure that the growth of each tree does not use up too many vertices, so that the exploration process yields many two-sided trees.


Finally, we note that Step 5 of Algorithm~\ref{alg:exploration} ensures that each two-sided tree uses up \emph{exactly} $2m$ pairs of vertices, so that the structure (i.e.~the isomorphism class) of the $K$ two-sided trees are independent. Moreover, from the definition
\begin{equation}
\label{eq:K.lower}
K= \frac{\gamma n}{2m}= \frac{\gamma}{2(1+\epsilon)^{2HL}\exp(3H\alpha+\epsilon H)}\cdot n,
\end{equation}
we have that after all trees are constructed, the total number of unexplored vertex pairs, i.e., the size of $\calU$, equals $(1-2\gamma) n$. 


\subsection{Existence of many two-sided trees}
\label{sec:many.good.trees}
For each two-sided tree $T_k$ constructed from Algorithm~\ref{alg:exploration}, define
\begin{align*}
L_k= &\text{ the set of selected leaf vertices at depth }2HL\text{ in the left subtree of }T_k ,\\
R_k= &\text{ the set of selected leaf vertices at depth }2HL\text{ in the right subtree of }T_k.
\end{align*}
Since the left tree and the right tree are rooted at $i_k\in V^c$ and $i_k'\in (V^c)'$ respectively, we have $L_k\subset V^c$ and $R_k\subset(V^c)'$. From the definition of the two-sided trees, each pair of vertices $i \in L_k$ and $j \in R_k$ are connected via an alternating path $P$ of length $4HL$ on $T_k$. In this subsection, we show that at least $K_1=\Omega(n)$ of the $K$ two-sided trees yield sufficiently large sets $L_k,R_k$, and large enough weight $\Delta(P)$.
\begin{theorem}\label{thm:exploration}
Suppose that the condition \prettyref{eq:impossibility} holds and Algorithm~\ref{alg:exploration} is run with parameters
\begin{equation}
\label{eq:params}
\gamma\leq \gamma_0 =\epsilon/32,\;\;\; \zeta  = \min\left\{\epsilon/32, D(\calP\|\calQ) + D(\calQ\|\calP) \right\},
\end{equation}
and  $H$ is a large constant depending only on $\epsilon, \calP, \calQ$. 
Then with probability at least $1-e^{-\Omega(n)}$, there exists $\mathcal{K}_1\subset \mathcal{K}$ with
\[
K_1=|\mathcal{K}_1|= \frac{\gamma}{16(1+\epsilon)^{2HL}\exp(7H\alpha+3\epsilon H/2)}\cdot n,
\]
such that for all $k\in \mathcal{K}_1$, $T_k$ satisfies
\begin{enumerate}
\item $|L_k|\geq (1+3\epsilon/4)^{2HL}$, $|R_k|\geq (1+3\epsilon/4)^{2HL}$;
\item for all $i \in L_k$ and $j \in R_k$, the path $P$ between $i$ and $j$ on $T_k$ has weight 
$
\Delta(\path)\geq 2\zeta HL-\tau_\mathsf{red},
$
where $\tau_\mathsf{red}=\inf\{x:\calP[\log(\calP/\calQ)\leq x]\geq 1/2\}$ is the median of $\log(\calP/\calQ)$ under $\calP$.
 \end{enumerate}
\end{theorem}

Before proving Theorem~\ref{thm:exploration}, we introduce the following auxiliary lemma. The lemma states that for each two-sided tree $T_k$, with sufficiently large probability, the termination condition~\eqref{eq:terminate} is not hit and there are sufficiently many leaf nodes at depth $2HL$ in both the left and the right subtrees. As a consequence, we get large sets $L_k$ and $R_k$ as desired in the statement of Theorem~\ref{thm:exploration}.

\begin{lemma}\label{lmm:good.trees}
Under the assumptions of Theorem~\ref{thm:exploration}, we have for each $k=1,...,K$, with probability at least $\tfrac{1}{2}\exp(-4\alpha H-\epsilon H/2)$, the following hold for both the left and the right subtree in the two-sided tree $T_k$.
\begin{enumerate}[(a)]
\item\label{enum:not.overgrown} ({\it not overgrown}) The construction of the tree is not terminated by hitting condition~\eqref{eq:terminate}.
\item\label{enum:many.leaves} ({\it many leaf nodes}) The number of selected leaf nodes at depth $2HL$ is at least
$
(1+3\epsilon/4)^{2HL}.
$
\end{enumerate}
\end{lemma}

With Lemma~\ref{lmm:good.trees}, we are ready to prove Theorem~\ref{thm:exploration}. The proof of~\prettyref{lmm:good.trees} is deferred to the end of this subsection.
\begin{proof}[Proof of Theorem~\ref{thm:exploration}]
For each $k$, each pair of $i \in L_k$ and $j \in R_k$, the path $\path$ between $i$ and $j$ on $T_k$ consists of the central edge $(i_k,i_k')$ and $2L$ subpaths, each of length $H$. By the leaf node selection step~\eqref{eq:pruning}, each subpath has weight at least $\zeta H$. Therefore,
\[
\Delta(\path)\geq 2\zeta HL-\log\frac{\calP}{\calQ} \left(W_{i_k,i_k'}\right).
\]
From the definition of $\tau_\mathsf{red}$,
\[
\mathbb{P}\left\{\log\frac{\calP}{\calQ}\left(W_{i_k,i_k'}\right)\leq \tau_\mathsf{red}\right\}\geq 1/2.
\]
Therefore, with probability at least $1/2$, the weight of all paths from $L_k$ to $R_k$ have weight at least $2\zeta HL-\tau_\mathsf{red}$.

Let $\widetilde{\mathcal{K}}$ be the set of all $k\leq K$ such that $|L_k|\geq (1+3\epsilon/4)^{2HL}, $ $|R_k|\geq (1+3\epsilon/4)^{2HL}, $ and $\Delta(\path)\geq 2\zeta HL-\tau_\mathsf{red}$ for all $\path$ from $L_k$ to $R_k$.
Since the construction of the trees are completely independent of the weight of the central edges, we have from property (\ref{enum:many.leaves}) in Lemma~\ref{lmm:good.trees},
\[
\mathbb{P}\left\{k\in \widetilde{\mathcal{K}}\right\}\geq \frac{1}{4}\exp(-4\alpha H-\epsilon H/2).
\]
Moreover, recall that we deliberately ensured that the construction of each tree uses up a deterministic number of vertices. As a result, the isomorphism class of each two-sided tree $T_k$, in particular its total number of vertices and the number of leaf nodes, 
are independent of the prior $T_1,...,T_{k-1}$. That, combined with the independence of the red edge weights, yields independence of the events $\{k\in \widetilde{\calK}\}_{k\leq K}$.
Thus
\begin{align*}
& \mathbb{P}\left\{\left|\widetilde{\mathcal{K}}\right|<\frac{1}{8}\exp(-4\alpha H-\epsilon H/2)K\right\}\\
\leq &\mathbb{P}\left\{\Binom\left(K,\frac{1}{4}\exp(-4\alpha H-\epsilon H/2)\right)<\frac{1}{8}\exp(-4\alpha H-\epsilon H/2)K\right\} =\exp(-\Omega(K))
\end{align*}
by Hoeffding's inequality. Combine with~\eqref{eq:K.lower} to conclude that with probability at least $1-e^{-\Omega(n)}$,
\[
\left|\widetilde{\mathcal{K}}\right|\geq \frac{1}{8}\exp(-4\alpha H-\epsilon H/2)K= \frac{\gamma}{16(1+\epsilon)^{2HL}\exp(7H\alpha +3\epsilon H/2)}\cdot n\triangleq c_5 n.
\]
To finish the proof of Theorem~\ref{thm:exploration}, let $\mathcal{K}_1$ be an arbitrary subset of $\widetilde{\mathcal{K}}$ of size exactly $c_5 n$.
\end{proof}

\begin{proof}[Proof of Lemma~\ref{lmm:good.trees}]
For a fixed $k$, let
\begin{align*}
\mathcal{A}_{L,k}\triangleq &\left\{\text{The left tree  of }T_k \text{ satisfies condition (a) ({\it not overgrown})}\right\},\\
\mathcal{B}_{L,k}\triangleq &\left\{\text{The left tree  of }T_k \text{ satisfies condition (b) ({\it many leaf nodes})}\right\}.
\end{align*}
Similarly define events $\mathcal{A}_{R,k}$, $\mathcal{B}_{R,k}$. To prove Lemma~\ref{lmm:good.trees}, it suffices to show that for each $k$,
\begin{equation}
\label{eq:conditions.ab}
\mathbb{P}\left(\mathcal{A}_{L,k} \cap \mathcal{B}_{L,k} \cap \mathcal{A}_{R,k} \cap \mathcal{B}_{R,k}\right) \geq \tfrac{1}{2}e^{-4\alpha H-\epsilon H/2}.
\end{equation}

Let us first focus on the analysis of the left tree. Recall that in Algorithm~\ref{alg:exploration}, in each step of the exploration process from vertex $v$, the set of offsprings of $v$ is defined as all members of $\mathcal{U}'$ that are incident to $v$ via a blue edge. Thus, the offspring distribution is $\Binom(|\mathcal{U}|,d/n)$. The randomness in $|\mathcal{U}|$ results in dependence between different steps of the exploration process, which is technically inconvenient. To resolve this problem, we utilize the fact that $|\mathcal{U}|$ is always between $(1-\gamma)n$ and $n$, and couple the tree with two two-sided trees $T_k^{(\mathsf{L})}=T_k^{(\mathsf{Lower})}$ and $T_k^{(\mathsf{U})}=T_k^{(\mathsf{Upper})}$, whose offspring distributions follow $\Binom((1-2\gamma)n,d/n)$ and $\Binom(n,d/n)$, respectively.

More specifically, in each epoch of the growth of the left tree of $T_k^{(\mathsf{L})}$, the exploration process follows a branching process with offspring distribution $\Binom((1-2\gamma)n,d/n)$, such that it is homomorphic to a subtree of $T_k$; the leaf node selection is done in a similar way that preserves the stochastic ordering, such that each vertex is selected with probability
\begin{equation}\label{eq:def.p}
p_\mathsf{sel}\triangleq \mathbb{P}\left\{\sum_{j\leq H}Y_j- \sum_{j\leq H}X_j\geq \zeta H\right\},
\end{equation}
where $X_i$'s are i.i.d.\ copies of $\log(\calP/\calQ)$ under distribution $\calP$ and $Y_i's$ are i.i.d.\ copies of $\log(\calP/\calQ)$ under distribution $\calQ$. Construct the right tree of $T_k^{(\mathsf{L})}$, and $T_k^{(\mathsf{U})}$ in similar fashion. One caveat is that the tree $T_k$ may hit the termination condition~\eqref{eq:terminate}, while  both $T_k^{(\mathsf{L})}$ and $T_k^{(\mathsf{L})}$ are assumed to be ``free-growing" without any termination conditions. As a result, on the event that neither subtree of $T_k$ is 
terminated by hitting the termination condition~\eqref{eq:terminate}, we have
\[
T_k^{(\mathsf{L})} \subset T_k \subset T_k^{(\mathsf{U})},
\]
where $\subset$ denotes stochastic dominance of the three trees. Let $\mathcal{A}_{L,k}^{(\mathsf{L})}, \mathcal{A}_{R,k}^{(\mathsf{L})}, \mathcal{B}_{L,k}^{(\mathsf{L})}, \mathcal{B}_{R,k}^{(\mathsf{L})}$ denote the analogue events on the tree $T_k^{(\mathsf{L})}$, and similarly define the events on $T_k^{(\mathsf{U})}$. A key observation is that
\[
\mathcal{A}_{L,k}^{\mathsf{(U)}}\cap \mathcal{B}_{L,k}^{(\mathsf{L})}\subset \mathcal{A}_{L,k} \cap \mathcal{B}_{L,k}.
\]
To see why, note that $T_k \subset T_k^{(\mathsf{U})}$ always holds. Therefore on $\mathcal{A}_{L,k}^{\mathsf{(U)}}$, we have that the left subtree of $T_k$ never hits the termination condition~\eqref{eq:terminate}. As a result, $T_k^{(\mathsf{L})} \subset T_k$ also holds, and thus $\mathcal{B}_{L,k}^{(\mathsf{L})}$ implies $\mathcal{B}_{L,k}$. A similar relationship holds for the right tree. It follows that
\begin{align*}
\mathbb{P}\left(\mathcal{A}_{L,k} \cap \mathcal{B}_{L,k} \cap \mathcal{A}_{R,k} \cap \mathcal{B}_{R,k}\right)
\geq & \mathbb{P}\left(\mathcal{A}_{L,k}^{(U)} \cap \mathcal{B}_{L,k}^{(L)} \cap \mathcal{A}_{R,k}^{(U)} \cap \mathcal{B}_{R,k}^{(L)}\right)\\
\geq & \left[\mathbb{P}\left(\mathcal{A}_{L,k}^{(U)}\right) + \mathbb{P}\left(\mathcal{B}_{L,k}^{(L)}\right)-1\right]^2.
\end{align*}
Thus, to prove~\eqref{eq:conditions.ab} it suffices to show the following two claims:
\begin{equation}
\label{eq:condition.a}
\mathbb{P}\left(\mathcal{A}_{L,k}^{(\mathsf{U})}\right) \geq 1-\tfrac{1}{4}e^{-2\alpha H-\epsilon H/4},
\end{equation}
\begin{equation}
\label{eq:condition.b}
\mathbb{P}\left(\mathcal{B}_{L,k}^{(\mathsf{L})}\right) \geq e^{-2\alpha H-\epsilon H/4}.
\end{equation}
We start with the proof of~\eqref{eq:condition.a}. First, we can control the probability $p_\mathsf{sel}$ defined in~\eqref{eq:def.p} from above via the large deviation bound~\eqref{eq:LDXY}:
\begin{equation}
p_\mathsf{sel}\leq \mathbb{P}\left\{\sum_{j\leq H}Y_j- \sum_{j\leq H}X_j\geq 0\right\} \leq \exp\left(-\alpha H\right).
\end{equation}

Let $W_i^{(\mathsf{U})}$ denote the number of vertices that are used in the exploration process at the $i$'th epoch (before selection) in the left subtree of $T_k^{(\mathsf{U})}$. Let $Z_{i}^{(\mathsf{U})}=|V_i^{(\mathsf{U})}|$ denote the number of leaf nodes at depth $iH$ that are selected. We have for all $i\leq L$,
\begin{align*}
\mathbb{E}(W_i^{(\mathsf{U})}) = & \mathbb{E}(Z_{i-1}^{(\mathsf{U})})\cdot \sum_{j=1}^H d^j\\
= & \left[\left(n\cdot\frac{d}{n}\right)^{H}p_\mathsf{sel}\right]^{i-1}\cdot \sum_{j=1}^H d^j\\
\leq & \left[d e^{-\alpha }\right]^{H(i-1)}\cdot \sum_{j=1}^H d^j\\
\leq & \left[d e^{-\alpha }\right]^{H(i-1)}d^H \frac{d}{d-1}.
\end{align*}
Recall that $de^{-\alpha}=(1+\epsilon)^2$. 
Therefore, the expected total number of vertices in the left tree
\begin{align}
\nonumber\mathbb{E}\left(1+\sum_{i=1}^L W_i^{(\mathsf{U})}\right)\leq & 1+\sum_{i=1}^L(1+\epsilon)^{2H(i-1)}d^H \frac{d}{d-1}\\
\nonumber \stackrel{(a)}{\leq} & \sum_{i=0}^L(1+\epsilon)^{2H(i-1)}d^H \frac{d}{d-1}\\
\nonumber \stackrel{(b)}{\leq} & \frac{d}{d-1}d^H(1+\epsilon)^{2H(L-1)}\frac{1}{1-(1+\epsilon)^{-H}}\\
\label{eq:explored.upper}\stackrel{(c)}{\leq} & \frac{2d}{d-1}(1+\epsilon)^{2HL}e^{H\alpha },
\end{align}
where (a) is from $(1+\epsilon)^{-2H}d^H d/(d-1)= e^{\alpha H}d/(d-1)\geq 1$; (b) is by bounding the finite geometric sum by an infinite sum; (c) is by choosing $H$ to be a sufficiently large constant so that $(1+\epsilon)^{-H} \le 1/2$.
Therefore,
\begin{align*}
1-\mathbb{P}\left(\mathcal{A}_{L,k}^{(\mathsf{U})}\right)
\stackrel{(a)}{\leq} & \mathbb{P}\left\{1+\sum_{i=1}^L W_i^{(\mathsf{U})} \geq (1+\epsilon)^{2HL}\exp(3H\alpha +\epsilon H)\right\}\\
\stackrel{(b)}{\leq} & \frac{\mathbb{E}\left(1+\sum_{i=1}^L W_i^{(\mathsf{U})}\right)}{(1+\epsilon)^{2HL}\exp(3H\alpha +\epsilon H)}\\
\stackrel{(c)}{\leq} & \frac{2d}{d-1}\exp(-2H\alpha -\epsilon H)\\
\stackrel{(d)}{\leq} & \frac{1}{4}\exp(-2H\alpha -\epsilon H/4),
\end{align*}
where (a) is from the definition of $\mathcal{A}_{L,k}^{(\mathsf{U})}$; (b) is from Markov's inequality; (c) is from~\eqref{eq:explored.upper}; (d) is because by the assumption $de^{-\alpha }=(1+\epsilon)^2$, and that $e^{-\alpha }=B(\calP, \calQ)^2\leq 1$, we have $d\geq (1+\epsilon)^2$. Thus $2d/(d-1)\leq 2(1+\epsilon)^2/(2\epsilon+\epsilon^2)\leq \exp(3\epsilon H/4)/4$ by choosing $H$ large enough.


Next we show~\eqref{eq:condition.b}. As before, let $Z_{i}^{(\mathsf{L})}=|V_i^{(\mathsf{L})}|$ denote the number of leaf nodes at depth $iH$ that are selected in the left subtree of $T_k^{(\mathsf{L})}$. We have
\begin{equation}
\label{eq:couple.b}
\mathbb{P}\left(\mathcal{B}_{L,k}^{(\mathsf{L})}\right) = \mathbb{P}\left\{Z_{L}^{(\mathsf{L})}\geq (1+3\epsilon/4)^{2HL}\right\}.
\end{equation}
We will bound $\mathbb{P}\{\calB_{L,k}^{(\mathsf{L})}\}$ by analyzing the first and second moments of $Z_L^{(\mathsf{L})}$. Note that for all $i$,
\[
\mathbb{E}\left(Z_{i}^{(\mathsf{L})}\right) = \left[\left((1-2\gamma)d\right)^{H}p_\mathsf{sel}\right]^i \triangleq \mu^i.
\]
We claim the following inequalities on $p,\mu$, and the second moment of $Z_L^{(\mathsf{L})}$.
\begin{equation}
\label{eq:p.lower}
p_\mathsf{sel} \geq \exp\left(- \alpha H - \epsilon H/16\right),
\end{equation}
\begin{equation}\label{eq:mu.lower}
\mu\geq 2\left(1+\tfrac{3}{4}\epsilon\right)^{2H},
\end{equation}
\begin{equation}
\label{eq:2nd.moment}
\mathbb{E}\left(\left[Z_{L}^{(\mathsf{L})}\right]^2\right)\leq \mu^{2L}+\left[(1-2\gamma)d\right]^{2H}\frac{(1-2\gamma)d}{(1-2\gamma)d-1}\cdot\frac{\mu^{2L-2}}{1-\mu^{-1}}.
\end{equation}
Assuming that~\eqref{eq:p.lower},~\eqref{eq:mu.lower},~\eqref{eq:2nd.moment} all hold, we first finish the proof of~\eqref{eq:condition.b}. 
By combining~\eqref{eq:couple.b} and~\eqref{eq:mu.lower}, we have
\begin{equation}\label{eq:half.mean}
\mathbb{P}\left\{\mathcal{B}_{L,k}^{(\mathsf{L})}\right\}
\geq \mathbb{P}\left\{Z_{L}^{(\mathsf{L})}>\tfrac{1}{2}\mathbb{E}(Z_{L}^{(\mathsf{L})})\right\}.
\end{equation}
By the Paley-Zygmund inequality,
\begin{equation}
\label{eq:PZ}
\mathbb{P}\left\{Z_{L}^{(\mathsf{L})}>\tfrac{1}{2}\mathbb{E}(Z_{L}^{(\mathsf{L})})\right\}
\geq 
\frac{1}{4}\frac{\left[\mathbb{E}(Z_{L}^{(\mathsf{L})})\right]^2}{\mathbb{E}\left(\left[Z_{L}^{(\mathsf{L})}\right]^2\right)}
\geq \frac{1}{4}\frac{1}{1+\left[(1-2\gamma)d\right]^{2H}\frac{(1-2\gamma)d}{(1-2\gamma)d-1}\frac{1}{\mu(\mu-1)}},
\end{equation}
where the last inequality follows from \eqref{eq:2nd.moment}. 
From~\eqref{eq:mu.lower}, we can choose $H$ large enough such that $\mu-1\geq\mu/2$. Therefore, the right-hand side of~\eqref{eq:PZ} is bounded from below by
\begin{align*}
& \frac{1}{4}\frac{1}{1+ 2\frac{(1-2\gamma)d}{(1-2\gamma)d-1}\left[(1-2\gamma)d\right]^{2H}\mu^{-2}}\\
\stackrel{(a)}{\geq} & \frac{1}{4\left(1+ 2\frac{(1-2\gamma)d}{(1-2\gamma)d-1}\right)} p^2\\
\stackrel{(b)}{\geq} & \frac{1}{4\left(1+ 2\frac{(1-\epsilon/16)(1+\epsilon)^2}{(1-\epsilon/16)(1+\epsilon)^2-1}\right)}\exp\left(-2\alpha H-\epsilon H/8\right)\\
\stackrel{(c)}{\geq} & \exp\left(-2\alpha H-\epsilon H/4\right),
\end{align*}
where (a) is from the definition of $\mu$, and $p_\mathsf{sel}\leq 1$; (b) is from~\eqref{eq:p.lower}, and $\gamma\leq \epsilon/32$, $d\geq (1+\epsilon)^2$ (which follows from \prettyref{eq:impossibility}); (c) is by choosing $H$ large enough so that the fractional factor is absorbed into $\exp(\epsilon H/8)$. Combine the display above with~\eqref{eq:half.mean} and~\eqref{eq:PZ} to finish the proof of~\eqref{eq:condition.b}.

It remains to prove \eqref{eq:p.lower}, \eqref{eq:mu.lower} and \eqref{eq:2nd.moment}. 

{\it Proof of~\eqref{eq:p.lower}:} 
By~\eqref{eq:LD.reverse}, we have that for $\zeta$ chosen as in~\eqref{eq:params},
\[
p_\mathsf{sel} \geq \exp\left(-\alpha H - \zeta H+ o(H) \right).
\]
\eqref{eq:p.lower} follows by choosing $H$ to be a large enough constant only depending on $\epsilon, \calP, \calQ$. 

{\it Proof of~\eqref{eq:mu.lower}:} By definition of $\mu$ and the inequality~\eqref{eq:p.lower}, we have
\begin{align*}
\mu\geq &\left[(1-2\gamma)d\right]^H\exp\left(-H\alpha -\epsilon H/16\right)\\
\geq &\left[(1-\epsilon/16)(1+\epsilon)^2\right]^H\exp\left(-\epsilon H/16\right)\geq 2\left(1+\tfrac{3}{4}\epsilon\right)^{2H}
\end{align*}
by choosing $H$ to be a large enough constant. 

{\it Proof of~\eqref{eq:2nd.moment}:} 
Following the arguments in~\cite[Theorem~2.1.6]{durrett2007random}, we can control the second moment of $Z_{i}^{(\mathsf{L})}$ with
\begin{equation}
\label{eq:variance.outer}
\mathbb{E}\left(\left[Z_{i}^{(\mathsf{L})}\right]^2\right)\leq \mu^{2i}+\sigma^2\frac{\mu^{2i-2}}{1-\mu^{-1}},
\end{equation}
where $\sigma^2=\text{Var}(Z_{1}^{(\mathsf{L})})\leq \mathbb{E}([Z_{1}^{(\mathsf{L})}]^2)$. For completeness we include a short proof of~\eqref{eq:variance.outer} here. 
First, note that
\[
\mathbb{E}\left(Z_{i}^{(\mathsf{L})} \mid Z_{i-1}^{(\mathsf{L})}\right)=\mu Z_{i-1}^{(\mathsf{L})},\;\;\;\text{and}
\]
\[
\text{Var}\left(Z_{i}^{(\mathsf{L})} \mid Z_{i-1}^{(\mathsf{L})}\right)=\sum_{j\leq Z_{i-1}^{(\mathsf{L})}}\sigma^2 = Z_{i-1}^{(\mathsf{L})} \sigma^2.
\]
Therefore
\begin{align*}
\mathbb{E}\left(\left[Z_{i}^{(\mathsf{L})}\right]^2\mid Z_{i-1}^{(\mathsf{L})}\right) = &\mu^2\left(Z_{i-1}^{(\mathsf{L})}\right)^2+ \text{Var}\left(Z_{i}^{(\mathsf{L})} \mid Z_{i-1}^{(\mathsf{L})}\right)
=  \mu^2\left(Z_{i-1}^{(\mathsf{L})}\right)^2 + Z_{i-1}^{(\mathsf{L})} \sigma^2.
\end{align*}
Take expectation on both sides. We have
\[
\mathbb{E}\left(\left[Z_{i}^{(\mathsf{L})}\right]^2\right) = \mu^2\mathbb{E}\left(\left[Z_{i-1}^{(\mathsf{L})}\right]^2\right)+\mu^{i-1}\sigma^2.
\]
The induction above with $Z_{0}^{(\mathsf{L})}=1$ yields
\[
\mathbb{E}\left(\left[Z_{i}^{(\mathsf{L})}\right]^2\right) = \mu^{2i}+\sigma^2\sum_{j=i-1}^{2i-2} \mu^{j} \leq \mu^{2i}+\sigma^2\frac{\mu^{2i-2}}{1-\mu^{-1}}.
\]
To bound $\sigma^2$, we upper bound $Z_{1}^{(\mathsf{L})}$ by the number of leaves at depth $H$ before selection. Following the same argument as in the derivation of~\eqref{eq:variance.outer}, we have
\begin{equation}
\label{eq:variance.inner}
\sigma^2\leq \left[(1-2\gamma)d\right]^{2H}+\left[(1-2\gamma)d\right]^{2H-1}\frac{1}{1-[(1-2\gamma)d]^{-1}} = \left[(1-2\gamma)d\right]^{2H}\frac{(1-2\gamma)d}{(1-2\gamma)d-1}.
\end{equation}
Combine~\eqref{eq:variance.outer} and~\eqref{eq:variance.inner} to yield~\eqref{eq:2nd.moment}.
\end{proof}

\section{Sprinkling stage}\label{sec:sprinkling}


Recall that we have constructed a family of disjoint sets $L_k $ of left vertices 
and $R_k $ of right vertices for $k \in \calK_1 $ 
such that each pair of vertices $i \in L_k$ and $j \in R_k$ are connected via an alternating
path $\path$ of length $\ell$ through red edge $(k,k')$ in $G_1$, 
where $|L_k|\geq s,|R_k| \geq s$ and $K_1 \triangleq |\calK_1| $.
Crucially the construction of $L_k$ and $R_k$ does not involve the vertices in $V$. 
Recall from Algorithm~\ref{alg:cycle_finding} that $G_2$ is an unweighted subgraph of $G$ that does not contain any edges in $V^c\times (V^c)'$. A blue edge $e$ appears in $G_2$ if and only if it appears in $G$, and $\log(\calP/\calQ)(W_e)\geq \tau_\mathsf{blue}$. Therefore, the blue edges in $G_2$ are independently generated with probability $\eta/n$, where $\eta=d\cdot \calQ[\log(\calP/\calQ)(W_e)\geq \tau_\mathsf{blue}]$.

Recall from \prettyref{alg:cycle_finding} 
that 
\[
V^*=\{i\in V:\log(\calP/\calQ)(W_{i,i'})\leq \tau_\mathsf{red}\}
\]
is a subset of the reserved vertices whose incident red edge weight is below threshold.
In this section, we connect the alternating paths 
between $L_k$ and $R_k$ through vertices in $V^*$ and $(V^*)'$ to form long alternating cycles in $G$. This scheme is referred to as ``sprinkling'', which is detailed in \prettyref{alg:sprinkling}.



\begin{algorithm}[H]
\caption{Sprinkling} \label{alg:sprinkling}
\begin{algorithmic}[1]
\State{\bfseries Input:} Parameters $n,\eta,s$, sets $V^*\subset V\subset [n]$, a bipartite graph $G_2$ whose blue edges are independently generated with probability $\eta/n$, and disjoints sets $\{L_k\}_{k\in \calK_1}$ of left vertices and $\{R_k\}_{k\in \calK_1}$ of right vertices with $|L_k|\geq s$ and $|R_k|\geq s$ for all $k\in \calK_1$.

\State Define
\begin{align*}
A_k' = & \left\{v'\in (V^*)': \exists u\in L_k, \text{ s.t. } (u,v')\in E(G_2)\right\},\\
B_k = & \left\{v\in V^*: \exists u'\in R_k, \text{ s.t. }  (u',v)\in E(G_2)\right\}.
\end{align*}

\State For each $v\in V^*$, let
\[
d_v = \sum_{k \in \calK_1} \indc{  v'\in A_k' } + \sum_{k \in \calK_1} \indc{ v\in B_k }.
\]
Define
$
A_{\mathsf{overlap}}= \{v\in V^*: d_v\geq 2\}.
$
For all $k\in\mathcal{K}_1$, define
\[
U_k' = A_k'\backslash A_{\mathsf{overlap}}',\;\;\; V_k = B_k\backslash A_{\mathsf{overlap}}.
\]
Define $\beta=|V^*|/n$, and let
\[
\calK_2=\left\{k\in \calK_1: |U_k'|\geq \frac{\sr\sz\sp}{4}, |V_k|\geq\frac{\sr\sz\sp}{4}\right\}.
\]

\State Define a bipartite graph $G_{\rm super}$ on $\mathcal{K}_2 \times ( \mathcal{K}_2)'$, 
where there is a red edge between $i$ and $i'$ for every $i \in \mathcal{K}_2$
and there is a blue edge between $i$ and $j'$ if and only if  $U_i$ and $V_j'$ are connected via a blue edge in $G_2$. 

\State For each alternating cycle $C_{\rm super}= (i_1, i_1', i_2, i'_2, \cdots, i_{r}, i'_r)$ in $G_{\rm super}$, extend it to an alternating cycle in $G$ as
$$
C=\left(  v_1' , v_1, P_1,  u'_{1}, u_{1},  v_2', v_2, P_2, u_2', u_2,  \cdots,  v_r', v_r, P_r, u_r', u_r  \right),
$$
where $u_k$ and $v_{k+1}'$ are two nodes in $U_{i_k}$ and $V_{i_{k+1} }'$ that are connected by a blue edge;
$P_k$ is an alternating path connecting $L_{i_k}$ and $R_{i_k}$;
and $u'_k$ (resp.\ $v_k$) and $P_k$ are connected by a blue edge according to the definition of $U_{i_k}'$ (resp.\ $V_{i_k}$). 
\end{algorithmic}
\end{algorithm}

See Figure~\ref{fig:sprinkling} for an illustration of the final step of Algorithm~\ref{alg:sprinkling}, which extends each alternating cycle on the super graph into an alternating cycle on $G$. The following theorem provides a sufficient condition under which with high probability, $G_{\rm super}$ contains exponentially many distinct alternating cycles, which correspond to distinct alternating cycles on $G$. Note that the theorem also states that $K_2=|\calK_2|$ is of order $\Omega(K_1)$. Thus, the sprinkling scheme yields $e^{\Omega(K_2)}=e^{\Omega(K_1)}$ distinct alternating cycles, each of length $3K_2/4=\Omega(K_1)$. 

\begin{theorem}[Sprinkling]\label{thm:sprinkling} 
Let 
$$ 
\kappa = \frac{2K_1 \sz \sp }{n},  \quad b= \frac{\sr \sz \sp}{4}, \quad d_{\rm super}=  \frac{1}{32n} K_1 b^2 \sp.
$$
Suppose $ b \ge 4$, $K_1 \ge 8400$, $\kappa \le 1/16^2$, and $d_{\rm super} \ge 256 \log(32 e)$. 
Then conditional on a fixed $V^*$, we have with probability at least $\left( 1-e^{-K_1/32}-\frac{2}{\sr n\kappa^3}\right) \left( 1- e^{-d_{\rm super} K_1/2^{18}}\right)$, $G_{\rm super}$ 
obtained by \prettyref{alg:sprinkling}
contains at least $\exp( K_2/20)$ distinct alternating cycles of length at least $\frac{3}{4} K_2$, and $K_2\ge K_1/16$.
\end{theorem}


The proof of \prettyref{thm:sprinkling} consists of two major steps. 
First, we show that with high probability, $\left| U_k' \right| \ge b$ and $\left|V_k \right| \ge b$ for all $k \in \calK_2$ for some $\calK_2$ of size $K_2 \ge K_1/16$. 
Then conditional on $\{U_k, V'_k \}_{k \in \calK_2}$, we show that $G_{\rm super}$ contains exponentially many long alternating cycles with high probability.
In the sequel, we present the detailed proof. 

\subsection{Existence of large $\calK_2$}\label{sec:large.K2}
To show that there exist many indices $k\in \calK_1$ with $\left| U_k' \right| \ge b$ and $\left|V_k \right| \ge b$, we first show in the following Lemma~\ref{lmm:discard_small} that there exist many indices with large $|A_k|$ and $|B_k|$. We then show in Lemma~\ref{lmm:remove_overlap} that there are not many indices who lose a large fraction of vertices when the overlapping vertices are removed. 
\begin{lemma}\label{lmm:discard_small}
If $\sr \sz \sp \ge 4$ and $\sz \sp \le n /4$, then
$$
\mathbb{P}\left\{\sum_{k\in\mathcal{K}_1}\indc{\left|A_k'\right|\geq \frac{\sr \sz \sp} {2},\left|B_k\right|\geq \frac{\sr \sz  \sp} {2}  } \le \frac{K_1}{8}\right\}\leq \exp \left(-K_1/32 \right)
$$
\end{lemma} 

\begin{lemma}\label{lmm:remove_overlap}
If $ \kappa \triangleq 2K_1 \sz \sp/n  \le \frac{1}{16^2}$. Then
$$
\prob{ \sum_{k\in \mathcal{K}_1}\indc{\left|A_k'\right| - \left|U_k'\right|\geq \frac{\sr \sz \sp} {4},\;\text{or }\left|B_k\right| - \left|V_k\right| \geq \frac{\sr \sz \sp} {4} } \le \frac{K_1}{16}  }
\ge 1-\frac{2}{\sr n  \kappa^3 }.
$$
\end{lemma} 

Combining Lemma~\ref{lmm:discard_small} and~\ref{lmm:remove_overlap}, we conclude that there exists some $\mathcal{K}_2$ of size $K_2 \ge K_1/16$, such that for all $k\in \calK_2$, $|U_k'|\geq \frac{\sr \sz \sp} {2}-\frac{\sr \sz \sp} {4}=\frac{\sr \sz \sp} {4}=b$, and $|V_k|\geq b$.

\begin{proof}[Proof of Lemma~\ref{lmm:discard_small}]
For each $k\in \mathcal{K}_1$ and each $v'\in (V^*)'$, 
$$
\prob{ \exists u \in L_k, \text{ s.t. } (u, v') \in \calE(G_2) } 
= 1-  \left( 1 - \frac{\sp}{n} \right)^{|L_k|} \geq 1-  \left( 1 - \frac{\sp}{n} \right)^{\sz}.
$$
Therefore, $|A_k'|$ is stochastically dominant over a random variable distributed $\Binom( \sr n, 1-  ( 1 - \frac{\sp}{n} )^{\sz } )$. Since the mean and median of a binomial distribution differ by at most $1$, the median of $|A_k'|$ is lower bounded by
\[
\sr  n \left[ 1-  \left( 1 - \frac{\sp}{n} \right)^{\sz} \right] -1 
\overset{(a)}{\geq} \frac{3 \sr \sz \sp  }{4} -1 \geq \frac{\sr \sz \sp} {2},
\]
where $(a)$ holds due to $(1-x)^\sz \le  e^{-\sz x} \le 1-3\sz x/4$ when $0 \le \sz x \le 1/4$ 
and the assumption that $\sz \sp \le n/4$; the last inequality follows from the assumption 
that $\sr \sz \sp \ge 4$. 
Thus $\prob{ |A_k'|\geq \frac{\sr \sz \sp} {2} } \geq 1/2$. 
Similarly argue that $\mathbb{P}\{|B_k|\geq \frac{\sr \sz \sp} {2} \}\geq 1/2$. 
By independence of $|A_k'|$ and  $|B_k|$, we have
\[
\mathbb{P}\left\{\left|A_k'\right|\geq \frac{\sr \sz \sp } {2},\left|B_k\right|\geq\frac{\sr \sz \sp} {2} \right\} \geq \frac{1}{4}.
\]
Combined with independence across all $k\in \mathcal{K}_1$, we have
\begin{align*}
&\mathbb{P}\left\{\sum_{k\in\mathcal{K}_1}\indc{\left|A_k'\right|\geq \frac{\sr \sz \sp} {2},\left|B_k\right|\geq \frac{\sr \sz  \sp} {2}  } \le \frac{K_1}{8}\right\}\\
\leq & \mathbb{P}\left\{ \Binom\left(K_1,1/4 \right) \le K_1/8\right\} \leq \exp(-K_1/32)
\end{align*}
by Hoeffding's inequality. 
\end{proof}

\begin{proof} [Proof of Lemma~\ref{lmm:remove_overlap}]
Recall that for each $v\in V^*$, 
\[
d_v = \sum_{k\leq K_1} \indc{ v'\in A_k' } + \sum_{k\leq K_1}\indc{ v\in B_k }.
\]
Since $v\in A_\mathsf{overlap}$ if and only if $d_v\geq 2$, we have 
\begin{align*}
 \sum_{v\in V^*}d_v\indc{d_v\geq 2}
= & \sum_{v\in A_\mathsf{overlap}} d_v\\
= & \sum_{v\in A_\mathsf{overlap}} \sum_{k \in \calK_1}\left(\indc{v'\in A_k' } + \indc{ v\in B_k } \right)\\
= & \sum_{k\in \calK_1}\left(\sum_{v\in A_\mathsf{overlap}} \indc{ v'\in A_k' } + \sum_{v\in A_\mathsf{overlap}} \indc{ v\in B_k } \right)\\
= & \sum_{k \in  \calK_1}\left(\left|A_k'\right| - \left|U_k'\right| + \left|B_k\right| - \left|V_k\right|\right) \\
 \geq & \frac{\sr \sz \sp} {4 } \sum_{k\in \calK_1}\indc{ \left|A_k'\right| - \left|U_k'\right| + \left|B_k\right| - \left|V_k\right| \geq \frac{\sr \sz \sp} {4} }\\
\geq & \frac{\sr \sz \sp} {4} \sum_{k\in \calK_1}\indc{ \left|A_k'\right| - \left|U_k'\right|\geq \frac{\sr \sz \sp} {4},\;\text{or }\left|B_k\right| - \left|V_k\right| \geq \frac{\sr \sz \sp} {4} }.
\end{align*}
In other words, $\sum_{v\in V^*}d_v\indc{d_v\geq 2}$ controls the number of $k \in \calK_1$ 
for which $A_k'$ or $B_k$ loses over $\frac{\sr \sz \sp} {4}$ vertices when $A_{\mathsf{overlap}}$ is removed. 
It remains to prove that with high probability 
$$
\sum_{v\in V^*}d_v\indc{d_v\geq 2 } \le 
 \frac{\sr \sz \sp} {4} \times \frac{K_1}{16} .
$$
First, note that the random variables $\{d_v\}_{v\in V^*}$ are independent since the blue edges in $G_2$ are independent. 
Moreover, for each $v\in V^*$,
\begin{align*}
d_v=&\sum_{k \in K_1}\indc{ \exists u\in L_k, s.t. (u,v')\in E(G_2) } + \sum_{k\in K_1}\indc{\exists u'\in R_k, s.t. (v',u)\in E(G_2) }\\
\sim & \Binom\left(2K_1,  1- \left(  1 -\frac{\sp}{n} \right)^\sz \right).
\end{align*}
Using $(1-x)^\sz \ge 1-\sz x$, we get that $1- \left(  1 -\frac{\sp}{n} \right)^\sz  \le \sz \sp  /n$.
Hence, there exist random variables $\widetilde{d}_v\stackrel{i.i.d.}{\sim} \Binom(2K_1, \sz \sp /n )$ such that
\begin{equation}
\label{eq:dv.tilde}
\sum_{v\in V^*}d_v  \indc{d_v\geq 2 }\leq \sum_{v\in V^*}\widetilde{d}_v\indc{\widetilde{d}_v\geq 2 }.
\end{equation}
Note that 
\begin{align*}
\mathbb{E}\left[\widetilde{d}_v\indc{ \widetilde{d}_v\geq 2 } \right] = & \mathbb{E}[\widetilde{d}_v]-\mathbb{P}\left\{\widetilde{d}_v=1\right\}\\
= & 2K_1\frac{\sz \sp}{n}\left[1-\left(1- \frac{\sz \sp }{n} \right)^{2K_1-1}\right] \le  \left(  2K_1 \sz \sp /n \ \right)^2 \triangleq \kappa^2.
\end{align*}
where the last inequality follows from using $ (1-x)^a \ge 1-ax$ for $0 \le x \le 1$ and $a \ge 1$.
Moreover,
 \begin{align*}
 \text{Var}\left[\widetilde{d}_v \indc{\widetilde{d}_v\geq 2 }\right]& \leq \expect{\widetilde{d}_v^2}\leq  
\kappa + \kappa^2 \le 2 \kappa,
\end{align*}
where the last inequality follows from $\kappa \le 1$.
Thus by Chebyshev's inequality and  \prettyref{eq:dv.tilde},
we have that 
with probability 
at least $1-2/ (\sr n \kappa^3) $, 
\[
\sum_{v\in V^*} d_v \indc{d_v\geq 2 }
\le \sum_{v\in V^*}\widetilde{d}_v \indc{\widetilde{d}_v\geq 2 } 
\le 2 \sr n  \kappa^2 =  4 K_1 \sr \sz \sp \kappa \le \frac{K_1 \sr \sz \sp}{4 \times 16},
\]
where the last inequality holds  by the assumption that $\kappa \le \frac{1}{16^2}$.
%
\end{proof}

\subsection{Construction of exponentially many alternating cycles}

In this subsection we show the existence of many long alternating cycles in $G_{\rm super}$. Recall that in $G_{\rm super}$ there is a planted (red) edge between $i$ and $i'$ for every $i \in \calK_2$.
Moreover, there is an unplanted (blue) edge between $i$ and $j'$ for $i \neq j$ if and only if  $B_i$ and $\tilde{B}'_j$ are connected via an edge in $G_2$. 
In the previous steps, we have not inspected any of the edges between $U_i$ and $V_j'$. Therefore
\begin{align}\label{eq:Gsuper.degree}
\mathbb{P}\left\{(i,j') \in E(G_{\rm super})\right\} 
=  1- \left( 1 - \frac{\sp}{n} \right)^{ \left| U_i \right| \cdot \left|V'_j \right|} 
\ge  \frac{1}{2}  \frac{\sp}{n} b^2
\ge  \frac{d_{\rm super}}{K_2},
\end{align}
where the first inequality holds by $(1-x)^a \le 1- ax/2$ for $ 0 \le x \le 1/a$, and the assumptions that 
$ \left| U_i \right| ,  \left|V'_j  \right| \ge b$ and 
$b^2 \sp \le n$; 
the last equality holds by the definition of $d_{\rm super} =\frac{K_1 \eta b^2}{32n}$, and $K_2\geq K_1/16$ as shown in Section~\ref{sec:large.K2}. 
Moreover, from the independence of the blue edges in $G_1$, all edges $(i,j')$ in $G_{\rm super}$ appear independently. 
Thus, $G_{\rm super}$ is a bipartite graph on $\calK_2 \times (\calK_2)'$
with planted red edges between $i'$ and $i$ and unplanted blue edges between $i$ and $j'$ appearing independently with probability 
at least $\frac{d_{\rm super}}{K_2}$. Lemma~\ref{lmm:bip.many.cycles} shows that we have $G_{\rm super}$ contains exponentially many long alternating cycles for $d_{\rm super}$ large enough.

\begin{lemma}\label{lmm:bip.many.cycles}
Let $G$ be a bi-colored bipartite graph on $[n]\times [n]'$ whose $n$ red edges are defined by a perfect matching, and blue edges are generated from a bipartite \ER graph with edge probability $D/n$. If $n\geq 525$ and $D\geq 256\log (32e)$, then with probability at least $1-\exp(-Dn/2^{14})$, $G$ contains $\exp(n/20)$ distinct alternating cycles of length at least $3n/4$.
\end{lemma}
\begin{proof}
In order to 
show the existence of exponentially many long cycles in $G$, 
we construct exponentially many subsets of $[n]$ with relatively small overlaps.
For ease of presentation, we assume $n$ is even; otherwise we replace $n/2$ by $\lfloor n/2\rfloor$ in the following proof.

First, it is well-known (by a volume argument) that there exists a collection $\mathbb{V}$ of subsets of $[n]$ of cardinality $n/2$, 
such that
for any distinct $S,T\in\mathbb{V}$, $|S\symdiff T| \geq n/3$ (so that $|S\cap T|\leq n/3$) and 
\begin{equation}
|\mathbb{V}| \geq \frac{\binom{n}{n/2}}{\sum_{i=0}^{n/3} \binom{n}{i}} \geq \binom{n}{n/2} e^{-n (\log 3 - \frac{2}{3}\log 2)} \geq e^{n/20},
\label{eq:bbV}
\end{equation}
where the second inequality follows from
the Chernoff bound 
$2^{-n}\sum_{i=0}^{n/3}{\binom{n}{i}} =\prob{\Binom(n,1/2)\leq n/3}\leq e^{-n D(\Bern(1/3) \|\Bern(1/2) ) };$
the last inequality is from the bound on the binomial coefficient $\binom{n}{n/2} \ge \frac{1}{\sqrt{2n}} 2^n$
by Stirling's approximation~\cite{robbins1955remark}, and the fact that $\frac{1}{\sqrt{2n}}\exp(-n(\log 3-\tfrac{5}{3}\log 2))\geq e^{n/20}$ for all $n\geq 525$.

Next we show that for each size-$(n/2)$ 
subset $V$ of $[n]$, $G[V\times V']$ contains a long alternating cycle with high probability. By the same argument as in~\cite[Theorem~6.8]{frieze2016introduction}, we claim that the graph $G[V\times V']$ contains an alternating cycle of length at least $3n/4$,
if for all subsets $S_1 \times S_2 \subset V\times V'$ such that $|S_1|, |S_2| \ge \frac{n}{32}-1 \ge \frac{n}{64}$, 
there is at least one pair of $u\in S_1, v'\in S_2$ that are connected by a blue edge.
For completeness we include a proof of this claim below.

The claim is shown by constructing a long alternating path on $G[V\times V']$ with the depth-first search (DFS) algorithm. Trace the DFS algorithm with the variables
\begin{enumerate}
\item $U=$the set of unexplored left vertices;
\item $D=$the set of dead (fully explored) left vertices;
\item $\path=(v_1', v_1, v_2', v_2,..., v_r',v_r)=$ the current path.
\end{enumerate}
Initialize at $U=V\backslash \{v_1\}, D=\emptyset$ and $\path=(v_1',v_1)$ where $v_1$ is an arbitrary member of $V$, e.g., the one with the smallest index. At each step of the DFS algorithm, we proceed according to the following two cases:
\begin{enumerate}
\item If there is some $u'\in U'$ such that $(v_r,u')$ is a blue edge, we update
\[
v_{r+1}\leftarrow u; \;\;\; \path\leftarrow (\path, u',u);\;\;\; U\leftarrow U\backslash \{u\};\;\;\; r\leftarrow r+1.
\]
\item If no vertex in $U'$ is incident to $v_r$ in $G$, we update
\[
D\leftarrow D\cup \{v_r\}; \;\;\; \path\leftarrow \path\backslash \{v_r,v_r'\}; \;\;\; r\leftarrow r-1.
\]
\end{enumerate}
By definition of the DFS algorithm, a vertex $v$ is only added to $D$ if $v$ is not incident to any vertex in $U'$. Since the algorithm never adds any new vertices to $U$, the set $D$ and $U'$ are always disconnected in $G$ at any stage of the algorithm.
Furthermore the size of $U$ is non-increasing, and the size of $D$ is non-decreasing, such that $|V|=r+|D|+|U|$ always holds. Therefore at some time point in the DFS algorithm, we have $|D|=|U|=|U'|$. Assuming that all pairs of $S_1\times S_2\subset V\times V'$ of size $S_1\geq n/32-1,S_2\geq n/32-1$ are connected by at least one blue edge, we must have $|D|=|U
|\leq n/32-1$, hence $r=n/2-|D|-|U|\geq 7n/16$.

Apply the assumption again with $S_1=\{v_{r-n/32+1},..., v_{r}\}$, $S_2= \{v_1',..., v_{n/32}'\}$. There exists $v_i\in S_1,v_j'\in S_2$ such that $(v_i,v_j')$ is a blue edge in $G$. We have constructed a cycle
\[
(v_j', v_j, v_{j+1}', v_{j+1},..., v_i', v_i, v_j')
\]
 of length at least $2|i-j|\geq 2(r-n/16)\geq 3n/4$. It follows that
\begin{align}
\nonumber&\mathbb{P}\left\{G[V,V'] \text{ contains an alternating cycle of length }3n/4\right\}\\
\nonumber\geq & \mathbb{P}\left\{\forall S_1\times S_2\subset V\times V' \text{ such that }|S_1|, |S_2| \ge \frac{n}{64}, \exists u\in S_1,v'\in S_2, \text{such that }(u,v')\in \mathcal{E}\left(G\right)\right\}\\
\nonumber\geq &1-{n/2\choose n/64}^2 \left(1- \frac{D}{n} \right)^{(n/64)^2}\\
\nonumber \geq & 1-\left[(32 e)^2\left(1-\frac{D}{n}\right)^{n/64}\right]^{n/64}\\
\nonumber\geq & 1-\left[\exp\left(2+2 \log 32- \frac{D}{64} \right)\right]^{n/64}\\
\label{eq:one.long.cycle}\geq & 1-\exp\left(- Dn/2^{13}\right).
\end{align}
where the last inequality holds 	by the assumption that $D/128 \ge 2+2\log 32$. 

Finally, since $D/2^{14}\ge 1/20$, combining \eqref{eq:one.long.cycle} with \prettyref{eq:bbV} and applying a union bound, we get that
with probability  at least $1-e^{- Dn/2^{14}}$,
$G[V\times V']$ contains an alternating cycle of length $3 n/4$ for every subset $V \in \mathbb{V}$. 
Since the number of left vertices visited by each alternating cycle is at least $3n/8 $, which exceeds the maximum overlap (at most $n/3$ by construction) between distinct subsets in $\mathbb{V}$, these alternating cycles must be distinct. This complete the proof.
\end{proof}

We now finish the proof of Theorem~\ref{thm:sprinkling} by combining the previous results. First, combining \prettyref{lmm:discard_small} and \prettyref{lmm:remove_overlap} yields that with probability at least $1-e^{-K_1/32}-\frac{2}{\sr n\kappa^3}$, there exists a subset $\calK_2 \subset \calK_1$
such that $|\calK_2|=K_2 \ge K_1/16$, and 
 $\left| U_k \right| \ge b$ and $\left|V'_k \right| \ge b$ for all $ k \in \calK_2$. Thus~\eqref{eq:Gsuper.degree} holds. Conditioning on $\{U_k, V'_k \}_{k \in \calK_2}$ and applying Lemma~\ref{lmm:bip.many.cycles}, we get that with probability at least $1-\exp(-d_{\rm super}K_2/2^{14})$, $G_{\rm super}$ contains $\exp(K_2/20)$ distinct alternating cycles of length at least $3K_2/4$. Thus the conclusion of  \prettyref{thm:sprinkling} readily follows.

\section{Proof of \prettyref{lmm:Mbad} under the sparse model}
\label{sec:impossibility.proof}


In this section, we prove \prettyref{lmm:Mbad}, which, as mentioned in Section~\ref{sec:neg.proof.outline}, completes the proof of \prettyref{thm:impossibility} for the sparse model and, in turn, also for the dense model in view of the reduction in \prettyref{app:reduction_dense_sparse}.

To prove~\prettyref{lmm:Mbad}, we apply Algorithm~\ref{alg:cycle_finding} on $G$ with $w_e=\log(\calP/\calQ)(W_e)$ for $e$ in $G$. Define thresholds
\begin{align}
\tau_\mathsf{red} &\triangleq \inf \left\{ x : \calP \left( \log \frac{\calP}{\calQ} \le x  \right) \ge 1/2 \right\}  \label{eq:def_tau_prime}\\
\tau_\mathsf{blue} & \triangleq  \sup \left\{ x: \calQ  \left( \log \frac{\calP}{\calQ} \ge x  \right) \ge 1/2 \right\}. \label{eq:def_tau}
\end{align}
Note that $\tau_\mathsf{red}$ and $\tau_\mathsf{blue}$ are well defined under the assumption that $\calP \ll \calQ$ and $\calQ \ll \calP$, which can be assumed WLOG in view of the reduction argument in \prettyref{app:acneg}.


Assume that Algorithm~\ref{alg:cycle_finding} succeeds. It returns at least $e^{c_6 K_1}$ distinct alternating cycles $C$ of the form 
\[
C=\left(  v_1' , v_1, P_1,  u'_{1}, u_{1},  v_2', v_2, P_2, u_2', u_2,  \cdots,  v_r', v_r, P_r, u_r', u_r  \right),
\]
where $r\geq c_7K_1$ for universal constants $c_6,c_7$. Here each $P_k$ is an alternating path of length $\ell\triangleq4HL$ in $G_1$, with $\Delta(P_k)\geq \zeta\ell/2-\tau_\mathsf{red}$. All the other edges in $C$ are contained in $G_2$. Since $K_1=c_5n$, it follows that $|C|\ge r\ell \ge n c_2$ for some constant $c_2$. Let $m$ be the perfect matching such that $m \symdiff m^*=C$. By choosing $\delta \le c_2$, we get that 
$m \in \Mbad$. Moreover, by construction 
$$
\Delta(C) \ge 3 r  \tau_\mathsf{blue} - 2r \tau_\mathsf{red}  + \sum_{k=1}^r \Delta\left( P_{k} \right)
\ge 3 r  \tau_\mathsf{blue} - 3r \tau_\mathsf{red}+ \zeta  r  \ell/2  \ge   \zeta  r \ell/4 \geq n c_4,
$$
for some constant $c_4>0$, where the last inequality holds by choosing $\ell$ large enough such that $ \zeta \ell \ge 12( \tau_\mathsf{red} - \tau_\mathsf{blue})$. 
Since there are at least $e^{c_6 K_1 } = e^{c_5c_6 n }$ distinct such alternating cycles $C$, the desired \prettyref{eq:prob_bad_desired} follows. To complete the proof, it suffices to show that Algorithm~\ref{alg:cycle_finding} succeeds with probability at least $1-O(1/n)$.

\paragraph{Path construction} Let the family of disjoint sets $L_k\subset V^c$ and $R_k\subset (V^c)'$ be defined in Section~\ref{sec:many.good.trees}. By Theorem~\ref{thm:exploration}, with probability $1-e^{\Omega(n)}$, there exists $\mathcal{K}_1\subset\mathcal{K}\subset A^c$, such that $K_1=|\mathcal{K}_1|= c_5 n$ with constant
\[
c_5=\frac{\gamma}{16(1+\epsilon)^{2HL}\exp\left(7H\alpha +3\epsilon H/2\right)}.
\]
For all $k\in \mathcal{K}_1$, we have $|L_k|\geq s$, $|R_k|\geq s$ where $s=(1+3\epsilon/4)^{2HL}$. 
Moreover, each pair of vertices $u\in L_k$ and $v\in R_k$ are connected via an alternating path $\path$ of length $\ell=4HL+1$ through red edge $(i_k,i_k')$ and $\Delta(\path)\geq 2\zeta HL-\tau_\mathsf{red}$.

\paragraph{Sprinkling} 
We need to check that the sprinkling step yields $e^{c_6 K_1}$ distinct alternating cycles of the form
\[
C=\left(  v_1' , v_1, P_1,  u'_{1}, u_{1},  v_2', v_2, P_2, u_2', u_2,  \cdots,  v_r', v_r, P_r, u_r', u_r  \right)
\]
for $r\geq c_7 K_1$. We show this using Theorem~\ref{thm:sprinkling}. We start by specifying the parameters $\beta$ and $\eta$ that appear in the statement of Theorem~\ref{thm:sprinkling}.

Recall from Algorithm~\ref{alg:cycle_finding} that $V^*=\{i \in V: \log(\calP/\calQ)(W_{i,i'}) \le \tau_\mathsf{red}\}$, and $W_{i,i'}\sim \calP$ for all $i$. By definition of $\tau_\mathsf{red}$ given in~\eqref{eq:def_tau_prime} and the right-continuity of the cumulative distribution function, we have $\mathbb{P}\{i\in V^*\}\geq 1/2$ for all $i$. 
By the independence of the edge weights, Hoeffding's inequality yields
\begin{equation}\label{eq:V*.large}
\mathbb{P}\left\{|V^*|< \frac{\gamma n}{4}\right\} \leq \mathbb{P}\left\{\Binom\left(\gamma n, 1/2\right)\leq \frac{\gamma n}{4}\right\}\leq e^{-\gamma n/8}.
\end{equation}
Therefore with probability at least $1-e^{-\gamma n/8}$, we have $\beta=|V^*|/n\geq \frac{\gamma}{4}$.

To bound the edge probability $\eta/n$ of the blue edges in $G_2$, note that $e$ appears in $G_2$ as a blue edge if and only if it is a blue edge in $G$, and $\log(\calP/\calQ)(W_e)\geq \tau_\mathsf{blue}$. Therefore $\eta/n \leq d/n$, and
\begin{align}
\frac{\eta}{n}=\frac{d}{n}\cdot \calQ\left(\log\frac{\calP}{\calQ}\geq \tau_\mathsf{blue}\right) \geq \frac{d}{2n}, \label{eq:eta_choice_sparse}
\end{align}
where the inequality is from the definition of $\tau_\mathsf{blue}$ given in~\eqref{eq:def_tau} and right-continuity of the cumulative distribution function. In summary, we have parameters
\[
\beta\geq \frac{\gamma}{4},\quad \frac{d}{2}\leq \eta \leq d,\quad s=(1+3\epsilon/4)^{2HL},\quad K_1= c_5 n.
\]
Next, we check that these parameters lead to $b,K_1,\kappa$ and $d_{\rm super}$ that satisfy the assumptions of Theorem~\ref{thm:sprinkling}.
Indeed, since $d \ge 1,$
\[
b=\frac{\beta s\eta}{4}\geq \frac{1}{32}\gamma d(1+3\epsilon/4)^{2HL}\geq \frac{1}{32}\gamma (1+3\epsilon/4)^{2HL}\geq 4
\]
by choosing $H$ large enough; $K_1 =c_5n \geq 8400$ for large enough $n$;
\begin{align*}
\kappa = \frac{2K_1 s\eta}{n} \leq & \frac{(1+3\epsilon/4)^{2HL} d\gamma}{8(1+\epsilon)^{2HL}\exp\left(7H \alpha +3\epsilon H/2\right)}\\
\stackrel{(a)}{\leq} & \frac{(1+3\epsilon/4)^{2HL} (1+\epsilon)^2\gamma}{8(1+\epsilon)^{2HL}\exp\left(6H \alpha +3\epsilon H/2\right)} \stackrel{(b)}{\leq} \frac{1}{16^2},
\end{align*}
where (a) is from $de^{-\alpha }=(1+\epsilon)^2$, (b) holds by choosing $H$ large enough;
\begin{align*}
d_{\rm super}=\frac{1}{32n}K_1 b^2\eta = & \frac{2K_1 \beta^2 s^2\eta^3}{32^2 n}\\
\geq & \frac{\gamma^3 (1+3\epsilon/4)^{4HL}d^3}{32^4(1+\epsilon)^{2HL}\exp\left(7H\alpha +3\epsilon H/2\right)}\\
= & \left[\frac{(1+3\epsilon/4)^{2}}{1+\epsilon}\right]^{2HL}\frac{\gamma^3(1+\epsilon)^6}{32^4\exp\left(4H\alpha +3\epsilon H/2\right)} 
\geq 256\log (32 e),
\end{align*}
where the last inequality is by choosing $L$ to be a large enough constant, since $(1+3\epsilon/4)^{2}/(1+\epsilon)>1$.
We have checked that all the assumptions of Theorem~\ref{thm:sprinkling} are satisfied. 
Thus for all $V^*$ with $|V^*|\geq \gamma n/4$, \prettyref{thm:sprinkling} gives that conditional on $V^*$, $G_{\rm super}$ contains at least $e^{K_2/20}\geq e^{c_6 K_1}$ distinct alternating cycles of length at least $3K_2/4 \geq c_7 K_1$ for universal constants $c_6, c_7$, with (conditional) probability at least
\[
\left(1-e^{-K_1/32}-\frac{2}{\beta n\kappa^3}\right)\left(1-d_{\rm super} K_1/2^{18}\right)= \left(1-e^{-\Omega(n)}-O(1/n)\right)\left(1-e^{-\Omega(n)}\right)=1-O(1/n).
\]
Combined with $\prob{|V^*|<\gamma n/4}=e^{-\Omega(n)}$ from~\eqref{eq:V*.large}, we have shown that the sprinkling step in \prettyref{alg:cycle_finding} goes through for constants $c_6,c_7$ with probability $1-O(1/n)$.

\section{Exponential model}\label{sec:exponential}

In this section, we focus on the special case of complete graph with exponential weights,
where $d=n$, $\calP=\exp(\lambda)$, and $\calQ=\exp(1/n)$, 
and prove the lower bound to the optimal reconstruction error given in 
\prettyref{eq:risk-lb-exp} in \prettyref{thm:exp}. 


As a convention, we call an alternating path a $(2\ell-1)$-alternating path if it consists of $\ell$ red edges and $\ell-1$ blue edges.
Recall that $\Delta(P)= \sum_{e \in \sfb(P)} \log \frac{ \calP}{\calQ} (W_e) -\sum_{e \in \sfr(P)} \log \frac{\calP}{\calQ} (W_e) $ for path $P$ in $G$.
The following result shows that with high probability there exist many disjoint $(2\ell-1)$-alternating $P$
with large $\Delta(P)$.

\begin{theorem}\label{thm:disjoint_paths}
Suppose that $\lambda=4-\epsilon$. There exists absolute constants 
$\epsilon_0, c_1, c_2, c_3>0$, and $n_0=n_0(\epsilon)$,
such that 
for all $\epsilon \le \epsilon_0$, $ c_1/\epsilon
\le \ell \le e^{-c_1/\sqrt{\epsilon}} \sqrt{n}$,
 and $n \ge n_0$, 
with probability at least $\frac{1}{2} - \frac{\ell^2 e^{c_2/\sqrt{\epsilon}} }{n}$, 
there is a set  $S^* $ of disjoint $(2\ell-1)$-alternating paths $P$ with
\begin{equation}
\left| S^* \right | \ge \frac{n}{\ell^2 e^{ c_3/\sqrt{\epsilon} } },
\label{eq:Sstar}
\end{equation}
such that for every $\ell/3 \le \ell' \le \ell$ and 
every $(2\ell'-1)$-alternating subpath $Q$ of $P$, it holds that 
\begin{equation}
\Delta(Q) \ge  (\lambda-1/n) \zeta_0 \epsilon \ell'. 
\label{eq:wQ}
\end{equation}
with $\zeta_0=\frac{1}{96}$. 
\end{theorem}

\prettyref{thm:disjoint_paths} provides the needed ingredient for proving the negative part of \prettyref{thm:exp}.
The proof of the positive part is deferred till \prettyref{app:ode}.

\begin{proof}[Proof of \prettyref{thm:exp}: negative part]
As mentioned in Section~\ref{sec:negative}, Theorem~\ref{thm:exp} is a direct consequence of~\prettyref{lmm:Mgood}
and~\prettyref{lmm:Mbad} with $c_0, c_2= e^{-O(1/\sqrt{\epsilon})}$. 
Therefore we only need to prove Lemma~\ref{lmm:Mbad} with the desired $c_0, c_2$.
We choose 
\begin{equation}
\gamma = \frac{\epsilon}{8}, \quad \zeta_0 = \frac{1}{192}, \quad \tau = e^{\tau_0/\sqrt{\epsilon}}, \quad \ell =\frac{18\tau}{\zeta_0 \epsilon}
\label{eq:exp-parameters}
\end{equation}
for some constant $\tau_0$ to be specified later. 
Without loss of generality, we assume that $\ell$ is a multiple of $3$. By the same argument that we used in the proof of \prettyref{thm:impossibility} to reduce condition~\eqref{eq:impossibility} to~\eqref{eq:impossible-assumption}, we also assume here that $\lambda=4-\epsilon$.

In Step 1 of the two-stage cycle finding scheme, we first apply~\prettyref{thm:disjoint_paths} to 
find a set $S^*$ of disjoint $(2\ell-1)$-alternating paths 
in $G_1=G[V^c\times (V^c)']$ with $|V^c|=(1-\gamma) n$.
Specifically, 
by shrinking every edge weight in $G_1$ by a multiplicative factor $1-\gamma$,
we arrive at an instance of the exponential model with $n'$ left (right) vertices, planted weight distribution $\exp(\lambda')$ and 
null weight distribution $\exp(\frac{1}{n'})$, where $n'=(1-\gamma) n$ and $\lambda'=\frac{\lambda}{1-\gamma}$. 
Since   $\lambda=4-\epsilon$ and
$\gamma=\epsilon/8$, it follows that 
$\lambda' \le 4- \epsilon/2 \equiv 4-\epsilon'$. Replacing 
$(n,\lambda,\epsilon)$ by $(n',\lambda',\epsilon')$, 
the same conclusion of \prettyref{thm:disjoint_paths} holds for $G_1$ (without weight shrinkage) with $\zeta_0 = \frac{1}{192}$. 

For any alternating path $\path$ in $S^*$, it is centered at a red edge $(k,k')$ with a $(\ell-1)$-alternating subpath on each side.
Let $L_k$ denote the set of left vertices in the first
 $(2\ell/3)$-segment of $\path$
and $R_k$ denote the right vertices in the last
 $(2\ell/3)$-segment of $\path$. 
Then we have 
$|L_k|=|R_k| =s$, where $s= \ell/3$. Moreover, each pair of vertices $u \in L_k$
and $v \in R_k$ is connected via a $(2\ell'-1)$-subpath $Q$ of $\path$ through the
red edge $(k,k')$ consisting of $\ell'$ red edges and $\ell'-1$ blue edges,
where $\ell' \ge \ell/3$ and $\Delta(Q) \ge   (\lambda-1/n) \zeta_0 \epsilon \ell'$. 
Let $\calK_1$ denote the collection of such indices $k$, where $K_1\triangleq |\calK_1| = |S^*|$. 
It follows from  \prettyref{thm:disjoint_paths} that 
with probability at least $1- O(1/n)$, $K_1 \ge c_5 n$ with constant
$$
c_5 \triangleq  \frac{1}{\ell^2 e^{ c_3/\sqrt{\epsilon} } }. 
$$

Following Step 2 of \prettyref{alg:cycle_finding}, we connect $\{L_k, R_k\}_{k \in \calK_1}$ to form alternating cycles in $G$ via sprinkling.
Choose 
$$
\tau_\mathsf{red} = \log (n \lambda), \quad \text{ and } \quad \tau_\mathsf{blue}= \log (n\lambda) - \left( \lambda-1/n \right) \tau.
$$
Note that $\log \frac{\calP}{ \calQ} (W_e) =\log (n\lambda)- (\lambda-1/n) W_e$. Thus $V^*=V$
and equivalently subgraph $G_2$ is the subgraph of $G$ that contains  every red edge in $V \times V'$, 
and every blue edge $e \in [n]\times [n]'\backslash(V^c \times V^{c'})$ if $W_e \le \tau$.
Then we apply Theorem~\ref{thm:sprinkling} with $V^*=V$
to show there exist exponentially many distinct alternating cycles via  sprinkling. 
We start by specifying the parameters $\beta$ and $\eta$ that appear in the statement of Theorem~\ref{thm:sprinkling}.
Note that $\beta=\frac{|V^*|}{n}=\gamma$ and the average blue degree is
\begin{align}
\eta = n \prob{ \exp (1/n) \le \tau } = n \left( 1- e^{-\tau /n}  \right). 
\label{eq:eta_choice_exp}
\end{align}
Using $e^{-x} \ge 1-x$ and $ e^{-x} \le 1- x/2$ for $x \in [0,1]$,
and $\tau/n \le 1$ for all sufficiently large $n$, we have
$ \tau/2 \le  \eta \le \tau$.  Next, we check the parameters above
lead to $b, K_1, \kappa$, and $d_{\rm super}$ that satisfy the assumptions of 
Theorem~\ref{thm:sprinkling}.
In particular, 
$$
b= \frac{\beta s \eta}{4} \ge \frac{ \gamma \ell \tau}{24} = \frac{\epsilon \ell \tau}{192} \ge 4,
$$
by choosing the constant  $\tau_0$ in \prettyref{eq:exp-parameters} sufficiently large; 
$K_1=c_5 n  \ge 8400$ for all large enough $n$;
$$
\kappa = \frac{2K_1 s \eta}{n} \le \frac{ 2  \ell \tau }{ 3 \ell^2 e^{ c_3/\sqrt{\epsilon} } }
= \frac{ \zeta_0 \epsilon } {27 e^{ c_3/\sqrt{\epsilon} } } \le \frac{1}{16^2}
$$
for all sufficiently small $\epsilon$; and
\begin{align}
d_{\rm super}= \frac{1}{32n} K_1 b^2 \eta \ge \frac{ 1}{32 \ell^2 e^{ c_3/\sqrt{\epsilon} } } \left(  \frac{\epsilon \ell \tau}{192} \right)^2   \frac{\tau}{2} 
= \frac{ \epsilon^2 e^{3 \tau_0/\sqrt{\epsilon}} }{ 64 \cdot (192)^2 e^{ c_3/\sqrt{\epsilon} } } \ge 256 \log (32 e),
\label{eq:avg_deg_sup_exp}
\end{align}
by choosing $\tau_0$ to be a sufficiently large constant. Having verified all assumptions of Theorem~\ref{thm:sprinkling}, we conclude that $G_\text{super}$ contains at least $e^{K_2/20}\geq e^{c_6 K_1}$ distinct alternating cycles of length at least $3K_2/4\geq c_7 K_1$ for universal constants $c_6, c_7$, with probability at least
\[
\left(1-e^{-K_1/32}-\frac{2}{\beta n\kappa^3}\right)\left(1-d_{\rm super} K_1/2^{18}\right)= \left(1-e^{-\Omega(n)}-O(1/n)\right)\left(1-e^{-\Omega(n)}\right)=1-O(1/n).
\]
In conclusion, the sprinkling step yields $e^{c_6 K_1}$ distinct alternating cycles (in fact here one such cycle suffices) $C \in \calC$ of the form
\[
C=\left(  v_1' , v_1, P_1,  u'_{1}, u_{1},  v_2', v_2, P_2, u_2', u_2,  \cdots,  v_r', v_r, P_r, u_r', u_r  \right)
\]
for $r\geq c_7 K_1$. 
It follows that $|C| \ge 2r \ell/3 = c_2 n$,
where 
$$
c_2 \triangleq 2 c_7 c_5 \ell /3 =  \frac{2c_7 }{3 \ell e^{ c_3/\sqrt{\epsilon} } }
=   \frac{c_7\zeta_0 \epsilon }{27  e^{ (c_3+\tau_0)/\sqrt{\epsilon} } }.
$$
 Moreover, by construction of $G_2$, 
 $$
\Delta(C) \ge 3 r  \tau_\mathsf{blue} - 2r \tau_\mathsf{red}  + \sum_{k=1}^r \Delta\left( P_{k} \right)
\ge - 3 r(\lambda-1/n)  \tau +  (\lambda-1/n) \zeta_0 \epsilon r \ell  /3  \ge   (\lambda-1/n)  r \zeta_0 \epsilon \ell  /6,
$$
where the last inequality holds by the choice of $\zeta_0 \epsilon \ell = 18 \tau$ in \prettyref{eq:exp-parameters}.
It follows that 
$$
\frac{\mu_W(\Mbad)}{\mu_W(M^*)}
\ge \left| \calC \right| e^{ \Delta(C) }
\ge e^{c_6 K_1} e^{ (\lambda- 1/n)  c_7 K_1 \zeta_0 \epsilon \ell  /6 } \ge \exp (c_0 n),
$$
where 
$$
c_0 \ge   c_7 K_1 \zeta_0 \epsilon \ell  /6  \ge \frac{c_7 \zeta_0 \epsilon \ell }{ 6 \ell^2 e^{ c_3/\sqrt{\epsilon} } }
=  \frac{c_7 \zeta_0^2  \epsilon^2 }{ 6  \cdot 18  e^{ (c_3+\tau_0) /\sqrt{\epsilon} } }.
$$
Theorem~\ref{thm:exp} then readily follows by combining Lemmas~\ref{lmm:Mgood}
and~\ref{lmm:Mbad}. 
\end{proof}


\medskip
In the sequel, we proceed to prove~\prettyref{thm:disjoint_paths}.
A more direct approach is to define 
$S$ as the set of 
$(2\ell-1)$-alternating paths $P$ with large $\Delta(P)$ and show that $\expect{|S|}$ is large while
$\var(|S|) \ll \left( \expect{|S|} \right)^2$ so that 
$|S|$ concentrates on its mean. 
Unfortunately, this idea fails as the second moment of 
$|S|$ blows up for $\ell=\Theta(n)$. 
This is because
conditioning on finding an
alternating path $P$ with large $\Delta(P)$, it is very likely to have a large number
of paths $P'$ with large $\Delta(P')$ overlapping with $P$. 
These clusters of overlapping paths induce an excessive contribution to the second moment.
To address this issue,  we adopt the notion of \emph{uniformity} introduced in~\cite{ding2015supercritical} for studying minimum mean-weight cycles.

Recall that for a set $T$ of edges, $\sfr(T)$ and $\sfb(T)$ denote the set of red and blue edges in $T$, respectively. Furthermore, define
\begin{align*}
\wt_\sfr(T) &= \sum_{e \in \sfr(T)} W_e,\\
\wt_\sfb(T) &= \sum_{e \in \sfb(T) } W_e.
\end{align*}

Let $\calP_\ell$ denote the set of $(2\ell-1)$ alternating paths with $\ell$ red edges and $\ell-1$ blue edges. 
We agree upon that each $P \in \calP_\ell$ is oriented so that it starts in the left vertex set and ends in the right vertex set.
Let $\phi_1, \phi_2, \ldots, \phi_{|\sfr(P)|}$ (resp.~$\psi_1, \psi_2, \ldots, \psi_{|\sfb(P)| } $) 
denote the sequence of red (resp.~blue) edge weights in this order. 
Define
\begin{align*}
\dev\sfr(P)  &= \sup_{1 \le k \le |\sfr(P)| } \left|   \sum_{j=1}^k  \frac{ \phi_j  |\sfr(P)| }{\wtr(P) }  - k   \right|, \\
\dev\sfb(P)  &= \sup_{1 \le k \le |\sfb(P)| } \left|   \sum_{j=1}^k  \frac{ \psi_j  |\sfb(P)| }{\wtb(P) }  - k   \right|
\end{align*}
which characterize the maximum fluctuation of edge weights on path $P$. 

\begin{definition}[Lightness and uniformity]\label{def:light_uniform}
We say an alternating path $P$ is \emph{$\left(a, b, \eta \right)$-light}, if 
\begin{align*}
\left| \wtr(P) -  a \cdot |\sfr(P)|  \right| & \le \eta/2  \\
\left| \wtb(P) - b \cdot |\sfb(P)|  \right|  & \le \eta/2;
\end{align*}
and \emph{$A$-uniform},  if $\dev\sfr(P) \le A$ and $\dev\sfb(P) \le A$. 
\end{definition}

Let us proceed to the definition of $S$. 
Define
$$
S \triangleq \left\{ P \in \calP_\ell: \text{ $P$ is $\left( a , b,
\eta \right)$-light and $A$-uniform} \right\}.
$$
with 
\begin{equation}
a= \frac{2}{\lambda}, \quad b = \frac{2-\zeta}{\lambda}.
\label{eq:ab}
\end{equation}
(We will choose $\eta=1,$ $\zeta=\frac{\epsilon}{4}$, and $A=\Theta(\frac{1}{\sqrt{\epsilon}})$ later.)
Then for any $P \in \calP_\ell$, we have
$$
\Delta(P) = - (\lambda - 1/n ) \left[ \wtb(P)  - \wtr(P) \right] \ge - (\lambda - 1/n )  \left[ b(\ell-1) - a \ell  - \eta \right] \ge (\lambda - 1/n ) \frac{\epsilon \ell}{96}.
$$
Furthermore, using the $A$-uniformity of $P$, it is not hard to verify that for every $(2\ell'-1)$-alternating subpath $Q$ of $P$ with $\ell/3 \le \ell' \le \ell$,
$\Delta(Q) \ge(\lambda - 1/n )  \frac{\epsilon \ell'}{96}$ (see forthcoming~\prettyref{eq:verify_Delta_Q}). 
Thus, to prove~\prettyref{thm:disjoint_paths}, the key remaining challenge is to show 
$S$ contains a large vertex-disjoint subcollection $S^*$ with $|S^*| =\Omega( n/\ell^2).$

The intuition that we require the mean weight of 
red (blue) edges in $P$ to be around $2/\lambda$ is as follows.
Given $P \in \calP_\ell$, $\wtr(P)$ is distributed as a sum of $\ell$ i.i.d.\ $\exp(\lambda)$
random variables, while $\wtb(P)$ is distributed as a sum of $\ell-1$ i.i.d.\ $n\cdot \exp(1)$. 
Conditional on $\wtr(P)$ being close to $\wtb(P)$, we expect\footnote{Indeed, using the density of sum of exponentials (see \prettyref{eq:erlang-pdf} in \prettyref{app:erlang}), 
the probability that $\wtr(P)/\ell$ and $\wtr(P)/(\ell-1)$ are both close to a given value $x$ is proportional to $x^{2\ell-3} e^{-\ell x(\lambda + \frac{\ell-1}{n \ell})}$, which, for large $n$ and $\ell$, is approximately maximized at $x = \frac{2}{\lambda}$.} that the mean weight for
both $\wtr(P)$ and $\wtb(P)$ are close to $2/\lambda$.
Moreover, we require the mean weight of blue edges to be slightly below $2/\lambda$ so that $\Delta(P)$ is positive.

Note that we further restrict the alternating path $P$ in $S$
to be uniform, in the sense that 
the mean weight of red (resp.\ blue) edges in every subpath $Q$ of $P$
concentrates around $a$ (resp.\ $b$). As we will see in the next section,
the uniformity can be interpreted as requiring an exp-minus-one random walk
conditioned on returning to the origin to have a restricted range, which
is shown to hold with a sufficiently large probability. This implies that  
after restricting to uniform alternating paths, the first moment
$\expect{|S|}$ is still
large.  Moreover, the number of uniform alternating paths has small
enough variance for the second moment method to go through.

%

\subsection{Exponential random walks }

In order to study the pathwise fluctuation of the edge weights on a given path, let us consider the following problem. Let $X_1, X_2, \ldots, X_\ell \iiddistr \exp(\mu)$ and let $X=\sum_{i=1}^\ell X_i$. Define a process
$$
R_j=\sum_{i=1}^j \left( \frac{X_i}{X} \ell  -1 \right), \quad 0 \le j \le \ell.
$$
Thanks to the property of the exponential distribution, conditional on any any realization of $X$, 
$\{ \frac{X_i}{X}: 1 \le i \le \ell\}$ are uniformly distributed on the $(\ell-1)$-dimensional probability simplex,
regardless 
of the value of $\mu$. In particular,  $\{ \frac{X_i}{X}: 1 \le i \le \ell\}$ and $X$ are independent.

\begin{lemma}
\label{lmm:expwalk}
The process $(R_j: 0 \le j \le \ell)$ is independent of $X$. Furthermore, $(R_j: 0 \le j \le \ell)$
is distributed as an exp-minus-one random walk started from the origin and conditioned to return to the origin at time $\ell$ (known as the \emph{exp-minus-one $\ell$-bridge}). 
\end{lemma}
\begin{proof}
Since $X=\sum_{i=1}^\ell X_i$ follows the $\text{Erlang}(\ell,\mu)$ distribution as defined in~\prettyref{app:erlang},
the joint conditional density of $X_1,...,X_\ell$ given $X$ takes the form
\[
f\left(x_1,...,x_\ell\mid x\right)=\frac{\prod_{i=1}^\ell\mu\exp\left(-\mu x_i\right)\indc{x=\sum_{i=1}^\ell x_i}}{\mu^\ell x^{\ell-1}\exp(-\mu x)/(\ell-1)!}=\frac{(\ell-1)!}{x^{\ell-1}}\indc{x=\sum_{i=1}^\ell x_i}.
\]
Therefore, the distribution of $(X_j/X: 1\leq j\leq \ell)$ conditional on $X$ is uniform on the $\ell$-dimensional simplex and does not depend on $X$. Since $(R_j:0\leq j\leq \ell)$ is a function of $(X_j/X: 1\leq j\leq \ell)$, it is also independent of $X$. We have
\begin{align*}
\left(R_j: 0\leq j\leq \ell \right)\overset{d}{=} & \left(R_j: 0\leq j\leq \ell \right) \text{ conditional on }X=\ell\\
\overset{d}{=} & \left( \sum_{i=1}^j\left(X_i-1\right): 0\leq j\leq \ell \right) \text{ conditional on } \sum_{i=1}^\ell \left(X_i-1\right) = 0
\end{align*}
is an exp-minus-one $\ell$-bridge.
\end{proof}

%

Fix an alternating path $P \in \calP_{\ell}$. 
Applying \prettyref{lmm:expwalk}, we conclude that, crucially, 
$\wtr(P)$ and  $\{ W_e/\wtr(P): e \in \sfr(P)\}$ are independent. Furthermore, $\{ W_e/\wtr(P): e \in \sfr(P)\}$ has the same distribution as $\{ \frac{X_i}{X}: 1 \le i \le \ell\}$. Therefore
\begin{equation}
\dev\sfr(P) \; 
\overset{d}{=}  \; \max_{0 \leq  j \leq \ell} \left| R_j \right|.
\label{eq:interpolate-subpath-walk}
\end{equation}
The similar conclusion applies to blue edge weights $\wtb(C)$ and $\{ W_e/\wtb(C): e \in b(C)\}$.

Adapted from \cite[Lemma 2.3]{ding2013scaling},
the following lemma bounds the probability that the range of an exp-minus-one $\ell$-bridge is at most $A$;
this result is 
crucial for lower bounding the first moment of $A$-uniform alternating paths. 
More precise version of the results can be also found in \cite[Lemma 3.9 and Equation (8)]{ding2015supercritical}.
\begin{lemma}\label{lmm:brownian_bridge}
There exist universal constants $c_0, c'_0>0$ such that for all $A \ge 1$ and $\ell \ge A^2$
\begin{align}
 \exp\left( - \frac{c_0 \ell}{A^2} \right)  \le  \prob{\max_{0 \leq  j \leq \ell} \left| R_j \right|  \le A } \le  \exp\left( - \frac{c'_0 \ell}{A^2} \right).
 \label{eq:lambdaA}
\end{align}
\end{lemma}
Applying \prettyref{eq:lambdaA} separately to both red and blue edges, we conclude that for any $P \in \calP_\ell$,
\begin{equation}
p_\ell \triangleq \prob{P \text{ is $A$-uniform}} 
\ge
\exp\left( - \frac{2c_0 \ell}{A^2} \right).
\label{eq:Auniform}
\end{equation}
Furthermore, the event that 
$P$ is $A$-uniform is independent of $\{ \wtr(P), \wtb(P)\}$ and hence also the event that $P$ is $(a,b)$-light.

\subsection{First moment estimates}
\label{sec:first}
In the remainder of this section, 
all the expectations are conditioned on $M^*= m$ for some fixed matching $m \in \calM$, e.g., the identity.

\begin{lemma}\label{lmm:impossible_first_moment}
Suppose $\lambda \le 4$ and $\eta \le 1$.
There exist a universal constant $c_0>0$, such that 
for all $A \ge 1$ and $\ell \ge  A^2$,
\begin{align}
\expect{\left| S \right|} 
\geq & ~ n
 \frac{ \eta^2 \lambda  }{ 8e^{3} b \ell }  
 \left( 2b e^{- \frac{b}{n} } \right)^{\ell-1} e^{- \frac{\ell^2}{n}} p_\ell  
\label{eq:firstmoment-precise}
 \\
\geq & ~ 
n  \frac{\eta^2 \lambda }{ 16 e^3 b^2 \ell}
\left( 2b \e^{- \frac{c_0}{A^2} - \frac{\ell +b}{n} } \right)^{\ell}.
\label{eq:firstmoment-crude}
\end{align}
\end{lemma}
\begin{proof}
Recall that the planted edge weights are i.i.d.\ $\exp(\lambda)$ 
and unplanted edge weights are i.i.d.\ $\exp(1/n)$. 
Recall that for $\ell \ge 1$, $\calP_\ell$ denotes the set of alternating paths of length $2\ell-1$ with $\ell$ red edges and $(\ell-1)$ blue edges.
We have that 
$$
|\calP_\ell| = n (n-1) \cdots (n-\ell +1).
$$
To see this, note that
there are $n(n-1) \cdots (n-\ell +1)$ different choices for $\ell$ left vertices on $P$. 
The right vertices are automatically fixed according to the
underlying true matching $M^*$. 
Write $|\calP_\ell|
= n^{\ell}
\exp\sth{\sum_{k=0}^{\ell} \log \pth{1- \frac{k}{n}}}$. Note that by monotonicity,
\[
- \frac{1}{\ell} \sum_{k=0}^{\ell -1} \log \pth{1- \frac{k}{n}} \leq \frac{n}{\ell} 
\int_0^{\ell/n} - \log(1-x)\dx  = F\pth{\frac{\ell}{n}}
\]
where 
\begin{equation}
F(\delta) \triangleq \frac{1}{\delta}(\delta + (1-\delta) \log(1-\delta))
\label{eq:Fdelta}
\end{equation}
is increasing in $\delta \in [0,1]$ 
and satisfies $F(\delta) \leq \delta$ for all $\delta \in [0,1]$.
Then 
 $$
 \exp\sth{\sum_{k=0}^{\ell-1} \log \pth{1- \frac{k}{n}}} 
 \ge \exp\sth{-\ell^2/n}.
 $$
 In conclusion, we get that
\begin{equation}
|\calP_\ell|
 \ge 
 \left(n \e^{ -\ell/n} \right)^{\ell}.
\label{eq:Cell-lb}
\end{equation}


Fix an alternating path $P \in \calP_\ell$. 
Next we bound the probability that $P$ is $(a,b,\eta)$-light. 
Recall that 
$\sfr(P)$ denote the set of red (planted) edges and 
$\wtr(P)$ denote the total of their weights. 
Then $|\sfr(P)|=\ell$ and 
$\wtr(P)\sim \text{Erlang}(\ell,\lambda)$. Using the Erlang density function in \prettyref{eq:erlang-pdf},  we have
\begin{align*}
\prob{a \ell- \frac{\eta}{2} \le \wtr(P) \le a \ell } 
 & = \int_{a \lambda \ell- \frac{ \lambda \eta }{2}}^{ a \lambda \ell } 
 \frac{x^{\ell-1} \e^{- x}}{(\ell-1) !} \dx \\
 & \ge \frac{ \eta \lambda }{2 (\ell-1) !} 
  x^{\ell-1} \e^{- x} \big\vert_{x= a \lambda \ell}\\
  & = \frac{\eta  \lambda}{2(\ell-1) !} \left(a \lambda \ell\right)^{\ell-1} \e^{-a \lambda \ell}
  =\frac{\eta  \lambda}{2(\ell-1) !} \left( 2 \ell\right)^{\ell-1} \e^{-2 \ell},
\end{align*}
where 
the inequality holds because 
$x^{\ell-1} \e^{- x}$ is decreasing for $x \ge \ell-1$
and by the assumptions $a\lambda = 2$, $\lambda\le 4$ and $\eta \le 1$ so that $a \lambda \ell-\frac{\eta \lambda }{2} \geq \ell-1$.

Similarly, recall that 
$\sfb(P)$ denote the set of blue (unplanted) edges and 
$\wtb(P)$ denote the total of their weights. 
Then $|\sfb(P)|=\ell-1$ and $\wtr(P)\sim \text{Erlang}(\ell-1,\frac{1}{n})$. Thus
\begin{align*}
\prob{b  (\ell -1) 
\le \wtb(P) \le b (\ell -1) + \frac{\eta}{2} } 
& = \int_{ \frac{b (\ell -1) }{n} }^{  \frac{b (\ell -1) }{n} + \frac{\eta}{2n} } 
\frac{ x^{\ell-2} \e^{- x } }{(\ell-2) !} \dx \\
 & \ge  \frac{1}{ (\ell-2) ! } 
 \left( \frac{ b (\ell -1) }{n} \right)^{\ell-2} 
\int_{ \frac{b (\ell -1) }{n} }^{  \frac{b (\ell -1) }{n} + \frac{\eta}{2n} } 
 \e^{- x } \dx \\
  & = \frac{1}{(\ell-2) ! } \left( \frac{b(\ell-1)}{n}  \right)^{\ell-2}  \e^{- \frac{b (\ell -1)}{ n}} \left( 1 - \e^{-\frac{\eta}{2n}} \right) \\
  & \ge  \frac{1}{(\ell-2) !}  \left( \frac{b(\ell-1)}{n}  \right)^{\ell-2}  \e^{- \frac{b (\ell -1)}{ n}}  \frac{\eta}{4n},
\end{align*}
where the last inequality holds because 
$  1 - \e^{-x}  \ge x/2$ for $0\le x\le 1$.

Since $\wtr(P)$ and $\wtb(P)$ are independent, 
it follows from the last two displayed equations that
\begin{align*}
\prob{P \text{ is $(a,b,\eta)$-light}} 
& \ge \prob{ 
a \ell- \frac{\eta}{2} 
\le \wtr(P) \le a \ell , b(\ell -1)
\le \wtb(P) \le b(\ell -1) + \frac{\eta}{2}  } \\
& \ge \frac{ \eta^2 \lambda \ell }{ 8 b \ell! (\ell-1)!  }  
\left(   \frac{ 2 b  \ell (\ell-1) }{ n } \right)^{\ell-1} \e^{-2 \ell - \frac{b(\ell-1)}{ n} } \\
& \overset{(a)}{\ge} 
\frac{ \eta^2 \lambda }{ 8e b \sqrt{\ell(\ell-1)} }  
\left(   \frac{ 2b e^2 }{ n } \right)^{\ell-1} \e^{-2 \ell - \frac{b(\ell-1)}{ n} } 
\\
& \ge
\frac{ \eta^2 \lambda  }{ 8e^{3} b \ell }  
\left(  \frac{2b}{n} e^{ - \frac{b}{n} } \right)^{\ell-1}
\numberthis \label{eq:light}
\end{align*}
where $(a)$ holds due to $n! \le \e n^{n+1/2} \e^{-n}$.

Finally, combining \prettyref{eq:Cell-lb}, \prettyref{eq:Auniform}, \prettyref{eq:light} and 
using the independence of the events 
$\{P \text{ is $(a, b,\eta)$-light}\}$ and $\{P \text{ is $A$-uniform}\}$, we get 
 \begin{align*}
 \expect{|S|} 
& =  |\calP_\ell| \cdot \prob{P \text{ is $(a, b,\eta)$-light}} \cdot \prob{P \text{ is $A$-uniform}} \\
 & \geq 
 \left(n \e^{ - \ell/n } \right)^{\ell} \cdot
 \frac{ \eta^2 \lambda  }{ 8e^{3} b \ell }  
\left(   \frac{ 2 b }{ n  } e^{ - \frac{b}{n} } \right)^{\ell-1}
\cdot 
 p_\ell \\
& \ge 
n  \frac{\eta^2 \lambda }{ 16 e^3 b^2 \ell}
\left( 2b \e^{- \frac{c_0}{A^2} - \frac{\ell +b}{n} } \right)^{\ell}.
 \end{align*}
\end{proof}

\subsection{Second-moment estimates}
\label{sec:second}

\begin{lemma}\label{lmm:impossible_second_moment}
There exist absolute constants $C_1,C_2$, and $\epsilon_0$ such that the following holds. 
Let $\lambda = 4-\epsilon$ for some $0<\epsilon \leq \epsilon_0$. 
Then there exists $n_0=n_0(\epsilon)$ such that upon choosing 
$A = \ceil{\frac{C_1}{\sqrt{\epsilon}}}$, 
$\zeta=\epsilon/4$, 
and $\eta = 1$, for any $\ell$ with $A^2 \le \ell \le e^{-C_2 A} \sqrt{n}$ and $n\geq n_0$, it holds that
$$
\var(|S|) \le \expect{|S|}
\times 
\left( 1+   \frac{ e^{C_2/\sqrt{\epsilon} } \ell^2  }{ n} 
\expect{|S|}  \right).
$$
\end{lemma}
\begin{remark}
In the above estimate it is crucial to get $\frac{\Var(|S|)-\expect{|S|}}{\expect{|S|}^2}= O\left( \frac{\ell^2}{n}\right)$,
so that we can extract from $S$  a vertex-disjoint subcollection $S^*$ of $|S^*| =\Omega( n/\ell^2)$ in~\prettyref{sec:Turan} by applying Tur\'an's theorem.
This turns out to be instrumental to ensure that the super graph in the sprinkling stage is 
supercritical as shown in~\prettyref{eq:avg_deg_sup_exp}.
As a result, we need to be careful with terms that are polynomial in $\ell$ in the second moment computation.
\end{remark}

\begin{proof}
Note that 
\begin{align*}
\mathrm{Var} \left( |S| \right)
& =\expect{|S|^2} - \left( \expect{|S|} \right)^2 \\
& = \sum_{P, P' \in \calP_\ell} 
\left( \prob{ P \in S , P' \in S }
- \prob{ P \in S } \prob{ P' \in S }\right) \\
& \overset{(a)}{=} \sum_{P \in \calP_\ell} \sum_{P' \in \calP_\ell: |P \cap P'| \ge 1}
\left( \prob{ P \in S , P' \in S }
- \prob{ P \in S } \prob{ P' \in S } \right) \\
& \le \sum_{P \in \calP_\ell} \sum_{P' \in \calP_\ell: |P \cap P'| \ge 1}
\prob{ P \in S , P' \in S } \\
&= \sum_{P \in \calP_\ell } \prob{  P \in S }
\pth{1 + 
\sum_{P' \in \calP_\ell: |P \cap P'| \ge 1, P'\neq P} \prob{ P' \in S \;\big\vert 
\; P \in S }}, \numberthis \label{eq:var1}
\end{align*}
where $(a)$ holds because the weights in $P$ and those in $P'$ 
are mutually independent if $P \cap P' =\emptyset$.

Fix any $P \in\calP_\ell$. 
Recall that $P$ is alternating paths with $\ell$ red edges and $\ell-1$ blue edges which starts from a left vertex and ends with a right vertex, and whose first and last edge are both red.
We group the summands of the inner sum according to how $P'$ overlaps with $P$ (see Figure~\ref{fig:count.P'}).
For each $P' \neq P$ with $|P \cap P'| \ge 1$:
\begin{itemize}
\item  $P'\cap P$ consists of $k$ disjoint \emph{alternating} paths $\calP=(P_1, \ldots, P_k)$ for some $k \ge 1$. 
Let $|P_i|=m_i$ and $m=\sum_{i=1}^{k} m_i < 2\ell$. 
Note that the first edge and the last edge in $P_i$ must be red.
Then $m_i\geq 1$ must be  odd  
and $P_i$ has $(m_i+1)/2$ red edges and $(m_i-1)/2$ blue edges.  

\item $P'\backslash P$ consists of $k+1$ disjoint \emph{alternating} paths $\calP'\triangleq (P'_1, \ldots, P'_{k+1})$,
where we allow $P'_1$ and $P'_{k+1}$ to be possibly empty. 
Let $|P'_i|=m'_i$ and $m'=\sum_{i=1}^{k+1} m'_i = 2\ell-1-m$.  
Note that for $i \in \{1, k+1\}$, $m'_i$ must be even and $P'_i$  has $m'_i/2$ red edges and 
$m'_i/2$ blue edges.  For $2 \le i \le k$, 
the first edge and the last edge in each $P'_i$ must be blue.
Hence, $m'_i \ge 1 $ must be odd  and  $P'_i$ has $(m'_i-1)/2$ red edges and $(m'_i+1)/2$ blue edges.  
\end{itemize}

\begin{figure}[ht]
  \begin{center}
  
  \begin{tikzpicture}[scale=0.9,every edge/.append style = {thick,line cap=round},font=\scriptsize]
  
  \tikzset{->-/.style={decoration={
  markings,
  mark=at position #1 with {
  \arrow[black, scale = 1]{stealth}}
  },postaction={decorate}}}

\draw (1,-0.5) node(1) [smallnode] {};
\draw (2,-0.5) node(2) [smallnode] {};
\draw (3,0) node(3) [smallnode] {};
\draw (4,0) node(4) [smallnode] {};
\draw (5,0) node(5) [smallnode] {};
\draw (6,0) node(6) [smallnode] {};
\draw (7,0) node(7) [smallnode] {};
\draw (8,0) node(8) [smallnode] {};
\draw (9,0.5) node(9) [smallnode] {};
\draw (10,0.5) node(10) [smallnode] {};
\draw (11,0) node(11) [smallnode] {};
\draw (12,0) node(12) [smallnode] {};
\draw (13,0) node(13) [smallnode] {};
\draw (14,0) node(14) [smallnode] {};
\draw (15,0) node(15) [smallnode] {};
\draw (16,0) node(16) [smallnode] {};
\draw (17,0.5) node(17) [smallnode] {};
\draw (18,0.5) node(18) [smallnode] {};
\draw (7,-0.7) node(19) [smallnode] {};
\draw (8,-1) node(20) [smallnode] {};
\draw (9,-1.2) node(21) [smallnode] {};
\draw (10,-1.2) node(22) [smallnode] {};
\draw (11,-1) node(23) [smallnode] {};
\draw (12,-0.7) node(24) [smallnode] {};

\path [red] (1) edge [line width = 2pt] (2);
\path [blue] (2) edge (3);
\path [red] (3) edge [line width = 2pt] (4);
\path [blue] (4) edge (5);
\path [red] (5) edge [line width = 2pt] (6);
\path [blue, bend left = 60] (6) edge[->- = 0.6] (7);
\path [red] (7) edge [line width = 2pt] (8);
\path [blue] (8) edge (9);
\path [red] (9) edge [line width = 2pt] (10);
\path [blue] (10) edge (11);
\path [red] (11) edge [line width = 2pt] (12);
\path [blue, bend left = 60] (12) edge[->- = 0.6] (13);
\path [red] (13) edge [line width = 2pt] (14);
\path [blue] (14) edge (15);
\path [red] (15) edge [line width = 2pt] (16);
\path [blue] (16) edge (17);
\path [red] (17) edge [line width = 2pt] (18);

\path [blue, bend left = 30] (12) edge[->- = .5] (7);
\path [blue] (6) edge[->- = .6] (19);
\path [red] (19) edge [line width = 2pt] (20);
\path [blue] (20) edge[->- = .6] (21);
\path [red] (21) edge [line width = 2pt] (22);
\path [blue] (22) edge[->- = .6] (23);
\path [red] (23) edge [line width = 2pt] (24);
\path [blue] (24) edge[->- = .6] (13);
\path [blue, bend left = 30] (16) edge[->- = .5] (11);

\node at (2,-0.9) {$P_1'$};
\node at (4.5,0.4) {$P_1$};
\node at (9.5,-1.5) {$P_2'$};
\node at (14.5,0.4) {$P_2$};
\node at (13.5,-1) {$P_3'$};
\node at (11.5,0.4) {$P_3$};
\node at (7.5,0.4) {$P_4$};

\node at (9.5,-0.45) {$P_4'$};

 \end{tikzpicture}
  \end{center}
\caption{An example of alternating paths $P$ and $P'$, with $\ell=17, k=4, m=8$. $P'$ is divided into $2k+1=9$ segments (subpaths) $\{P_1',P_2,P_2',...,P_4',P_4,P_5'\}$, where in this example $P_5'$ is empty. The subpaths $P_1,P_2,P_3,P_4$ in the intersection $P\cap P'$ lie on the horizontal line; below this horizontal line are the segments $P_1',P_2',P_3',P_4'$ in $P'\backslash P$; above the horizontal line are the edges in $P\backslash P'$. From the agreement on the orientation, all the red edges are traversed from left to right in both $P$ and $P'$.}%
\label{fig:count.P'}%
\end{figure}
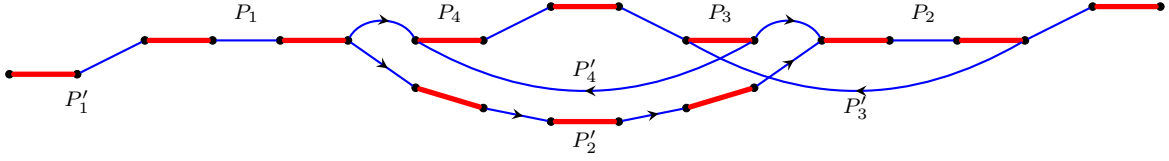

For the sake of enumeration, let us agree on the orientation so that all alternating paths with $2\ell+1$ edges start from a left vertex. This way, all the red edges are traversed from left to right in both $P$ and $P'$. Moreover, each $P_i$ starts from a left vertex, and each $P_i'$ starts from a right vertex for $2 \le i \le k$; 
$P_1$ and $P_{k+1}$ each starts from a left vertex.

Since $P$ is fixed, to specify $P'$, it suffices to specify the alternating paths $\calP'$ that constitute $P'\backslash P$, which
further reduces to  specifying for each $P'_i$ its start and end points, as well as the internal vertices.
	%
	%
\begin{itemize}
	\item We first specify the start and end points of $P'_i$ for $2 \le i \le k$, which are chosen from the vertices on $P$.
It suffices to specify $P_i$ for $1 \le i \le k$. 
For a given sequence of lengths $(m_1, m_2, \ldots, m_k)$, to specify $P_i$ 
for $1 \le i \le k$, it suffices to specify the starting point of each $P_i$. In total, there are at most $\ell^k$ choices for the starting points of $P_i,i\in[k]$. Next, 
since $\sum_{i=1}^k m_i= m$ and $m_i \geq 1$, there are at most $\binom{m-1}{k-1}$ choices of $(m_1, \ldots, m_k)$. 
Hence, in total there are at most 
$\binom{m-1}{k-1} \ell^k$ choices for the start and end points of $P'_i$ for $1 \le i \le k$.

\item
Next, we specify the length configuration $(m'_1, m'_2, \ldots, m'_{k+1})$ of $\calP'$. 
Since  $m'=\sum_{i=1}^{k+1} m'_i$,  $m'_1, m'_{k+1} \ge 0$,
and $m'_i \ge 1$ for $ 2 \le i \le k$, in total there are at most 
$\binom{m'+1}{k}$ choices for $(m'_1, m'_2, \ldots, m'_{k+1})$.

\item
Next we specify the vertices of $P'_1$  if $m'_1 \neq 0$. Note that the end point of $P'_1$ is on $P$ and has already been chosen. 
Thus it remains to choose the starting points of red edges in $P'_1$ which must be in the left vertex set per the agreed-upon orientation. 
Since there are $m'_1/2$ red edges in $P'_1$, there are at most $n^{m'_1/2}$ different choices for the vertices of $P'_1$ that are not on $P$. 
Analogously we can specify the vertices of $P'_{k+1}$ if $m'_{k+1} \neq 0$
and there are at most $n^{m'_{k+1}/2}$ different choices for the vertices of $P'_{k+1}$ that are not on $P$.

%

\item
Finally, we specify the internal vertices of $P'_i$ for each  $2 \le i \le k$. 
Note that there are $\sum_{i=2}^{k} (m_i'-1)/2$ red edges in $\{ P'_i: 2\le i \le k\}$. 
For each red edge, it suffices to specify its starting point, which must be in the left vertex set per the agreed-upon orientation. 
Thus in total there are at most $\prod_{i=2}^{k} n^{(m_i'-1)/2}$  
different choices for the interval vertices of $P'_i$ for $2 \le i \le k$.
\end{itemize}


Putting the above points together, we conclude that there are at most 
\begin{equation}
\ell^k \binom{m-1}{k-1}   \binom{m'+1}{k}
n^{ \frac{ m'- (k-1)}{2 } } 
\label{eq:counting}
\end{equation}
different choices of alternating paths $P'_1, \ldots, P'_{k+1}$ of total length $m'$.


Next we bound each conditional probability in \prettyref{eq:var1} from above.
Note that $P' \in S$ means it is $(a,b,\eta)$-light and $A$-uniform (see \prettyref{def:light_uniform}).
Since $P$ and $P'$ are overlapping, conditioned on the edge weights on $P$, 
the event that $P'$ is $(a,b,\eta)$-light  and the event that $P'$ is $A$-uniform are no longer independent. 
Nevertheless,  
these two events imply that 
the $i$-th alternating path $P'_i$ satisfies:
\begin{align*}
\left(  \left| r(P'_i)  \right|  - 2A \right)  \left( a - \frac{\eta}{2\ell} \right)    
& \le \wtr(P'_i)    \le \left(  \left| r(P'_i)  \right|  + 2A \right)   \left( a + \frac{\eta}{2\ell} \right)     \\
\left(  \left| b(P'_i) \right|  - 2A \right) 
 \left( b-  \frac{\eta}{2(\ell -1) } \right) & \le  \wtb(P'_i)  \le \left(  \left| b(P'_i) \right|  + 2A \right) 
 \left( b+ \frac{\eta}{2(\ell -1) } \right).
\end{align*}
Note that $\sum_{i=1}^{k+1} \left| r(P'_i)  \right| = \frac{m'-(k-1)}{2} \triangleq L $
and  $\sum_{i=1}^{k+1} \left| b(P'_i)  \right| = \frac{m'+(k-1) }{2} \triangleq M $.
Thus, summing over all $1 \le i \le k+1$  yields that 
\begin{align*}
\left(  L  - 2(k+1) A \right) 
 \left( a - \frac{\eta}{2\ell} \right)  & \le  \wtr(P' \backslash P ) \le \left(  L  + 2(k+1) A \right) 
 \left( a + \frac{\eta}{2\ell} \right), \\
 \left( M - 2 (k+1) A \right) 
 \left( b- \frac{\eta}{ (2\ell -1) } \right) & \le \wtb(P'\backslash P )  \le \left( M + 2 (k+1) A \right) 
 \left( b+ \frac{\eta}{ (2\ell -1) } \right).
\end{align*}
Let $\calE$ denote the event such that the last displayed equation holds. Then 
\begin{align*}
 & \prob{ P' \in S \;\big\vert \; P \in S } \\
& = \prob{ \text{$P'$ is  $(a,b,\eta)$-light and $A$-uniform} \mid P \in S }  \\
& \le \prob{ \calE \cap \left\{ \text{$P'$ is $(a,b,\eta)$-light and $A$-uniform} \right\} \mid P \in S }\\
 & = \prob{ \calE \cap \left\{ \text{$P'$ is $(a,b,\eta)$-light } \right\}  \mid P \in S } 
  \times \prob{  \text{$P'$ is $A$-uniform} \mid \calE, \text{$P'$ is $(a,b,\eta)$-light},  P \in S } \\
 & \le \prob{\calE  \mid P \in S }  \times \prob{  \text{$P'$ is $A$-uniform} \mid \calE, \text{$P'$ is $(a,b,\eta)$-light},  P \in S }.
 \numberthis \label{eq:disint1}
 \end{align*}
 We proceed to separately bound the two probability terms in the RHS of the  last displayed equation. 
First, we bound  $\prob{\calE  \mid P \in S }$. 
Note that the edge weights in $P'\backslash P$ are independent from the edge weights in $P$.
Thus, 
$\prob{\calE  \mid P \in S } = \prob{\calE}$. 
Let 
$$
g_\ell(x)=e^{-x} x^{\ell-1} / (\ell-1)! \indc{x\ge 0}
$$
denote the probability density function of $\text{Erlang}(\ell,1)$ (cf.~\prettyref{app:erlang}).
Then we have that for $L \ge 1$, 
\begin{align}
& \prob{   \left(  L  - 2(k+1) A \right)  \left( a - \frac{\eta}{2\ell} \right) \le  \wtr( P'\backslash P ) \le
 \left(  L  + 2(k+1) A \right)  \left( a + \frac{\eta}{2\ell} \right)  } \nonumber  \\
 & = \int_{ \lambda \left( L-2(k+1)A \right) \left( a - \frac{\eta}{2\ell} \right)}^{\lambda \left(  L  + 2(k+1) A \right)  \left( a + \frac{\eta}{2\ell} \right) }  
 g_L(x) \diff x. \nonumber  \\
 & \le  \frac{ \left( \lambda \left(  L  + 2(k+1) A \right)  \left( a + \frac{\eta}{2\ell} \right)  \right)^{L-1} }{(L-1)!} \exp \left(   - \lambda \left( L-2(k+1)A \right) \left( a - \frac{\eta}{2\ell} \right) \right) \nonumber  \\
 & \le \frac{\left( \lambda a L \right)^{L-1} }{(L-1)!} e^{-\lambda La} e^{2(k+1) A (1+\lambda a ) + \frac{\eta L}{2\ell a} (1+ \lambda a  )  }
 =\frac{\left( 2 L \right)^{L-1} }{(L-1)!} e^{-2 L} e^{6(k+1) A + \frac{3\eta \lambda L}{4\ell }   }
, \label{eq:prob_L}
\end{align}
where the inequality holds due to 
$$
\int_{u}^v g_L(x)  \diff x \le \frac{v^{L-1}}{(L-1)!} \int_{u}^v e^{-x}  \diff x \le \frac{v^{L-1}}{(L-1)!} e^{-u}.
$$
Similarly,
\begin{align}
& \prob{   \left( M - 2 (k+1) A \right) 
 \left( b- \frac{\eta}{ (2\ell -1) } \right) \le \wtb( P'\backslash P ) \le \left(  M+ 2(k+1) A \right)  \left( b + \frac{\eta}{2\ell -1} \right) } \nonumber  \\
& \le  \int_{0}^{ \left(  M+ 2(k+1)A \right) \left( b + \frac{\eta}{2\ell} \right) /n}    g_M ( x) \diff x\nonumber  \\
 & \le \frac{1 }{ n^{M} M! }  \left(  \left( M+ 2(k+1)A \right)  \left( b + \frac{\eta}{2\ell} \right)   \right)^{M} \nonumber  \\
 & \le \frac{ (Mb)^M }{ n^{M} M!} \exp \left( 2 (k+1) A + \frac{\eta M}{2b \ell} \right), \label{eq:prob_M}
\end{align}
where the second inequality holds because
$$
\int_0^{u} g_\ell(x) \diff x \le \frac{1}{(\ell-1)!} \int_0^u x^{u-1}  d x =  \frac{u^{\ell}}{\ell!} , \quad \forall u \ge 0.
$$
Note that $ \wtr( P'\backslash P )$ and $ \wtb( P'\backslash P )$ are independent.
Combining \prettyref{eq:prob_L} and \prettyref{eq:prob_M} gives that\footnote{We emphasize that it is crucial to keep the polynomials terms in \prettyref{eq:PE} so that
in \prettyref{eq:binom_bound_ell} we can get the upper bound $4 (2\ell)^{k-1}$, which in turn yields the desired  $\frac{\ell^2}{n}$ factor in~\prettyref{eq:final_second_estimate}.}
\begin{align*}
 \prob{\calE} & \le 
 \left( \frac{\left(2 L \right)^{L-1} }{(L-1)!} \indc{L\ge 1} + \indc{L=0}\right) e^{-2 L} e^{ 6(k+1) A + \frac{3\eta \lambda L}{4\ell }  }
 \times \frac{ (Mb)^M }{ n^{M} M!} \exp \left( 2 (k+1) A + \frac{\eta M}{2b \ell} \right) \\
 & \le   \left( \frac{\indc{L\ge 1}}{2\sqrt{L}}  + \indc{L=0} \right) \frac{1}{ \sqrt{M} } \left( \frac{2}{e} \right)^L
 \left(  \frac{eb}{n}  \right)^M \exp \left(  8(k+1) A  + 6 \eta \right)  \\
 & =  \left( \frac{\indc{k\le m'}}{ \sqrt{(m')^2-(k-1)^2} } + \frac{ \indc{k=m'+1} }{\sqrt{m'} } \right)   n^{ - \frac{m'+(k-1) }{2} } \left( 2 b  \right)^{m'/2}
 \left(  \frac{b e^{2} }{2}  \right)^{\frac{k-1}{2}}  \exp \left( 8 (k+1)A+ 6 \eta \right), 
\numberthis \label{eq:PE}
\end{align*}
where in the first  inequality we bound the LHS of \prettyref{eq:prob_L} by $1$ when $L=0$;
the second inequality follows due to $n! \ge \sqrt{n} (n/e)^n$,
$a \lambda = 2$, $b\lambda \ge 1$, and $\lambda \le 4$. 
\medskip

Next, we bound $\prob{  \text{$P'$ is $A$-uniform} \mid \calE, \text{$P'$ is $(a,b,\eta)$-light},  P \in S }$.
For ease of notation, let us denote $\calE'=\calE \cap \{ \text{$P'$ is $(a,b,\eta)$-light}\} \cap \{P \in S\}$.
Then we have 
\begin{align*}
&\prob{  \text{$P'$ is $A$-uniform} \mid \calE' } \\
\overset{(a)}{=} & \mathbb{E}\left[\prob{ \text{$P'$ is $A$-uniform} \mid \wtr(P'), \wtb(P'), \{W_e :e\in P\} } \mid \calE'\right]\\
\overset{(b)}{\leq}& \sup_{\alpha: |\alpha -a | \le \frac{\eta}{2\ell} } 
 \sup_{\beta: |\beta- b| \le \frac{\eta}{2(\ell-1) } } 
 \sup_{w_e: e\in P}   \prob{  \text{$P'$ is $A$-uniform} \mid  \wtr(P') = \alpha \ell, \wtb(P') = \beta \ell, \{ W_e=w_e: e \in P\}  },
\end{align*}
where $(a)$ is because $\calE'$ is measurable with respect to $\{\wtr(P'),\wtb(P'),\{W_e:e\in P\} \}$, and (b) is because on $\calE'$, $P'$ is $(a,b,\eta)$-light so that
$|\wtr(P') -a\ell | \le \eta/2 $ and $|\wtb(P')- b\ell | \le \eta/2$.

Fix any $\alpha$ such that $|\alpha -a | \le \frac{ \eta}{2\ell} $
and any $\beta$ such that $| \beta- b| \le \frac{\eta}{2(\ell-1) } $. 
Note that $\eta=1$, $a=2/\lambda$ and $b=(2-\zeta)/\lambda$ with $\zeta=\epsilon/4$
and $\lambda=4-\epsilon$. By choosing $\epsilon_0$ to be sufficiently small and $n_0$ to be
sufficiently large,
we have $1/4 \le \alpha, \beta \le 1$  for all $\epsilon \le \epsilon_0$ and $n \ge n_0$. 
We proceed to bound 
$\prob{  \text{$P'$ is $A$-uniform} \mid \wtr(P') = \alpha \ell, \wtb(P') = \beta \ell, \{ W_e: e \in P\} }$. 
Applying~\cite[Lemma 3.3]{ding2015percolation} (restated as \prettyref{lmm:uniform.conditional} in \prettyref{app:bridge}) with $\rho=\alpha$ or $\beta$, 
we get that 
\begin{align*}
& \prob{ \devr(P') \le A
  \mid    \wtr( P') = \alpha \ell ,  \{ W_e : e \in P \} }
  \le 
  \sqrt{m} e^{ c( A k + m/A^2)  }  \prob{ \devr(P') \le A } \\
 & 
  \prob{ \devb(P') \le A
  \mid    \wtb( P') = \beta \ell ,  \{ W_e : e \in P \} }
  \le 
  \sqrt{m} e^{ c( A k + m/A^2)  } \prob{ \devb(P') \le A },
\end{align*}
where $c$ is a universal constant.
Hence, we get that 
\begin{align}
& \prob{  \text{$P'$ is $A$-uniform} \mid \wtr(P') = \alpha \ell, \wtb(P') = \beta \ell, \{ W_e: e \in P\} } \nonumber \\
& \le m e^{ 2c( A k + m/A^2)  } p_\ell, \label{eq:Auniform-cond}
\end{align}
where $p_\ell=\prob{\text{$P'$ is $A$-uniform}}=\prob{ \devr(P') \le A } \times \prob{ \devb(P') \le A }$.

Combining \prettyref{eq:counting}, \prettyref{eq:disint1}, \prettyref{eq:PE}, and \prettyref{eq:Auniform-cond}, we get
\begin{align}
& \sum_{P' \in \calP_\ell: |P \cap P'| \ge 1, P \neq P'} 
\prob{\text{$P'$ is  $(a, b,\eta)$-light and $A$-uniform} \mid P \in S } \nonumber \\
& \le \sum_{m=1}^{2\ell-2} \sum_{k=1}^{m}    
\binom{m-1}{k-1}  \binom{m'+1}{k} \ell^k n^{ \frac{ m'- (k -1) }{2} }
\nonumber \\
& ~~~ \times
\left( \frac{\indc{k\le m'}}{ \sqrt{(m')^2-(k-1)^2} } + \frac{ \indc{k=m'+1} }{ \sqrt{m'} } \right) n^{ - \frac{m'+(k-1) }{2} } \left(2b   \right)^{m'/2}
 \left(  \frac{b e^{2} }{2}  \right)^{\frac{k-1}{2}}  e^{8(k+1) A + 6 \eta } \nonumber \\
 & ~~~ \times  m e^{ 2c( A k + m/A^2)  } p_\ell. \label{eq:variance_bound_1}
 \end{align}
 
 To further bound the RHS, we claim that when $\ell \ge e^2$, 
\begin{align}
  \binom{m'+1}{k}  \left( \frac{\indc{k\le m'}}{ \sqrt{(m')^2-(k-1)^2} } + \frac{ \indc{k=m'+1} }{ \sqrt{m'} } \right)   \le 4 (2\ell)^{k-1}.  \label{eq:binom_bound_ell}
 \end{align}
 To see this, note that if $k \ge (m'+1)/2$, then $\binom{m'+1}{k} \le \left(  \frac{e(m'+1)}{k} \right)^k \le (2e)^k$ and the claim holds as  $(m')^2 - (k-1)^2\ge 1$
 when $k \le m'$ and $m' \ge 1$. 
 If $k \le (m'+1)/2$, then $ \sqrt{(m')^2-(k-1)^2} \ge (m'+1)/4$.  Thus,
 $$
   \binom{m'+1}{k}  \frac{1}{ \sqrt{(m')^2-(k-1)^2} }   \le 4 (m'+1)^{k-1} \le 4 (2\ell)^{k-1}.
 $$
 Therefore, plugging \prettyref{eq:binom_bound_ell} into \prettyref{eq:variance_bound_1}, we get that 
\begin{align*}
& \sum_{P' \in \calP_\ell: |P \cap P'| \ge 1, P \neq P'} 
\prob{\text{$P'$ is  $(a, b,\eta)$-light and $A$-uniform} \mid P \in S } \\
& \leq 4 p_\ell \ell c_1 \e^{ (16 +2c) A  + 6 \eta } \\
&~~~\times \sum_{m=1}^{2\ell-1}  m  \left( 2  b  \right)^{\ell-1/2-m/2}  e^{2c m/A^2}
\sum_{k=1}^{m}  \binom{m-1}{k-1}
\left( \frac{2\ell^2 \sqrt{b} }{n \sqrt{2} } e^{(8+2c) A +1} \right)^{k-1} \\
 & = 4 p_\ell \ell  \e^{  A (16+2c ) + 6 \eta + 2c/A^2 } \left( 2b\right)^{\ell-1}
 \sum_{m=1}^{2\ell-1} m \kappa^{m-1},
\end{align*}
where
\begin{align}
\kappa \triangleq   \frac{e^{2c/A^2} }{ \sqrt{ 2b  }} \left( 1 + \frac{\ell^2 \sqrt{2b} }{n } e^{(8+2c)A + 1 }   \right). \label{eq:def_kappa}
\end{align}
Note that $b=\frac{2-\zeta}{\lambda}$ with $\zeta=\frac{\epsilon}{4}$ 
and $\lambda=4-\epsilon$. Also, 
$
A = \ceil{\frac{C_1}{\sqrt{\epsilon}}}
$
and $A^2 \le \ell \le \sqrt{n} e^{-C_2 A}$,
where $C_1, C_2 $ are absolute constants. 
By choosing $C_1, C_2, n_0$ to be sufficiently large 
and $\epsilon_0$ to be sufficiently small,  
we get that 
$\kappa \le 1- \epsilon/100$  for all $n \ge n_0$ and $\epsilon \ge \epsilon_0$.
Hence, 
$$ 
\sum_{m=1}^{\infty} m \kappa^{m-1}  \le \left( 
\sum_{m=1}^{\infty} \kappa^{m}  \right)' = \left( 
\frac{\kappa}{1-\kappa} \right)'= \frac{1}{(1-\kappa)^2}
\le 10^4 \epsilon^{-2}.  
$$
Recall that $\eta=1$. Then we conclude that for universal constants $c_2, c_3>0$, 
\begin{align}
& \sum_{P' \in \calP_\ell: |P \cap P'| \ge 1, P \neq P'}
\prob{\text{$P'$ is  $(a,b,\eta)$-light and $A$-uniform} \mid  P \in S } \nonumber  \\
 & \le  p_\ell \ell \e^{ c_2/\sqrt{\epsilon} } \left( 2 b \right)^{\ell}
 \le \frac{  \ell^2 e^{c_3/\sqrt{\epsilon} } }{ n } \expect{|S|} 
 \label{eq:final_second_estimate}
\end{align}

Substituting the last displayed equation back to \prettyref{eq:var1}, we get
$$
\var(|S|) \le \expect{|S|}
\times 
\left( 1+   \frac{ \ell^2 e^{c_3/\sqrt{\epsilon}} }{ n} 
\expect{|S|}  \right).
$$
\end{proof}

\subsection{Extracting many vertex-disjoint alternating paths via Tur\'an's Theorem} \label{sec:Turan}
To extract many disjoint alternating paths from $S$, following~\cite{ding2015percolation}, we first construct a graph under which
an independent set corresponds to a set of disjoint alternating paths, and then
use Tur\'an's theorem to prove the existence of a large independent set. 

Specifically, we define a graph $H$ on the set $S$ of alternating paths, where $P$ and $P'$ are adjacent 
if they share at least one common vertex. 
Thus any collection of vertex-disjoint paths is an independent set in $H$.
 The following result due to Tur\'an (see e.g.~\cite[Theorem 1, p.~95]{Alon2016}), 
provides a lower bound to the size of the largest independent set in a general graph.

\begin{lemma}[Tur\'an's Theorem]
\label{lmm:turan}
Let $G=(V, E)$ be a finite, simple graph. Then $G$ contains an independent subset of size at least 
$|V|^2 / ( 2|E| + |V|) $. 
\end{lemma}

Next we lower bound the number of vertices $|V(H)|$ and 
upper bound the number of edges $|E(H)|$ using~\prettyref{lmm:impossible_first_moment}
and \prettyref{lmm:impossible_second_moment},
and then apply Tur\'an's theorem to finish the proof of
\prettyref{thm:disjoint_paths}.

\begin{proof}[Proof of \prettyref{thm:disjoint_paths}] 
Throughout the proof, $A = \ceil{\frac{C_1}{\sqrt{\epsilon}}}$, $ A^2  \le \ell \le e^{-C_2 A} \sqrt{n} $,  
$\eta=1$, $a=2/\lambda$, $b=\frac{2-\zeta}{\lambda}$ with $\lambda=4-\epsilon$ and 
$\zeta = \frac{\epsilon}{4}$.

We start by verifying that, by definition of $S$, each $(2\ell-1)$-alternating path $P \in S$ satisfies the subpath requirement in \prettyref{eq:wQ}.
Recall that $P$ is oriented so that it starts with a left vertex. In this order, let $\phi_1, \phi_2, \ldots, \phi_\ell$ denote the sequence of red edge weights 
and $\psi_1, \psi_2, \ldots, \psi_{\ell-1}$ the blue edge weights.
Fix any $\ell'$ such that $\ell/3 \le \ell' \le \ell$
and a $(2\ell'-1)$-alternating subpath $Q$ of $P$ that has $\ell'$ red edges and $\ell'-1$ blue edges. 
Then $Q$ consists of red edge weights $\phi_{k+1}, \ldots, \phi_{k+\ell'}$
and blue edge weights $\psi_{k+1}, \ldots, \psi_{k+\ell'-1}$ in order for some
$0\le k \le \ell-\ell'$.  In particular,
\begin{align*}
\wtr(Q) & = \sum_{i=1}^{\ell'} \phi_{k+i} = \sum_{i=1}^{k+\ell'} \phi_i  - \sum_{i=1}^k \phi_i \\
\wtb(Q) & = \sum_{i=1}^{\ell'-1 } \psi_{k+i} = \sum_{i=1}^{k+\ell'-1 } \psi_i - \sum_{i=1}^{k} \psi_i.
\end{align*}
Since $P$ is $(a, b,\eta)$-light and $A$-uniform with $\eta=1$, 
it follows that $ \sum_{i=1}^{k+\ell} \phi_i \ge \left( k+\ell' - A \right) \frac{\wtr(P)}{\ell}$, $\sum_{i=1}^{k} \phi_i \le \left( k + A \right) \frac{\wtr(P)}{\ell}$, 
and hence
$$
\wtr(Q)  \ge \left( \ell' - 2A \right) \frac{\wtr(P)}{\ell}  \ge  \left( \ell' - 2A \right)  \left( a -  \frac{1}{2\ell} \right). 
$$
Analogously, we have
\begin{align*}
\wtb(Q) & \le  \left( \ell' + 2 A \right) \left( b +  \frac{1}{2(\ell-1)}   \right). 
\end{align*}
Since $\log \frac{\calP}{\calQ} (W_e) = \log (n\lambda) - (\lambda-1/n) W_e$, it follows that 
\begin{align}
\Delta(Q) & = - (\lambda-1/n)  \left[ \wtb(Q)- \wtr(Q)  \right]  \nonumber \\ 
& \ge - (\lambda-1/n) \ell' \left( b - a + \frac{1}{2\ell} + \frac{1}{2(\ell-1)}   \right)
- 2(\lambda-1/n) A \left( b +  \frac{1}{2(\ell-1)} + a - \frac{1}{2\ell} \right) \nonumber \\
&  \ge - (\lambda-1/n)\left[ - \ell' \frac{\epsilon}{4\lambda} + \frac{8A}{\lambda}  + 2  \right] 
\ge  (\lambda-1/n) \frac{\epsilon}{96} \ell', \label{eq:verify_Delta_Q}
\end{align}
where the last inequality holds as $\ell' \ge \ell/3$, $\lambda=4-\epsilon$, and 
$\ell \ge A^2$, and $A \ge C_1/\sqrt{\epsilon}$.

Next we apply Tur\'an's theorem to the graph $H$ to extract a vertex-disjoint subcollection $S^*$ of $S$.
In view of \prettyref{eq:firstmoment-crude} and $A = \ceil{\frac{C_1}{\sqrt{\epsilon}}}$, by choosing $C_1, n_0$ to be sufficiently large and $\epsilon_0$ to be sufficiently small,
for all $0<\epsilon \le \epsilon_0$ and $n \ge n_0$, 
$$
\expect{|V(H)|}=\expect{ \left| S \right| } \ge \frac{c_4 n}{\ell} e^{ c_4 \epsilon \ell } 
$$
for a universal constant $c_4>0$. 
By Chebyshev's inequality and \prettyref{lmm:impossible_second_moment},
$$
\prob{ |S| \le \frac{1}{2} \expect{|S|} }
\le \frac{4 \var(|S|) }{  \left( \expect{|S|} \right)^2 }
\leq \frac{ \ell^2 e^{c_2/\sqrt{\epsilon}} }{n} 
$$
for a universal constant $c_2>0$. 
By \prettyref{eq:final_second_estimate},
\begin{align*}
\expect{\left|E(H) \right|}
&= \sum_{P \in \calP_\ell } \prob{  P \in S }
\sum_{P' \in \calP_\ell: |P \cap P'| \ge 1, P'\neq P} \prob{ P' \in S \;\big\vert \; P \in S }  \\
& \le  \frac{ \ell^2 e^{c_2/\sqrt{\epsilon} } }{ n}  \left( \expect{|S|} \right)^2.
\end{align*}
By Markov's inequality, 
$$
\prob{\left|E(H) \right| \ge 2 \expect{\left|E(H) \right|} } \le 1/2.
$$
Hence, with probability at least $1/2- \frac{ \ell^2 e^{c_2/\sqrt{\epsilon}} }{n} $, we have that 
$|S| \ge \frac{1}{2} \expect{|S|} $ and $|E(H) | \le 2\expect{|E(H) |}$, so that \prettyref{lmm:turan} implies the existence of $S^* \subset S $ of disjoint alternating paths 
such that 
$$
\left| S^* \right| \ge \frac{ \frac{1}{4} \left( \expect{|S|} \right)^2 }{ 4 \expect{|E(H) |} + \frac{1}{2} \expect{|S| } }
\ge  \frac{n}{\ell^2e^{c_3/\sqrt{\epsilon} }  } 
$$
for a universal constant $c_3>0$. 
\end{proof}


\appendix

\section{Large deviation estimates}
\label{app:LD}
\begin{lemma}
Let $\calP,\calQ$ be two probability distributions such that $\calP\ll \calQ$. Let $X_i$'s and $Y_i$'s be two independent sequences of random variables, where $X_i$'s are i.i.d.\ copies of $
\log(\calP/\calQ)$ under distribution $\calP$ and $Y_i $'s are i.i.d.\ copies of $\log(\calP/\calQ)$ under
distribution $\calQ$. For all $x\geq 0$ and positive integer $\ell$, we have
\begin{equation}
\prob{\sum_{i=1}^\ell (Y_i - X_i) \geq x\ell }\leq \exp \left(-\ell (\alpha+x/2) \right),
\label{eq:LDXY}
\end{equation}
where $\alpha = -2\log B(\calP,\calQ)$ as defined in~\eqref{eq:alpha}. Furthermore, if we further assume that $\calQ\ll\calP$, then for all $0\leq x\leq D(P\| Q)+D(Q\| P)$ and positive integer $\ell$, we have
\begin{equation}\label{eq:LD.reverse}
\prob{\sum_{i=1}^\ell (Y_i - X_i) \geq x\ell } \geq \exp\left(-\ell (\alpha +x+o(1))\right).
\end{equation}
\end{lemma}
\begin{proof}
The proof of~\eqref{eq:LDXY} follows from standard large deviation analysis (cf.~\cite[Appendix B]{bagaria2020hidden}).
Let $F$ denote the Legendre transform of the log moment generating function of $Y_1-X_1$, i.e.,
\[
F(x) =  \sup_{\theta \ge 0} \sth{ \theta x  - \psi_\calP(-\theta) -\psi_\calQ(\theta) },
\]
where $\psi_\calP(\theta)=\expect{e^{\theta X_1}}$ and 
$\psi_\calQ(\theta)=\expect{e^{\theta Y_1}}$.
Then from the Chernoff bound we have the following large deviation inequality:
\begin{equation}\label{eq:LD.plain}
\prob{\sum_{i=1}^\ell (Y_i - X_i) \geq x\ell }\leq \exp \left(-\ell F(x) \right).
\end{equation}
Note the following facts:
\begin{enumerate}
	\item $F(0) = -  \psi_\calP( -1/2) - \psi_\calQ( 1/2)= - 2 \log \int \sqrt{\calP \calQ}=\alpha$;
	\item $F(x) \geq F(0) +x/2$, for all $x\geq 0$.
\end{enumerate}
Combining these facts with~\eqref{eq:LD.plain} yields~\eqref{eq:LDXY}.

Next we prove~\eqref{eq:LD.reverse}. By Cram\'er's theorem, we have 
\begin{equation}
\label{eq:cramer}
\prob{\sum_{i=1}^\ell (Y_i - X_i) \geq x\ell }= \exp \left( -F(x)\ell + o(\ell)  \right),
\end{equation}
where $o(\ell)$ converges to $0$ as $\ell$ grows to infinity. 
Second, we have 
$$
F(x) \le \alpha  + x , \quad \forall 0 \le x \le D(\calP\|\calQ) + D(\calQ\|\calP). 
$$
To see this, note that $\psi_\calP(\theta) =  \psi_\calQ(1+\theta)$. Thus, the optimal $\theta$ is given by 
$$
x + \psi'_\calQ(1-\theta) - \psi_\calQ'(\theta)  =0. 
$$
Note that $\psi_\calQ'(0)=-D(\calQ\|\calP)$ and $\psi_\calQ'(1)=D(\calP\|\calQ)$. Moreover, since $\psi_\calQ(\theta)$
is convex, it follows that 
$\psi_\calQ'(\theta)$ is non-decreasing in $\theta$. 
Thus the optimal $\theta$ must lie in $[1/2, 1]$ when $0 \le x \le D(\calP\|\calQ) + D(\calQ\|\calP)$. Hence,
\[
F(x) = \sup_{\theta \in [1/2, 1] } \sth{ \theta x  - \psi_\calP(-\theta) -\psi_\calQ(\theta) }
\le x + \sup_{\theta \in [1/2, 1] } \sth{- \psi_\calP(-\theta) -\psi_\calQ(\theta) } 
=x+F(0)
= x+ \alpha .
\]
Combine with~\eqref{eq:cramer} to finish the proof of~\eqref{eq:LD.reverse}.
\end{proof}

\section{Erlang distribution and Chernoff bounds}
\label{app:erlang}
The sum of $\ell$ \iid $\exp(\lambda)$ random variables has an Erlang distribution with parameters $\ell$ and $\lambda$, 
denoted by $\text{Erlang}(\ell,\lambda)$, whose density is given by 
\begin{equation}
f(x) = \frac{\lambda^{\ell} x^{\ell-1} \e^{-\lambda x}}{(\ell-1) !}, \quad x \geq 0
\label{eq:erlang-pdf}
\end{equation}

\begin{theorem}\label{thm:Chernoff_Erlang}
Let $X_i \iiddistr \exp(1)$. Then
\begin{align*}
& \prob{\sum_{i=1}^n X_i \ge n \xi } \le  \exp \left( -n \left( \xi - \log (\xi ) -1 \right) \right), \quad \forall \xi >1  \\
& \prob{\sum_{i=1}^n X_i \le n \xi } \le  \exp \left( -n \left( \xi - \log (\xi ) -1 \right) \right), \quad \forall \xi <1 
\end{align*}
\end{theorem}

\section{Exp-minus-one random bridge}
\label{app:bridge}
Let $X_1, X_2, \ldots, X_\ell \iiddistr \exp(\mu)$ and let $X=\sum_{i=1}^\ell X_i$. 
Recall from \prettyref{lmm:expwalk} the exp-minu-one $\ell$-bridge $R$ is defined as 
$$
R_j=\sum_{i=1}^j \left( \frac{X_i}{X} \ell  -1 \right), \quad 0 \le j \le \ell.
$$
Define $\dev(R)=\max_{0 \le j \le \ell} \left| R_j \right| $. 
The following result adapted from \cite[Lemma 3.3]{ding2015percolation} 
(which is a slight extension of~\cite[Lemma 3.2]{ding2013scaling})
bounds the probability of $\dev(R) \le A$
conditional on the total weight of $X_i$ and the value of $X_i$
for a set of indices $i$ in a union of intervals,  
in terms of the unconditional probability $p_\ell \triangleq \prob{ \dev(R) \le A}$. 
This result is crucial for the second moment computation in \prettyref{sec:second}.


\begin{lemma}
\label{lmm:uniform.conditional}
Let $1/4 \le \rho \le 1$ and $1 \le A \rho \le \sqrt{\ell}$. 
Consider the integer intervals $[a_1, b_1], \ldots, [a_k, b_k]$
such that $1 \le a_1 \le b_1 \le \cdots \le a_k \le b_k \le \ell$
and $m=\sum_{i=1}^k (b_i-a_i +1)$. 
Write $J= \cup_{i=1}^k [a_i, b_i]$. Then
$$
\prob{ \dev(R)  \le A \mid \sum_{i=1}^\ell X_i = \rho \ell, \{ X_j= x_j, j \in J \} } 
\le c_1 A \sqrt{m \wedge (\ell-m) } p_\ell 10^{100 k A} e^{c_0 m/A^2},
$$
where $c_0, c_1>0$ are two universal constants. 
\end{lemma}

\newcommand{\dV}{\mathrm{d}V}
\newcommand{\dW}{\mathrm{d}W}
\newcommand{\dU}{\mathrm{d}U}
\newcommand{\dtU}{\mathrm{d}\tilde{U}}

\section{Minimum-weight matching for the exponential model}
\label{app:ode}
In this section, we prove the positive part \prettyref{eq:risk-ub-exp} in \prettyref{thm:exp}. Namely, in the complete graph case with exponentially distributed weights,
when $\lambda=4-\epsilon$, 
the minimum weighted matching $\Mmin$ (linear assignment), which corresponds to the maximum likelihood estimation, 
misclassifies at most $O\left( \frac{1}{\epsilon^3} \e^{-\frac{2\pi}{\sqrt{\epsilon}}} \right)$ fraction of edges on average.  
This together with the negative part \prettyref{eq:risk-lb-exp} in  \prettyref{thm:exp} shows that  the minimum weighted matching achieves the optimal rate $1/\sqrt{\epsilon}$ of the error exponent. 
Prior work~\cite{Moharrami2020a} 
provides the exact characterization of the asymptotic error of $\Mmin$ in terms of a system of ordinary differential equations when $\lambda<4$.
Our proof follows by analyzing this system of ODEs when $\lambda=4-\epsilon$ for small $\epsilon$, 
and is inspired by the heuristic arguments in  \cite[Section VI]{Semerjian2020}.


\begin{proof}[Proof of \prettyref{thm:exp}: positive part]
First, it has been shown in~\cite[Theorem 2]{Moharrami2020a}  that 
$$
\lim_{n\to \infty}  \expect{\ell \left(M^*,\Mmin \right)}
=4 \int_0^\infty \left(1-U(x) V(x) \right) \left(1-\left( 1- U(x) \right) W(x) \right) V(x) W(x) \,dx,
 $$
 where $(U, V, W)$ is the unique solution to the following system of equations 
 \begin{equation}
\begin{aligned}
&\frac{\dU}{\dx}  = - \lambda U(1-U ) + (1- UV)  \left(1- (1-U) W \right)  \\
&\frac{\dV}{\dx}   = \lambda V (1-U)  \\
& \frac{\dW}{\dx} = - \lambda W U 
\end{aligned}\label{eq:ODE_UVW}
\end{equation}
with initial condition
\begin{align}
& U(0) = \frac{1}{2}, \quad V(0)=W(0)=\delta, \quad \delta \in (0,1) \label{eq:ODE_UVW_initial},
\end{align}
and $\delta$ is the unique value in $(0,1)$ such that $U(x), V(x) \to 1$ as $x \to +\infty$. 

Furthermore, it has been shown in~\cite[Section B]{Moharrami2020a} that $UV<1,$ $(1-U) W<1,$ $0<U,V,W<1$, 
and 
\begin{align}
V(x) &=\delta \exp \left(  \lambda \int_{0}^x \left(1- U(y) \right) \dy \right)  , 
\label{eq:V_expression}\\
W(x) & = V(x) \,\e^{-\lambda x}. 
\label{eq:W_expression}
\end{align}

Therefore, we have  that 
$$
\lim_{n\to \infty}  \expect{\ell \left(M^*,\Mmin \right)}
\le 4 \int_0^\infty  V^2(x) e^{-\lambda x} \,dx,
 $$

 
 Let $x_0=\inf\{x\geq 0: U(x) < \frac{1}{2}\}$ denote the first time that $U(x)$ crosses $1/2$.
Therefore by~\prettyref{eq:V_expression} we have that for all $0 \leq x \le x_0$,  $V(x) \leq  \delta e^{\lambda x/2}$. 
Hence for any $0 \le \tau \le x_0$,
\begin{align*}
 \int_0^\infty  V^2(x) e^{-\lambda x} \,dx
& = \int_0^{\tau}  V^2(x) e^{-\lambda x} \,dx  + \int_{\tau}^{\infty}  V^2(x) e^{-\lambda x} \,dx \\
& \le \delta^2 \tau +  \int_{\tau}^{\infty}   e^{-\lambda x} \,dx \\
& =  \delta^2 \tau + \frac{1}{\lambda} e^{-\lambda \tau}. 
\end{align*}
We claim that $\lambda x_0 \geq 2 \log \frac{\epsilon}{8\delta}$. 
Suppose not. Then $e^{\lambda x_0/2} < \frac{\epsilon}{8\delta}$, in particular, $x_0<+\infty.$
By continuity, we have $U(x_0)=\frac{1}{2}$, $U'(x_0) \leq 0$. But we have
\begin{align*}
U'(x_0) = & ~  -\lambda U(x_0) \left(1-U(x_0) \right) + \left(1-U(x_0) V(x_0) \right) \left( 1-\left(1-U(x_0) \right)W(x_0) \right)\\
= & ~ -\lambda/4 + \left( 1- \frac{1}{2} V(x_0) \right)  \left( 1- \frac{1}{2} W(x_0) \right ) \\
\geq & ~ -\lambda/4 + 1 - V(x_0) \\
\geq & ~  \epsilon/4  - \delta e^{\lambda x_0/2} > \epsilon/8 > 0,
\end{align*}
which is the needed contradiction.
Therefore, by setting $\tau= \frac{2}{\lambda} \log \frac{\epsilon}{8\delta}$, 
we get that 
\begin{align*}
 \int_0^\infty  V^2(x) e^{-\lambda x} \,dx
 \le  \frac{2 \delta^2}{\lambda}   \log \frac{\epsilon}{8\delta} + \frac{1}{\lambda}  
 \left( \frac{8\delta}{\epsilon} \right)^2 
 \le  \frac{c}{\epsilon^3} \exp \left( - \frac{2\pi}{\sqrt{\epsilon}} \right),
 \end{align*}
where $c>0$ is a universal constant and the last inequality follows from the claim that 
\begin{align}
\delta  \le \frac{c'}{\sqrt{ \epsilon} } \exp \left( - \frac{\pi}{\sqrt{\epsilon}} \right), \label{eq:delta_bound}
\end{align}
where $c'>0$ is a universal constant. 

It remains to prove \prettyref{eq:delta_bound}.
In view of $V<1$ and \prettyref{eq:V_expression}, we have that 
\begin{align}
\delta \exp \left(  \lambda \int_{0}^\infty \left(1- U(y) \right) \dy \right) \le 1. \label{eq:delta_upper_1}
\end{align}
To proceed, we derive an upper bound to $U(x)$. Let $\tU(x)$ denote the unique solution of the
following ODE:
\begin{align}
\frac{\dtU}{\dx}  = - \lambda \tU(1-\tU ) +  1, \quad \tU(0) = \frac{1}{2}.  \label{eq:ode_U_upper}
\end{align}
We claim that $U(x) \le \tU(x)$ for all $x \ge 0$. 
To show this, let $f=\tU-U$. 
Then we have (a) $f(0)=0$; (b) $f'(0)>0$; (c) $f'(x)>0$ whenever $f(x)=0$. 
Thus $f(\delta_0)>0$ for some small $\delta_0$. 
Let $x_1 = \inf\{x \geq 0: f(x)<0 \}$. 
Suppose for the sake of contradiction that $x_1<\infty$. 
Then $x_1$ is the first time that $f$ crosses zero. Thus $f'(x_1) \le 0$. 
But by continuity we have $f(x_1)=0$ and hence $f'(x_1)>0$, which is a contradiction.

Solving the ODE~\prettyref{eq:ode_U_upper}, we get that 
$$
\tU(x) = \frac{1}{2} + \frac{1}{2} \sqrt{ \frac{4-\lambda}{\lambda} } \tan \left( \frac{x}{2} \sqrt{\lambda(4-\lambda) } \right), \quad 0 \le x < \frac{\pi}{\sqrt{\lambda(4-\lambda) }}.
$$
Therefore, for some $0 \le a<\frac{\pi}{\sqrt{\lambda(4-\lambda) }}$ to be determined, 
\begin{align*}
\int_{0}^\infty \left(1- U(y) \right) \dy 
& \ge \int_{0}^{a} \left(1- U(y) \right) \dy  \\
& \ge \int_{0}^{a} \left( 1- \tU(y) \right) \dy \\
& = \frac{a}{2} + \frac{1}{\lambda} \log \cos \left(  \frac{a}{2} \sqrt{\lambda(4-\lambda) } \right),
\end{align*}
where the last inequality holds due to $\int_{0}^x \tan (y) \dy = - \log \cos(y)$ for $0 \le x <\pi/2$. 
Combining the last displayed equation with \prettyref{eq:delta_upper_1} yields that
$$
\delta \le  \frac{1}{ \cos \left(  \frac{a}{2} \sqrt{\lambda(4-\lambda) } \right)} \exp \left( - \lambda \frac{a}{2}  \right).
$$
Recall that $\lambda=4-\epsilon$. 
Choose $a = \frac{ \pi -\sqrt{\epsilon} }{ \sqrt{\lambda (4-\lambda)} }$. Then we get
$$
\delta \le \frac{1}{\sin \left( \frac{ \sqrt{ \epsilon} }{2} \right) } 
\exp \left(  - \sqrt{\lambda}  \frac{\pi - \sqrt{\epsilon} }{2 \sqrt{\epsilon} }  \right)
\le  \frac{c'}{\sqrt{ \epsilon} } \exp \left( - \frac{\pi}{\sqrt{\epsilon}} \right),
$$
for a universal constant $c'>0$. 
\end{proof}

\section{Reduction arguments}\label{app:reduction}
\subsection{Reduction from dense to sparse model}
\label{app:reduction_dense_sparse}
In this subsection, we prove 
 \prettyref{thm:impossibility} for the dense model
by reducing it to the sparse model. Suppose \prettyref{thm:impossibility} holds under the sparse model.
Recall that in the dense model, 
the planted weight density $p(x)$ and the null weight density $q(x)=\frac{1}{d} \rho(\frac{x}{d})$, 
where  $p$ and $\rho$ are fixed probability densities on $\reals$ and $d \to \infty$. 
In this case the impossibility condition~\prettyref{eq:impossibility} simplifies to:
	\begin{equation}
	\int_{-\infty}^\infty \sqrt{  p(x) \rho(0)   }  \diff x \geq 1+\epsilon.
	\label{eq:threshold-dense}
	\end{equation}
	
Define a sparse model with parameter $(d',\calP ',\calQ ')$, where
\[
d'=\Gamma\rho(0)\left(1-\frac{\epsilon}{2}\right)^2, \quad \calP '=\calP, \quad \calQ '=\text{Unif}[-\Gamma/2,\Gamma/2], 
\]
and
for some positive constant $\Gamma$ to be specified. 
Note that $d',\calP ',\calQ '$ are all independent of $n$.
Next we show that, given graph $G'$ drawn from the $(d',\calP ',\calQ ')$ model with planted matching $M^*$, there exists a (randomized) mapping $f$ 
such that $G\triangleq f(G')$ is distributed according to the $(d,\calP,\calQ)$ model with the same planted matching $M^*$. 
Note that crucially such a mapping needs to be agnostic to the latent $M^*$.

To construct the mapping $f$, we first express $\calQ$ as a mixture of $\calQ '$ and some other distribution $\overline{\calQ}$,
that is $\calQ=t \calQ '+(1-t)\overline{\calQ}$, where 
$t=\frac{d'}{d}=\frac{\Gamma}{d}\rho(0)\left(1-\frac{\epsilon}{2}\right)^2$ and $\overline{\calQ}$ has density $\overline{q}(x) \triangleq \frac{ q(x)- tq'(x) }{1-t}$.
We claim that $q'$ is well-defined density for large $d$. 
Indeed, note that $q(x)=\frac{1}{d}\rho(\frac{x}{d})$, $q'(x)=\indc{|x|\leq \Gamma/2}/\Gamma$ and $d\diverge$.
By continuity of $\rho$ at $0$, $\rho(0)(1-\frac{\epsilon}{2})^2\leq \rho(\frac{x}{d})$ holds for all $|x| \leq \Gamma/2$, provided that $d$ is sufficiently large.
Thus $tq'(x) \leq q(x)$ for all $x$.

Next, given $G'$, we generate a denser graph $G=f(G')$ as follows. 
For each edge $e$ in $G'$, we leave its edge weight unchanged. 
For each edge $e$ not in $G'$, we connect it in $G$ independently with probability $r/n$ and draw its edge weight $W_e \sim \overline{\calQ}$, where $r\triangleq \frac{d-d'}{1-d'/n}$. The choice of $r$ is such that for each $e\in [n]\times [n]'\backslash M^*$, the probability that $e$ is an unplanted edge in $G$ (namely, $e\in E(G)\backslash M^*$) equals $(1-\frac{d'}{n})\frac{r}{n}+\frac{d'}{n}=\frac{d}{n}$.
Furthermore, for each unplanted $e$ in $G$, 
it is an unplanted edge in $G'$ 
with probability $\frac{d'/n}{d/n}=\frac{d'}{d}=t$. As a result, 
conditioned on $e$ being an unplanted edge in $G$, the distribution of $W_e$ is $t \calQ' + (1-t) \overline{\calQ}= \calQ$. Also, by construction, conditioned on the true matching $M^*$ of $G'$, the edge weights of $G$ are independent.
In other words, $G=f(G')$ is distributed according the dense $(d,\calP,\calQ)$ model with $M^*$ being the planted matching.
Thus, it remains to verify the impossibility for the sparse $(d',\calP ',\calQ ')$ model by verifying the condition of Theorem~\ref{thm:impossibility}. Indeed,
\begin{align*}
\sqrt{d'} B(\calP',\calQ')= & \sqrt{td}\int_{-\Gamma/2}^{\Gamma/2}\sqrt{p(x)\frac{1}{\Gamma}}dx = \left(1-\frac{\epsilon}{2}\right)\sqrt{\rho(0)}\int_{-\Gamma/2}^{\Gamma/2}\sqrt{p(x)}dx\geq 1+\frac{\epsilon}{4},
\end{align*}
where the last inequality holds as a consequence of the condition~\eqref{eq:threshold-dense} and by choosing a large enough $\Gamma$.
This completes the proof of Theorem~\ref{thm:impossibility} for dense models.

\subsection{Reduction for general weight distributions: positive result}
\label{app:acpos}
In Section~\ref{sec:positive}, we proved~\prettyref{thm:first_moment_bound} under the assumption $\calP\ll\calQ$. Here, we prove the theorem for general $\calP,\calQ$. Note that our main positive result~\prettyref{thm:possibility} follows directly from~\prettyref{thm:first_moment_bound}, per the argument in \prettyref{sec:positive}. 

Recall that in \prettyref{eq:B}, $f$ and $g$ denote the densities of $\calP$ and $\calQ$ with respect to a common dominating measure 
$\mu$, respectively. Under this general model, the maximum likelihood estimator $\Mmin$ takes the form
\[
\Mmin\in\arg\max_{M\in \calM} \prod_{e\in M}f(W_e)\prod_{e\notin M}g(W_e).
\]
Let $p=\calP\{g>0\}$. If $p=0$, then $\calP$ and $\calQ$ are mutually singular. In this case, it is easy to see that $\Mmin=M^*$ with probability one.
Next, suppose $p>0.$
Let $\calP'$ denote the distribution $\calP$ conditioned on the support of $\calQ$, with density $f'=f\indc{g>0}/p$. 
Then we have $\calP'\ll\calQ$. 

Let $V=\{i\in [n]:g(W_{i,i'})>0\}$. Therefore, $g(W_{i,i'})=0$ for all $i\in V^c$. We first observe that if a perfect matching $M\in \calM$ does not contain all the red edges $\{(i,i'): i\in V^c\}$, then it has a zero likelihood. 
Thus, $\Mmin$ reduces to:
\begin{equation}\label{eq:ML.decompose}
\Mmin =\left\{(i,i'):i\in V^c\right\} \cup \widehat{M}_{\mathsf{ML},V},
\end{equation}
where for $\calM_V$ defined as the set of perfect matchings on $V\times V'$, 
\[
\widehat{M}_{\mathsf{ML},V}\in \arg\max_{M\in \calM_V} \prod_{e\in M}f(W_e)\prod_{e\notin M}g(W_e).
\]
In other words, $\widehat{M}_{\mathsf{ML},V}$ is the maximum likelihood estimator over the subgraph $G'=G[V\times V']$. Moreover, conditional on the set $V$, all the red edges weights on $G'$ are {\it i.i.d.} following distribution $\calP'$, and all the blue edge weights are ${\it i.i.d.}$~according to $\calQ$. Denote $M^*_V=\{(i,i'):i\in V\}$ for the true matching on $G'$, and let $n'=|V|$. We can therefore apply~\prettyref{thm:first_moment_bound} on $G'$ with $B(\calP',\calQ)=B(\calP,\calQ)/\sqrt{p}$. Under the condition $\sqrt{n}B(\calP,\calQ)\leq 1+\epsilon$, we have
\[
\sqrt{n'}B(\calP',\calQ)=\sqrt{n}B(\calP,\calQ)\sqrt{\frac{n'}{np}}\leq (1+\epsilon)\sqrt{\frac{n'}{np}}.
\]
\prettyref{thm:first_moment_bound} yields
\[
\expect{|\widehat{M}_{\mathsf{ML},V}\symdiff M^*_V| \;\Biggr\lvert\; V}\leq Cn'\max\left\{\log(1+\epsilon'),\sqrt{\frac{\log n'}{n'}}\right\},
\]
where $\epsilon'$ is such that
\[
2\log (1+\epsilon')= 2\log(1+\epsilon)+\log\frac{n'}{np}.
\]
Thus,
\begin{align}
\nonumber\expect{|\widehat{M}_{\mathsf{ML},V}\symdiff M^*_V| \;\Biggr\lvert\; V}\leq & C\max\left\{n'\log(1+\epsilon)+\frac{n'}{2}\log\frac{n'}{np}, \sqrt{n'\log n'}\right\}\\
\label{eq:ML.on.V}\leq & C\max\left\{n\log(1+\epsilon)+\frac{n'}{2}\log\frac{n'}{np}, \sqrt{n\log n}\right\},
\end{align}
where the last inequality is from $n'=|V|\leq n$. Next, we average over $V$ to obtain that
\begin{align*}
\expect{|\widehat{M}_{\mathsf{ML},V}\symdiff M^*_V|}
\leq & C\expect{\max\left\{n\log(1+\epsilon)+\frac{n'}{2}\log\frac{n'}{np}, \sqrt{n\log n}\right\}}\\
\leq & C\left[n\log(1+\epsilon)+\expect{\frac{n'}{2}\log\frac{n'}{np}}+\sqrt{n\log n}\right].
\end{align*}
Under our model, $W_{i,i'}\sim \calP$ for all $i$. Therefore $n'= |V|\sim \Binom(n,p)$.
To bound $\expect{n'\log n'}$, note that for any $u>0,$ and  $x \ge 0$, $\log (x/u) \le x/u -1$ so that $x \log x \le x^2/u + x \log (u/e)$.
Thus,
$$
\expect{ n' \log n' } \le \expect{(n')^2}/ u + \expect{n'} \log (u/e)
 \le \expect{n'} \log   \frac{\expect{(n')^2}}{\expect{n'}} 
\le np \log (n p + 1),
$$
where the first inequality holds by optimally choosing $u=\expect{(n')^2} /\expect{n'}.$
Hence, $\expect{ n'\log\frac{n'}{np} } \le np \log (1 + 1/(np) ) \le 1$
and consequently, 
$$
\expect{|\widehat{M}_{\mathsf{ML},V}\symdiff M^*_V} 
\le C \left[n\log(1+\epsilon)+ 1/2+ \sqrt{n\log n}\right]
\le C_1 \max\left\{n\log(1+\epsilon), \sqrt{n\log n}\right\}
$$
for some universal constant $C_1$. This finishes the proof of~\prettyref{thm:first_moment_bound} for general $\calP,\calQ$.

\subsection{Reduction for general weight distributions: negative results}
\label{app:acneg}
In \prettyref{app:reduction_dense_sparse} we have already reduced the impossibility result Theorem~\ref{thm:impossibility} from the dense model to the sparse model. In this section, we show that for the sparse model $(d,\calP,\calQ)$, we can assume WLOG that $\calP\ll \calQ$ and $\calQ\ll \calP$. In other words, assuming Theorem~\ref{thm:impossibility} holds under the sparse model with $\calP$ and $\calQ$ mutually absolutely continuous, we prove Theorem~\ref{thm:impossibility} for general (fixed) $\calP$, $\calQ$.

Recall that $\mu$ is a common dominating measure of $\calP,\calQ$, and $f,g$ are densities of $\calP,\calQ$ respectively. As in the reduction argument for the positive result, we first define distribution $\calP'$ with density $f\indc{g>0}/p$, and distribution $\calQ'$ with density $g\indc{f>0}/q$, where $p=\calP\{g>0\}$ and $q=\calQ\{f>0\}$ are the normalizing constants. Note that since $B(\calP,\calQ)>0$, both $p,q$ are strictly positive, so that $\calP'$ and $\calQ'$
are well-defined. 
Then we have $\calP'\ll \calQ'$, and $\calQ'\ll \calP'$.

We first reduce the original $(d,\calP,\calQ)$ model to the $(d', \calP,\calQ')$ model where $d'=dq$. Note that for edges $e$ that either do not appear in $G$ (in which case recall that we set $W_e =\star$ for a special symbol $\star$ to signify that $e$ is not in $G$), 
or are such that $f(W_e)=0$, we have $(f/g)(W_e)=0$.
Thus the posterior distribution of $M^*$ under $(d,\calP,\calQ)$ model is identical to the posterior distribution under the $(d', \calP,\calQ')$ model. 

Therefore, we only need to show that the conclusion of \prettyref{thm:impossibility} holds for graph $G$ with weights $(W_e)$ that follow the $(d', \calP,\calQ')$ model.
Let $V=\{i\in[n]:g(W_{i,i'})=0\}$. We can bound the optimal overlap as follows:
\begin{align}
\nonumber\sup_{\widehat{M}}\mathbb{E}\left(\sum_{i\leq n}\indc{(i,i') \in \hat{M}}\right) \leq & \sup_{\widehat{M}}\left[\mathbb{E}\left(V^c\right) +\mathbb{E}\left(\sum_{i\in V}\indc{(i,i') \in \hat{M}}\right)\right]\\
\label{eq:ac.reduction}\leq & np+\mathbb{E}\left[\sup_{\widehat{M}}\mathbb{E}\left(\sum_{i\in V}\indc{(i,i') \in \hat{M}}\;\bigg\rvert \;V\right)\right].
\end{align}
Conditional on $V$, the subgraph $\widetilde{G}=G[V\times V']$ follows the $(d'|V|'/n, \calP',\calQ')$ model. To see that, note that the red edges in $\widetilde{G}$ are distributed {\it i.i.d.} $\calP'$; the blue edges follow distribution $\calQ'$, and appear independently with probability $d'/n = (d'|V|/n)/|V|$. To apply the impossibility result under the $(d'|V|'/n, \calP',\calQ')$ model, we need to establish a lower bound on $|V|$. Let
\[
\calA=\left\{|V|\geq \left(1-\frac{\epsilon}{2}\right)^2pn\right\}.
\]
Recall that under the sparse model, $\calP,\calQ$ are fixed distributions that do not depend on $n$. Therefore $p$ does not depend on $n$, and since $|V|\sim \Binom (n,p)$, we have $\mathbb{P}(\calA)=1-o(1)$.

Furthermore, note that
\[
B(\calP',\calQ')=\frac{\int \sqrt{ fg } \diff\mu }{\sqrt{pq}} = \frac{B(\calP,\calQ)}{\sqrt{pq}}.
\]
Therefore on $\calA$, we have
\[
\sqrt{\frac{d'|V|}{n} } B \left( \calP', \calQ' \right)
=\sqrt{\frac{d|V|}{np} }  B \left( \calP, \calQ \right)
\geq (1+\epsilon)(1-\epsilon/2)\geq 1+\epsilon/3
\]
for $\epsilon \le 1/3$. Therefore, we can invoke the conclusion of Theorem~\ref{thm:impossibility} under the $(d'|V|/n, \calP',\calQ')$ model to deduce that on $\calA$, for some constant $c>0$ and large enough $n$,
\[
\sup_{\widehat{M}}\mathbb{E}\left(\sum_{i\in V}\indc{(i,i') \in \hat{M}} \;\bigg\rvert \;V\right)\leq (1-c)|V|,
\]
where the supremum is over all (possibly random) mappings from $G[V\times V']$ to $\calM_V$, the set of perfect matchings on $V\times V'$. Conditional on $V$, the edges in $G[V\times V']$ are independent of those in  $G\setminus G[V\times V']$. Thus, we can replace the range of the supremum with all (random) mappings from $G$ to $\calM_V$ without changing its value, allowing us to continue upper bounding~\eqref{eq:ac.reduction} with
\[
np + \mathbb{E}\left[(1-c)|V|\mathbf{1}_\calA\right] + n\mathbb{P}\left(\calA^c\right)
\leq \left(p+ (1-c)(1-p)+o(1)\right)n
= \left(1-c(1-p)+o(1)\right)n.
\]
In other words, the fraction of misclassified edges is lower bounded by $c(1-p)+o(1)$. Under the sparse model, $c(1-p)$ is a constant that only depends on $d,\calP,\calQ$. We have proved Theorem~\ref{thm:impossibility} under the $(d',\calP,\calQ')$ model. The conclusion of \prettyref{thm:impossibility} under the $(d,\calP,\calQ)$ model immediately follows.

\section{Finite-order phase transition under the unweighted model}\label{app:unweighted}
In this section, we prove that in the unweighted case $\calP=\calQ$, under the impossibility condition~\eqref{eq:impossibility}, the minimal reconstruction error is of order $\Omega(\epsilon^8)$. The impossibility condition translates to $\sqrt{d}\geq 1+\epsilon$. Similar to the proof of \prettyref{thm:impossibility}, we can assume WLOG that $\sqrt{d}=1+\epsilon$. As in the proof under general weight distributions, this result is again proven via the two-stage cycle finding scheme described in \prettyref{alg:cycle_finding}. Recall from \prettyref{alg:cycle_finding} that we reserve a set $V$ of $\gamma n$ left vertices from $[n]$. We will first construct many disjoint alternating paths on the subgraph $G_1=G[V^c\times (V^c)']$, and then use the reserved vertices to connect the paths into alternating cycles.

\subsection{Path construction}
The sets $L_k$ and $R_k$ and the alternating paths will again be constructed using two-sided trees. However, unlike the model with general weight distribution, we do not need to keep track of the weights of the paths, hence the construction is much simpler. For example, the leaf node selection step is no longer necessary. Moreover, in the previous sections we needed the paths to be long enough, so that the weights on the paths dominate the weights of the sprinkling edges. In the unweighted case, that restriction is also lifted. Therefore we can define $L_k$ (resp.~$R_k$) to be all the left (resp.~right) vertices in the left (resp.~right) subtree, instead of only the selected leaf nodes. The detailed construction is given in \prettyref{alg:exploration.unweighted} below.

\begin{algorithm}[H]
\caption{Construction of two-sided trees (under the unweighted model)} \label{alg:exploration.unweighted}
\begin{algorithmic}[1]
\State{\bfseries Input:} $n,\gamma$, a bipartite graph $G_1$ that contains $(1-\gamma)n$ pairs of vertices, and parameter $\ell$.
\State Initialize $\mathcal{U}=\{\text{all left vertices of }G_1\}$ as the set of unexplored left vertices.
\State For $k=1,2,...$, repeat the following steps 4-7 to construct the two-sided tree $T_k$, until $|\mathcal{U}|=(1-2\gamma) n$.
\State Let $i_k$ be the member of $\calU$ with the smallest index. Update $\mathcal{U}\leftarrow\mathcal{U}\backslash \{i_k\}$. Initialize $T_k=\{(i_k,i_k')\}$ to be a tree containing only one red edge.
\item Construct the left tree of $T_k$ via a color-alternating breadth-first search on $G_1$.
Concretely, define the offsprings of $i_k$ as
\[
O_{i_k}'=\{u'\in (\mathcal{U})': (u', i_k)\in E(G_1)\}.
\]
For all $u'\in O_{i_k}'$, append edges $(u',v)$ and $(u,u')$ to $T_k$. Update $\calU\leftarrow\calU\backslash O_{i_k}'$. To construct the next 2 layers of the left tree, sequentially (ordering defined by the vertex indices) for all $u\in O_{i_k}'$, define its offsprings as the set of unexplored vertices that are connected to $u$ via a blue edge in $G_1$; append to $T_k$ all the blue edges from $u$ to the offsprings and their corresponding red edges; and mark all the offsprings as explored. Repeat this process until the branching process dies, or the left tree contains $\ell$ vertices in total.
\item Construct the right tree of $T_k$ via the same scheme, starting from the vertex $i_k'$.
\item Let $L_k$ be the set of all left vertices in the left tree of $T_k$; let $R_k$ be the set of all right vertices in the right tree of $T_k$.
\State Define $\mathcal{K}_1=\{k: |L_k|=|R_k|=2\ell\}$.
\end{algorithmic}
\end{algorithm}

Note that \prettyref{alg:exploration.unweighted} explores $\gamma n$ pairs of vertices in total. Next, we show that with high probability, the size $K_1$ of the set $\calK_1$ is at least $c_3n/\ell$ for some constant $c_3=\Omega(\epsilon^3)$. That is, a constant proportion of two-sided trees contain exactly $2\ell$ pairs of vertices in both the left and the right tree.

By construction, for each $k$, $|L_k|$ can only be strictly smaller than $2\ell$ if the breadth-first search cannot find more vertices to explore, namely, the branching process dies. Since the number $|\calU|$ of unused vertices to explore is at least $(1-2\gamma)n$, we have
\begin{align*}
\mathbb{P}\left\{|L_k|<2\ell\right\}\leq & \text{ Probability of extinction for a branching process}\\
&\text{with offspring distribution }\Binom((1-2\gamma)n, d/n)=: 1-c_4.
\end{align*}
By choosing $\gamma=\epsilon/2$, the mean of the offspring distribution is $(1-2\gamma)d =(1-\epsilon)(1+\epsilon)^2\geq 1+\epsilon/2$ for $\epsilon$ smaller than some universal constant $\epsilon_0$. Therefore,
according to the standard Branching process theory (see e.g.~\cite[Theorem 23.1]{frieze2016introduction}), $1-c_4$
is the unique solution $\rho<1$ so that $\phi(\rho)=\rho$, where $\phi(\rho)=\expect{\rho^X}$ with 
$X \sim \Binom((1-2\gamma)n, d/n)$. 
In particular, the probability of survival $c_4$ is strictly positive, and it can be further shown that $c_4\geq c_5 \epsilon$ for some universal constant $c_5$.

Note that the argument would be simplified if the events $\{|L_k|=2\ell\}$ and $\{|R_k|=2\ell\}$ were independent. This, however, is not true since when constructing $R_k$, the number of unexplored vertices depends on $|L_k|$. To resolve this technicality, note that there are always at least $(1-2\gamma)n$ unused vertices, and that the construction never reuses vertices. Therefore we can couple the construction of $L_k$, $R_k$ with two independent branching processes, each with offspring distribution $\Binom((1-2\gamma)n, d/n)$, such that
\[
\mathbb{P}\left\{|L_k|=|R_k|=2\ell\right\} \geq \mathbb{P}\left\{\text{both branching processes survive}\right\} = c_4^2.
\]
Since the algorithm uses $\gamma n$ pairs of vertices in total, and at most $2\ell$ pairs are used for each $k$, we have $K\triangleq\text{the total number of two-sided trees}\geq \gamma n/(2\ell)$. Thus
\begin{align*}
\mathbb{P}\left\{\sum_{k\leq K}\indc{|L_k|=|R_k|=2\ell}<\frac{c_4^2 \gamma n}{4\ell}\right\}
\leq & \mathbb{P}\left\{\Binom(K, c_4^2)<\frac{c_4^2 \gamma n}{4\ell}\right\}\\
\leq & \exp\left(-\frac{c_4^4\gamma}{4\ell}n\right)
\end{align*}
by Hoeffding's inequality. We have shown that with probability $1-\exp(-\Omega(n))$, $K_1\geq c_3n/\ell$, with constant $c_3\triangleq c_4^2 \gamma/4 \geq c_5^2 \epsilon^3/8$. 

\subsection{Sprinkling stage}
In this subsection, we apply \prettyref{alg:sprinkling} and 
Theorem~\ref{thm:sprinkling} to show the existence of exponentially alternating cycles in $G$. 
Since there are no weights, we set the thresholds in \prettyref{alg:sprinkling} as $\tau_\mathsf{red}=\tau_\mathsf{blue}=0$.
Since $\calP=\calQ$, these thresholds yield $V^*=V$, and $G_2$ contains all edges in $G$ except those in $V^c\times (V^c)'$. The parameters in Theorem~\ref{thm:sprinkling} are specified as follows. Since $|V|=\gamma n$, we have $\beta=\gamma=\epsilon/2$. The blue edge probability in $G_2$ is $d/n$, so that $\eta=d$. In the path construction stage, we showed that the set $\calK_1$ is of size at least $c_5^2\epsilon^3 n/(8 \ell)$. By taking its subset, we can assume WLOG that $K_1=|\calK_1|=c_5^2\epsilon^3 n/(8\ell)=c_6 \epsilon^3 n/\ell$ for some universal constant $c_6$. 

Let $\ell=c_7/\epsilon^5$ for some universal constant $c_7$ that will be chosen later. Next, we check that the assumptions of Theorem~\ref{thm:sprinkling}. Note that $s=2\ell$ and 
\[
b=\frac{\beta s \eta}{4}=\frac{c_7 (1+\epsilon)^2}{4\epsilon^4}\geq 4
\]
for all $\epsilon<\epsilon_0$ for small enough $\epsilon_0$;
\[
K_1=\frac{c_6 \epsilon^3 n}{ \ell}\geq 8400
\]
for large enough $n$;
\[
\kappa=\frac{2K_1 s\eta}{n}=4c_6\epsilon^3 (1+\epsilon)^2\leq \frac{1}{16^2}
\]
for small enough $\epsilon_0$;
\[
d_{\rm super}=\frac{K_1 b^2\eta}{32n}=\frac{c_6c_7 (1+\epsilon)^6}{512} \geq 256\log(32 e)
\]
for $c_7$ chosen large enough. We have checked that all the assumptions of Theorem~\ref{thm:sprinkling} are satisfied. Therefore, with high probability, \prettyref{alg:cycle_finding} yields at least $\exp(K_2/20)=\exp(c_6\epsilon^8 n/(320 c_7))$ alternating cycles of length at least $3K_2/4=3c_6\epsilon^8 n/(64 c_7)$.
Each alternating cycle corresponds to a perfect matching in $G$. By taking $\delta= 3K_2/8=3c_6\epsilon^8 n/(128 c_7)$,
all these perfect matchings are in $\Mbad$.
Note that under the unweighted model, all the perfect matchings that appear in $G$ occupy the same posterior mass. 
Thus we have shown that with high probability, $\mu_W(\Mbad)/\mu_W(M^*)\geq \exp(c_6\epsilon^8 n/(320 c_7))$. 
By Lemma~\ref{lmm:Mgood}, $\mu_W(\Mgood)/\mu_W(M^*) \leq 2 e^{7\epsilon \delta n}$ with high probability.
Thus we conclude that
$\mathbb{E}[\ell(\widetilde{M},M^*)]\gtrsim \delta=\Omega(\epsilon^8)$.



\section*{Acknowledgment}
J.~Xu would like to thank Cristopher Moore for many inspiring discussions on the posterior sampling.
J.~Xu is also grateful to Guilhem Semerjian, Gabriele Sicuro, and Lenka Zdeborov\'a for sharing the early draft
of~\cite{Semerjian2020}.

J.~Ding is supported by the NSF Grants DMS-1757479 and DMS-1953848.
Y.~Wu is supported in part by the NSF Grant CCF-1900507, an NSF CAREER award CCF-1651588, and an Alfred Sloan fellowship.
J.~Xu is supported by the NSF Grants IIS-1838124, CCF-1850743, and CCF-1856424. 
D.~Yang is supported by the NSF Grants CCF-1850743 and IIS-1838124. 

\bibliographystyle{plain}
\bibliography{optimal_overlap,bibliography}

\end{document}